\documentclass[11pt]{amsart}

\usepackage[margin=1in]{geometry}
\usepackage{comment}
\usepackage{hyperref}
\hypersetup{
	colorlinks=true,
	linkcolor=blue,
	citecolor=blue}

\usepackage{amsthm}
\usepackage{amsmath}
\usepackage{graphicx}
\usepackage{amsfonts,amssymb,tikz,mathrsfs}
\usepackage{amsmath}
\usepackage{caption}
\usepackage{stmaryrd}

\usepackage{tikz}
\usepackage{dsfont}
\usepackage{marginnote}
\usetikzlibrary{matrix, fit}
\usepackage{amssymb, shuffle}
\usepackage{amssymb}
\usepackage{dirtytalk}
\usepackage{imakeidx}
\makeindex
\usepackage{enumitem}
\usepackage{bbm}

\newtheorem{Theorem}{Theorem}[section]
\newtheorem{Corollary}[Theorem]{Corollary}

\newtheorem{Lemma}[Theorem]{Lemma}
\newtheorem{Example}[Theorem]{Example}
\newtheorem{Proposition}[Theorem]{Proposition}
\newtheorem{Remark}[Theorem]{Remark}
\newtheorem{Conjecture}[Theorem]{Conjecture}
\newtheorem{Definition}[Theorem]{Definition}

\newtheorem*{remark}{Remark}
\numberwithin{equation}{section}

\newcommand{\NSym}{\mathrm{NSym}}
\newcommand{\NCQSym}{\mathrm{NCQSym}}
\newcommand{\NCSym}{\mathrm{NCSym}}

\newcommand{\QSym}{\mathrm{QSym}}
\newcommand{\Sym}{\mathrm{Sym}}
\newcommand{\Peak}{\mathrm{Peak}}
\newcommand{\SetComp}{\mathrm{SetComp}}
\newcommand{\OddSetComp}{\mathrm{SetCompOdd}}
\newcommand{\peak}{\mathrm{Peak}}
\newcommand{\odd}{\mathrm{Odd}}
\newcommand{\oddd}{\mathrm{ODiff}}
\newcommand{\std}{\mathrm{Std}}

\newcommand{\Cat}{\mathrm{Cat}}
\newcommand{\NCPeak}{\mathrm{NC}\Pi}
\newcommand{\NCGamma}{\mathrm{NC}\Omega}
\newcommand{\Des}{\mathrm{Des}}

\newcommand{\bM}{\mathbf{M}}
\newcommand{\bF}{\mathbf{F}}

\newcommand{\bx}{{\bf x}}

\newcommand{\bbQ}{{\mathbb{Q}}}
\newcommand{\bfF}{{\mathbf{F}}}

\newcommand{\scrX}{{\mathscr{X}}}
\newcommand{\scrY}{{\mathscr{Y}}}
\newcommand{\scrE}{{\mathscr{E}}}
\newcommand{\scrF}{{\mathscr{F}}}
\newcommand{\bfone}{{\mathbf{1}}}
\newcommand{\bbmone}{{\mathbbm{1}}}
\newcommand{\bftwo}{{\mathbf{2}}}
\newcommand{\bfthree}{{\mathbf{3}}}
\newcommand{\bffour}{{\mathbf{4}}}

\newcommand{\bfdet}{{\mathbf{det}}}
\newcommand{\frakS}{{\mathfrak{S}}}

\newcommand{\bfh}{{\mathbf{h}}}
\newcommand{\bfe}{{\mathbf{e}}}
\newcommand{\bfp}{{\mathbf{p}}}
\newcommand{\bfm}{{\mathbf{m}}}
\newcommand{\bfs}{{\mathbf{s}}}

\newcommand{\bfM}{{\mathbf{M}}}
\newcommand{\bfK}{{\mathbf{K}}}

\newcommand{\bfG}{{\mathbf{G}}}

\newcommand{\calB}{{\mathcal{B}}}
\newcommand{\boldeta}{{\boldsymbol{\eta}}}
\newcommand{\sort}{{\mathrm{sort}}}

\newcommand{\cf}{{\mathrm{cf}}}
\newcommand{\sgn}{{\mathrm{sgn}}}
\newcommand{\reg}{{\mathrm{reg}}}

\title{The peak algebra in noncommuting variables}

\author{Farid Aliniaeifard}
\address{
	Research Center for Mathematics and Interdisciplinary Sciences, Shandong University \\ Frontiers Science Center for Nonlinear Expectations, Ministry of Education, Qingdao, Shandong, 266237, P. R. China}
\email{farid@sdu.edu.cn}

\author{Shu Xiao Li}

\address{
	Research Center for Mathematics and Interdisciplinary Sciences, Shandong University \\ Frontiers Science Center for Nonlinear Expectations, Ministry of Education, Qingdao, Shandong, 266237, P. R. China}
\email{lishuxiao@sdu.edu.cn}
\thanks{Both author were supported in part by the Provincial Nature Science Foundation of Shandong, Project No. ZR2024QA026 and the Fundamental Research Funds for the Central Universities.}

\subjclass[2010]{05E05, 05E10, 16T30, 20C30}
\keywords{quasisymmetric functions, peak algebra, descent-to-peak map}
\date{}
 
\begin{document}

\maketitle

\begin{abstract}
The well-known descent-to-peak map $\Theta_{\QSym}$ for the Hopf algebra of quasisymmetric functions, $\QSym$, and the peak algebra $\Pi$ were originally defined by Stembridge in 1997. We introduce their noncommutative analogues,  the labelled descent-to-peak map $\Theta_{{\NCQSym}}$ for the Hopf algebra of quasisymmetric functions in noncommuting variables, $\NCQSym$, and the peak algebra in noncommuting variables $\NCPeak$. Then, we define the Hopf algebra of  Schur $Q$-functions in noncommuting variables. We show that our generalizations possess many properties analogous to their classical counterparts. Furthermore, we show that the coefficients in the expansion of certain elements of $\NCPeak$ in the monomial basis of $\NCQSym$ satisfy the generalized Dehn-Sommerville equation of Bayer and Billera. In the end, we give representation-theoretic interpretations of the descent-to-peak map for the Hopf algebras of symmetric functions and noncommutative symmetric functions. 
\end{abstract}

\tableofcontents

\section{Introduction}

The first comprehensive study on the combinatorics of peaks was conducted by Stembridge \cite{S97}, who developed and introduced enriched $(P,\gamma)$-partitions, which is an analouge to Stanley’s theory of $(P,\gamma)$-partitions, with the key distinction that the notion of peaks replaces the notion of descents in the context of linear extensions of posets.
Generating functions of $(P,\gamma)$-partitions are the fundamental functions, $F_{(P,\gamma)}$, which give the Hopf algebra of quasisymmetric functions $\QSym$, and the generating functions of enriched $(P,\gamma)$-partitions are the enriched fundamental functions, $K_{(P,\gamma)}$, which span the peak algebra $\Pi$. Stembridge also defined the descent-to-peak algebra morphism $\Theta_\QSym$ from $\QSym$ to $\Pi$ where  $F_{(P,\gamma)}$ maps to $K_{(P,\gamma)}$, and showed that the dimension of the homogenous functions of degree $n$ of $\Pi$, $\Pi_n$, is equal to the number of odd compositions, compositions whose all parts are odd, which is equal to Fibonacci number $f_n$. 
Moreover, he showed that restricting the map $\Theta_\QSym$ to symmetric functions gives the Hopf algebra of Schur $Q$-functions. The Hopf algebra of Schur $Q$-functions, whose bases are indexed by odd partitions, are introduced in \cite{S11} to study the projective representations of symmetric and alternating groups. Combinatorially, the Schur $Q$-functions are equipped with a theory of shifted tableaux, including RSK correspondence, Littlewood-Richardson rule, and jeu de taquin \cite{S87, S89, W84}. In \cite{BMSvW}, Bergeron et al. showed that the peak algebra is a Hopf algebra and also the map $\Theta_\QSym$ is a Hopf algebra morphism. Also, the main result of \cite{S05} by Schocker is that the peak algebra is a left co-ideal of $\QSym$ under internal coproduct. \\

One intriguing application of peak algebra is its role in providing a combinatorial formula for the Kazhdan–Lusztig polynomials, as discussed.
In \cite{B14}, Brenti and Caselli gave
new characterization of the peak subalgebra of the
algebra of quasisymmetric functions and used this characterization to construct
a new basis for the peak algebra and a combinatorial formula for the Kazhdan–Lusztig polynomials. 
\\
 
  It is also shown that the peak algebra corresponds to the representations of the $0$-Hecke-Clifford algebra \cite{B04}. 
Further studies revealed connections between peaks and a variety of seemingly unrelated topics, such as the generalized Dehn-Sommerville equation \cite{A06, BMSvW0, B03} and the Schubert calculus of isotropic flag manifolds \cite{BMSvW0, BH}. Notably, in \cite{Nym01, N03}, the peak algebra is generalized to the Poirier-Reutenauer Hopf algebra of standard Young tableaux, which is introduced in \cite{P95}. Other generalizations can be found in \cite{A04, B06, H06}. \\

We provide an elegant and natural definition of the peak algebra in noncommuting variables and extend the theta map. We also generalize many known results related to the peak algebra to noncommuting variables. The proofs we provide here differ from the commutative cases; as a result of our constructions, we give alternative and often shorter proofs.  Our construction also establishes the definition of Schur $Q$-functions in noncommutative variables. We show that the relation between $\NCPeak$ and generalized Dehn-Sommerville equation is the same as that of the commutative case. 
 We also mention two surprising and interesting applications of the internal coproduct of $\NCPeak$ given by \cite{ZLZ} in the enumeration of Genocchi numbers and $q$-tangent numbers, in addition to giving representation-theoretic interpretations of theta maps of symmetric functions and noncommutative symmetric functions. \\

 In Section \ref{sec:painting} we provide a summary of the results using generalized chromatic functions along with a compelling combinatorial description of the findings. This section also offers a clear presentation of the results in this paper without proofs, except for the representation-theoretic interpretations. 
Section \ref{sec:per} gives the relevant background.
Section \ref{sec:sten} describes set compositions and odd set compositions of $[n]=\{1,2,\dots, n\}$ using pairs of subsets of $[n-1]$ and permutations of $[n]$. 
In Section \ref{Sec:fundamental} we extend the notion of the fundamental functions to noncommuting variables and in Proposition \ref{prop:fund1} prove an analogue of the fundamental theorem of $(P,\gamma)$-partitions. 
In Section \ref{sec:peak-en} we extend the notion of the enriched fundamental functions to noncommuting variables and in Proposition \ref{prop:fund2} prove an analog of the fundamental theorem of enriched $(P,\gamma)$-partitions. 
In Section \ref{sec:peak-n}, we introduce enriched monomial functions. In Thereom \ref{Thm:NCPeakBasis}, we show that enriched monomial functions give a basis for the peak algebra in noncommuting variables, $\NCPeak$, and the dimension of the homogenous functions of degree $n$ of $\NCPeak$, $\NCPeak_n$, is equal to the number of odd set compositions, set compositions whose all parts are odd, which is equal to $a_n$ where $a_n$ is the sequence A006154 in the OEIS. 
In Section \ref{Sec:NCPeak}, we show that $\NCPeak$ is a Hopf algebra and describe its product and coproduct in the language of generalized chromatic functions.
In Section \ref{sec:thetamap-n}, we extend the notion of descent-to-peak map and define the labelled descent-to-peak map $\Theta_{\NCQSym}$, and give short proofs for the main results in \cite{BMSvW}. In Proposition \ref{prop:M}, we show that $\Theta_{\NCQSym}$ maps the monomial basis of $\NCQSym$ to a constant multiple of the enriched monomial basis of $\NCPeak$. In Theorem \ref{thm:theta}, we find the value of $\Theta_{\NCQSym}$ at fundamental functions and that applying the map $\rho$, where it commutes the variables, we obtain the descent-to-peak map $\Theta_{\QSym}$.
  In Section \ref{sec:D-S},  we fix a permutation $\sigma_n$ in each degree and we show that the coefficients of the expansion of certain elements in $\NCPeak$ in terms of monomial elements $\bfM_{(A,\sigma_n)}$ of $\NCQSym$ give a function that satisfies the generalized Dehn-Sommerville equation. 
In Section \ref{Sec:NCSym}, we define the Hopf algebra of Schur $Q$-functions in noncommuting variables $\NCGamma$ and show that it is equal to the restriction of the map $\Theta_{\NCQSym}$ to symmetric functions in noncommuting variables. The Schur $Q$-functions that are a basis for $\NCGamma$ are indexed by odd set partitions.  
In Section \ref{Sec:combid} and \ref{Sec:com}, we present several combinatorial properties of set compositions that we use them in the last section to show that the peak algebra in noncommuting variables is a left co-ideal of $\NCQSym$ under internal coproduct, extending    Schocker's result in \cite{S11} that the peak algebra is a left co-ideal of $\QSym$ under internal coproduct. In the last section, we present a functor that yields the theta map for the Hopf algebra of symmetric functions. We also provide a representation-theoretic interpretation of the theta map for noncommutative symmetric functions, based on a structure originating from supercharacter theory in Section \ref{Sec:rep}.

\section{A painted synopsis}\label{sec:painting}

We summarize the combinatorial results of this paper using generalized chromatic functions \cite{ALvW23}. For the representation theoretic results see the last section. 

\subsection{Edge-coloured digraphs and some operators}
Stanley \cite{Sta71} defined $(P,\gamma)$-partitions by generalizing MacMahon's work on plane partitions \cite{Mac11}\footnote{For a complete history of $P$-partitions see I. M. Gessel's paper \cite{Ges16}.}. $(P,\gamma)$-partitions can be identified as certain vertex-colourings of some family of edge-coloured digraphs. We describe this family of edge-coloured digraphs and some useful operators between them. 

The Hasse diagram of a poset $P=(X,\leq)$ can be seen as a digraph whose vertices are the elements of the ground set $X$ of the poset, and there is a directed edge from $a$ to $b$ if $a\leq b$ and if there is $c\in X$ such that $a\leq c\leq b$, then either $c=a$ or $c=b$. {\bf \emph{ Throughout this paper, all digraphs are Hasse diagrams of some posets.}} For example, the following digraph corresponds to a poset with the ground set $\{1,2,3,4\}$.
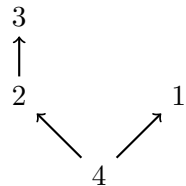
$$
 \begin{tikzpicture}[scale=0.7]
 \node(a) at (0,0){$4$};
 \node(b) at (-1.5,1.5){$2$};
 \node(c) at (1.5,1.5){$1$};
 \node(d) at (-1.5,3){$3$};

 \draw[thick,->] (a)--(b);
 \draw[thick,->] (a)--(c);
 \draw[thick,->] (b)--(d);
 \end{tikzpicture} 
 $$

 An \emph{edge-coloured digraph} is a digraph whose edges are of the form $\rightarrow$ or $\Rightarrow$.
\begin{center}
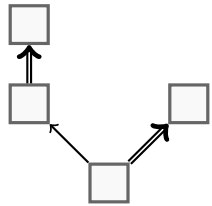
 
  \begin{tikzpicture}[scale=0.7,
roundnode/.style={circle, draw=black!60, fill=gray!10, very thick, minimum size=5mm},
squarednode/.style={rectangle, draw=black!60, fill=gray!5, very thick, minimum size=5mm},]

 \node[squarednode](a) at (0,0){};
 \node[squarednode](b) at (-1.5,1.5){};
 \node[squarednode](c) at (1.5,1.5){};
 \node[squarednode](d) at (-1.5,3){};
 
 \draw[thick,->] (a)--(b);
 \draw[thick,double,->] (a)--(c);
 \draw[thick,double,->] (b)--(d);
 \end{tikzpicture} 
   \captionof{figure}{An edge-coloured digraph}\label{Fig:edge-coloured}
 \end{center} 
 
 We now define some useful edge-coloured digraphs. Let
 $Q_n$ (resp. $P_n$) be the edge-coloured directed path with $n$ vertices whose all edges are of the form  $\Rightarrow$ (resp. $\rightarrow$).  
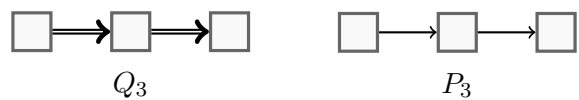
  $$
  \begin{tikzpicture}[
roundnode/.style={circle, draw=black!60, fill=gray!10, very thick, minimum size=7mm},
squarednode/.style={rectangle, draw=black!60, fill=gray!5, very thick, minimum size=5mm},]
 \node[squarednode](1) at (1.3,0){};
  \node[squarednode](2) at (2.6,0){};
    \node[squarednode](3) at (3.9,0){};
    \node at (2.6,-0.7){$Q_{3}$};
    \draw[thick,double,->](1)--(2);
       \draw[ thick,double,->](2)--(3);
 \end{tikzpicture}\quad \quad  \quad   \begin{tikzpicture}[
roundnode/.style={circle, draw=black!60, fill=gray!10, very thick, minimum size=7mm},
squarednode/.style={rectangle, draw=black!60, fill=gray!5, very thick, minimum size=5mm},]
 \node[squarednode](1) at (1.3,0){};
  \node[squarednode](2) at (2.6,0){};
    \node[squarednode](3) at (3.9,0){};
    \node at (2.6,-0.7){$P_{3}$};
    \draw[thick,->](1)--(2);
       \draw[ thick,->](2)--(3);
 \end{tikzpicture} 
 $$
 
 The \emph{disjoint union} of edge-coloured digraphs $G_1$ and $G_2$ with $V(G_1)\cap V(G_2)=\emptyset$, denoted $G_1\sqcup G_2$,  is an edge-coloured digraph such that 
 \begin{enumerate} 
 \item The vertex set of $G_1\sqcup G_2$ is  the disjoint union of the vertex sets of $G_1$ and $G_2$. 
 \item The edge set of $G_1\sqcup G_2$ is  the disjoint union of the edge sets of $G_1$ and $G_2$. 
 \item $a\Rightarrow b$ in $G_1\sqcup G_2$ if either $a\Rightarrow b$ in $G_1$ or in $G_2$. 
  \item $a\rightarrow b$ in $G_1\sqcup G_2$ if either $a\rightarrow b$ in $G_1$ or in $G_2$.
 \end{enumerate} 
   $$
  \begin{tikzpicture}[
roundnode/.style={circle, draw=black!60, fill=gray!10, very thick, minimum size=7mm},
squarednode/.style={rectangle, draw=black!60, fill=gray!5, very thick, minimum size=5mm},]
 \node[squarednode](1) at (1.3,0){};
  \node[squarednode](2) at (2.6,0){};
    \node[squarednode](3) at (3.9,0){};
    \draw[thick,double,->](1)--(2);
       \draw[ thick,double,->](2)--(3);
 \node[squarednode](11) at (4.8,0){};
  \node[squarednode](21) at (6.1,0){};
    \node[squarednode](31) at (7.4,0){};
    \draw[thick,->](11)--(21);
       \draw[ thick,->](21)--(31);
       \node at (4.35,-0.7){$Q_3 \sqcup P_3$};
 \end{tikzpicture} 
 $$

 The \emph{solid sum} of edge-coloured directed paths $G_1$ and $G_2$, denoted by $G_1\rightarrow G_2$,  is an edge-coloured digraph obtained by connecting the last vertex of $G_1$ to the first vertex of $G_2$ by a solid edge $\rightarrow$. 
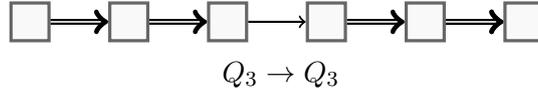
    $$
  \begin{tikzpicture}[
roundnode/.style={circle, draw=black!60, fill=gray!10, very thick, minimum size=7mm},
squarednode/.style={rectangle, draw=black!60, fill=gray!5, very thick, minimum size=5mm},]
 \node[squarednode](1) at (1.3,0){};
  \node[squarednode](2) at (2.6,0){};
    \node[squarednode](3) at (3.9,0){};
    \node at (4.6,-0.7){$Q_{3}\rightarrow  Q_{3}$};
    \draw[thick,double,->](1)--(2);
       \draw[ thick,double,->](2)--(3);
 \node[squarednode](a) at (5.2,0){${}$};
  \node[squarednode](b) at (6.5,0){};
    \node[squarednode](c) at (7.8,0){};
    \draw[thick,double,->](a)--(b);
       \draw[ thick,double,->](b)--(c);
          \draw[ thick,double,->](b)--(c);
            \draw[ thick,->](3)--(a);
 \end{tikzpicture} 
 $$

 \subsection{Proper colourings and generalized chromatic functions}\label{subsec:gcf}
 
 As we mentioned earlier, $(P,\gamma)$-partitions can be identified as certain types of vertex-colourings of edge-coloured digraphs. We describe these types of vertex-colourings of edge-coloured digraphs, and then we construct their generating functions, which are called generalized chromatic functions. 

A \emph{proper} colouring of an edge-coloured digraph $G$ is a function $$\kappa:V(G)\rightarrow \mathbb{N}=\{1,2,3,\dots\}$$ such that 
\begin{enumerate}
\item If $a\Rightarrow b$, then $\kappa(a)\leq \kappa(b)$.
\item If $a\rightarrow b$, then $\kappa(a)< \kappa(b)$.
\end{enumerate}

Given an edge-coloured digraph, we usually write the colours of the vertices outside of the vertices, which are non-bold positive integers. 
\begin{center}
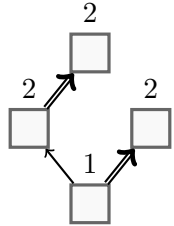
 
\begin{tikzpicture}[scale=0.4,
roundnode/.style={circle, draw=black!60, fill=gray!10, very thick, minimum size=5mm},
squarednode/.style={rectangle, draw=black!60, fill=gray!5, very thick, minimum size=5mm},]
                 \node[squarednode,label=above:{$2$}](d) at (0,5) {};
				\node[squarednode,label=above:{$2$}](b)  at (-2,2.5) {};
				\node[squarednode,label=above:{$1$}](a)  at (0,0) {};
				\node[squarednode, label=above:{$2$}](c)  at (2,2.5) {};
				\draw[black,double, thick,->] (a)--(c);
				\draw[black, thick,->] (a)--(b);
				\draw[black, double, thick,->] (b)--(d);
			\end{tikzpicture}
			 \captionof{figure}{A proper colouring of the edge-coloured digraph in Figure \ref{Fig:edge-coloured}}\label{Fig:proper-colouring-edge-coloured}
\end{center}

Recall that $\bbQ[[x_1,x_2,\dots]]$ is the algebra of formal power series in infinitely many commuting variables $x=\{x_1,x_2,\dots\}$ over $\bbQ$.

The \emph{generalized chromatic function}  of an edge-coloured digraph $G$ is 
$$\scrX_G=\sum_{\kappa} \prod_{v\in V(G)} x_{\kappa(v)}$$  where the sum is over all proper colourings $\kappa$ of $G$. For example, if $G$ is the edge-coloured digraph in Figure \ref{Fig:edge-coloured}, then 
\begin{align*}
		\begin{array}{ccccccccc}
			& \begin{tikzpicture}[scale=0.4,
roundnode/.style={circle, draw=black!60, fill=gray!10, very thick, minimum size=5mm},
squarednode/.style={rectangle, draw=black!60, fill=gray!5, very thick, minimum size=5mm},]
				\draw[black,double, thick,->] (a)--(c);
				\draw[black, thick,->] (a)--(b);
				\draw[black, double, thick,->] (b)--(d);
				\node[squarednode,label=above:{$2$}](d) at (0,5) {};
				\node[squarednode,label=above:{$2$}](b)  at (-2,2.5) {};
				\node[squarednode,label=above:{$1$}](a)  at (0,0) {};
				\node[squarednode, label=above:{$2$}](c)  at (2,2.5) {};
			\end{tikzpicture} &
			&\begin{tikzpicture}[scale=0.4,
roundnode/.style={circle, draw=black!60, fill=gray!10, very thick, minimum size=5mm},
squarednode/.style={rectangle, draw=black!60, fill=gray!5, very thick, minimum size=5mm},]
				\draw[black,double, thick,->] (a)--(c);
				\draw[black, thick,->] (a)--(b);
				\draw[black, double, thick,->] (b)--(d);
				\node[squarednode,label=above:{$3$}](d) at (0,5) {};
				\node[squarednode,label=above:{$2$}](b)  at (-2,2.5) {};
				\node[squarednode,label=above:{$1$}](a)  at (0,0) {};
				\node[squarednode, label=above:{$2$}](c)  at (2,2.5) {};
			\end{tikzpicture} &
			&  \begin{tikzpicture}[scale=0.4,
roundnode/.style={circle, draw=black!60, fill=gray!10, very thick, minimum size=5mm},
squarednode/.style={rectangle, draw=black!60, fill=gray!5, very thick, minimum size=5mm},]
				\draw[black,double, thick,->] (a)--(c);
				\draw[black, thick,->] (a)--(b);
				\draw[black, double, thick,->] (b)--(d);
				\node[squarednode,label=above:{$2$}](d) at (0,5) {};
				\node[squarednode,label=above:{$2$}](b)  at (-2,2.5) {};
				\node[squarednode,label=above:{$1$}](a)  at (0,0) {};
				\node[squarednode, label=above:{$3$}](c)  at (2,2.5) {};
			\end{tikzpicture} & 
			& \begin{tikzpicture}[scale=0.4,
roundnode/.style={circle, draw=black!60, fill=gray!10, very thick, minimum size=5mm},
squarednode/.style={rectangle, draw=black!60, fill=gray!5, very thick, minimum size=5mm},]
				\draw[black,double, thick,->] (a)--(c);
				\draw[black, thick,->] (a)--(b);
				\draw[black, double, thick,->] (b)--(d);
				\node[squarednode,label=above:{$3$}](d) at (0,5) {};
				\node[squarednode,label=above:{$2$}](b)  at (-2,2.5) {};
				\node[squarednode,label=above:{$1$}](a)  at (0,0) {};
				\node[squarednode, label=above:{$3$}](c)  at (2,2.5) {};
			\end{tikzpicture}\\
			\scrX_G=& x_1x_2^3 & + & x_1x_2^2x_3 & + & x_1x_2^2x_3 & + & x_1x_2x_3^2 &+\cdots.
		\end{array}
	\end{align*}
	
	Generalized chromatic functions can be seen as generating functions of $(P,\gamma)$-partitions, and so they are quasisymmetric functions (see Section \ref{subsec:(enr)Ppar}).


	\subsection{Enriched colourings and enriched chromatic functions}\label{subsec:ecf}
	Stembridge \cite{S97} defined enriched $(P, \gamma)$-partitions and used them to associate tableaux with Schur $Q$-functions \cite{S11}\footnote{For an English reference, see I. G. Macdonald's book \cite[Chapter III, Section 8]{M95}, where he described Schur's $Q$-functions in more detail. }. Enriched $(P,\gamma)$-partitions can be identified as certain types of vertex-colourings of edge-coloured digraphs. We describe these types of vertex-colourings of edge-coloured digraphs, and then we construct their generating functions, which are called enriched chromatic functions. 
	
	Given an edge-coloured digraph $G$, an \emph{enriched} colouring of $G$ is a function 
	$$\kappa: V(G)\rightarrow \{\dots \prec -1\prec1\prec -2\prec 2 \prec \cdots \}$$ such that 
	\begin{enumerate} 
	\item If $a\Rightarrow b$, then either $\kappa(a)\prec \kappa(b)$ or $\kappa(a)=\kappa(b)>0$. 
	\item If $a\rightarrow b$, then either $\kappa(a)\prec \kappa(b)$ or $\kappa(a)=\kappa(b)<0$.
	\end{enumerate}

The \emph{enriched chromatic function}  of an edge-coloured digraph $G$ is 
$$\scrE_G=\sum_{\kappa} \prod_{v\in V(G)} x_{|\kappa(v)|}$$  where the sum is over all enriched colourings $\kappa$ of $G$. For example, if $G$ is the edge-coloured digraph in Figure \ref{Fig:edge-coloured}, then 
\begin{align*}
		\begin{array}{ccccccccc}
			& \begin{tikzpicture}[scale=0.4,
roundnode/.style={circle, draw=black!60, fill=gray!10, very thick, minimum size=5mm},
squarednode/.style={rectangle, draw=black!60, fill=gray!5, very thick, minimum size=5mm},]
				\draw[black,double, thick,->] (a)--(c);
				\draw[black, thick,->] (a)--(b);
				\draw[black, double, thick,->] (b)--(d);
				\node[squarednode,label=above:{$1$}](d) at (0,5) {};
				\node[squarednode,label=above:{$-1$}](b)  at (-2,2.5) {};
				\node[squarednode,label=above:{$-1$}](a)  at (0,0) {};
				\node[squarednode, label=above:{$1$}](c)  at (2,2.5) {};
			\end{tikzpicture} &
			&\begin{tikzpicture}[scale=0.4,
roundnode/.style={circle, draw=black!60, fill=gray!10, very thick, minimum size=5mm},
squarednode/.style={rectangle, draw=black!60, fill=gray!5, very thick, minimum size=5mm},]
				\draw[black,double, thick,->] (a)--(c);
				\draw[black, thick,->] (a)--(b);
				\draw[black, double, thick,->] (b)--(d);
				\node[squarednode,label=above:{$1$}](d) at (0,5) {};
				\node[squarednode,label=above:{$1$}](b)  at (-2,2.5) {};
				\node[squarednode,label=above:{$-1$}](a)  at (0,0) {};
				\node[squarednode, label=above:{$1$}](c)  at (2,2.5) {};
			\end{tikzpicture} &
			&\begin{tikzpicture}[scale=0.4,
roundnode/.style={circle, draw=black!60, fill=gray!10, very thick, minimum size=5mm},
squarednode/.style={rectangle, draw=black!60, fill=gray!5, very thick, minimum size=5mm},]
				\draw[black,double, thick,->] (a)--(c);
				\draw[black, thick,->] (a)--(b);
				\draw[black, double, thick,->] (b)--(d);
				\node[squarednode,label=above:{$2$}](d) at (0,5) {};
				\node[squarednode,label=above:{$-2$}](b)  at (-2,2.5) {};
				\node[squarednode,label=above:{$1$}](a)  at (0,0) {};
				\node[squarednode, label=above:{$1$}](c)  at (2,2.5) {};
			\end{tikzpicture} & 
			& \begin{tikzpicture}[scale=0.4,
roundnode/.style={circle, draw=black!60, fill=gray!10, very thick, minimum size=5mm},
squarednode/.style={rectangle, draw=black!60, fill=gray!5, very thick, minimum size=5mm},]
				\draw[black,double, thick,->] (a)--(c);
				\draw[black, thick,->] (a)--(b);
				\draw[black, double, thick,->] (b)--(d);
				\node[squarednode,label=above:{$2$}](d) at (0,5) {};
				\node[squarednode,label=above:{$-2$}](b)  at (-2,2.5) {};
				\node[squarednode,label=above:{$1$}](a)  at (0,0) {};
				\node[squarednode, label=above:{$-2$}](c)  at (2,2.5) {};
			\end{tikzpicture}\\
			\scrE_G=& x_1^4 & + & x_1^4 & + & x_1^2x_2^2 & + & x_1x_2^3&+\cdots.
		\end{array}
	\end{align*}

\subsection{Peak algebra, descent-to-peak map,  and earlier results}\label{subsec:peak-theta-results}

Stembridge defined the peak algebra $\Pi$\footnote{See Section \ref{lem:ppar-to-gcf} for the justification that this space is called the peak algebra.} in \cite{S97} as a space spanned by the generating functions of enriched $(P,\gamma)$-partitions. Since generating functions of enriched $(P,\gamma)$-partitions are the enriched chromatic functions of edge-coloured digraphs and vice versa (Lemma  \ref{lem:ppar-to-gcf}), we have that the peak algebra is spanned by the set 
$$\{\scrE_G: \text{$G$ is an edge-coloured digraph}\}.$$ 

Stembridge also defined the \emph{descent-to-peak map}\footnote{See Section \ref{subsec:(enr)Ppar} for the justification that this map is called descent-to-peak map.} $\Theta_{\QSym}$; as shown in Corollary \ref{cor:theta-gcf1}, we can write it as follows, 
 \begin{align}\label{Def:ThetaMap}
\begin{array}{cccc}
\Theta_\QSym:& \QSym & \rightarrow & \Pi\\
& \scrX_G & \mapsto & \scrE_G.
\end{array} 
\end{align} 
He showed that the descent-to-peak map is a surjective algebra morphism \cite[Theorem 3.1]{S97}.

We now present a series of known results for peak algebra and the descent-to-peak map that we later extend to noncommuting variables. 

\begin{enumerate}[leftmargin=*]
\item   A \emph{composition} $\alpha$ of $n$, denoted $\alpha\vDash n$, is a list of positive integers whose sum is $n$. An  \emph{odd composition} is a composition whose all parts are odd. A set $B\subseteq \{2,3,\dots,n-1\}$ is called a \emph{peak set}\footnote{The \emph{peak set of a permutation} $\sigma\in \mathfrak{S}_n$ is the set $\peak(\sigma)=\{i\in \{2,3,\dots,n-1\}: \sigma(i-1)<\sigma(i) >\sigma(i+1)\}$. If $i\in \peak(\sigma)$, then $i-1,i+1\not\in \peak(\sigma)$. Thus each peak set is the peak set of a  permutation.} if $b\in B$ implies that $\{b-1,b+1\}\cap B=\emptyset$. By \cite[Theorem 3.1]{S97} the dimension of the space of the homogeneous elements of degree $n$ of $\Pi$, $\Pi_n$, is 
\begin{align*}
\dim(\Pi_n)&=|\{B \subseteq\{2,3,\dots,n-1\} \text{~is a peak set}\}|\\
&=|\{\alpha\vDash [n]: \text{$\alpha$ is an odd composition} \}|\\
&=f_n,
\end{align*}
where $f_n$ is the $n$th Fibonacci number.

\item  The peak algebra $\Pi$ is a Hopf algebra \cite[Theorem 2.2]{BMSvW}.

\item The descent-to-peak map, $\Theta_{\QSym}$, is a Hopf algebra morphism \cite[Section 2]{BMSvW}.

\item The value of the descent-to-peak map at the monomial and fundamental bases elements of $\QSym$ are as follows. 
\begin{enumerate} 
\item {\bf Fundamental basis.}  
Given a subset $A=\{a_1<a_2<\cdots<a_{k}\}$ of $[n-1]$, the \emph{fundamental} basis element $F_A$ of $\QSym$ is the generalized chromatic function of the following edge-coloured digraph,
$$Q_{a_1}\rightarrow Q_{a_2-a_1}\rightarrow \dots \rightarrow Q_{n-a_k}.$$
Therefore, by the expression of the descent-to-peak map $\Theta_\QSym$ in (\ref{Def:ThetaMap}) we have
$$\Theta_{\QSym}(F_A)=\Theta_{\QSym}(\scrX_{Q_{a_1}\rightarrow Q_{a_2-a_1}\rightarrow \dots \rightarrow Q_{n-a_k}})=\scrE_{Q_{a_1}\rightarrow Q_{a_2-a_1}\rightarrow \dots \rightarrow Q_{n-a_k}}.$$
\item {\bf Monomial basis.} Given a subset $A=\{a_1<a_2<\cdots<a_{k}\}$ of $[n-1]$, the \emph{monomial} basis element $M_A$ of $\QSym$ is 
$$M_A=\sum_{C\subseteq A}(-1)^{|A|-|C|} F_C.$$  
Given a peak set $B\subseteq \{2,3,\dots,n-1\}$, the \emph{monomial peak function} $\eta_B$ is 
$$\eta_B=(-1)^{|B|}\sum_{A\subseteq \odd(B)}  2^{|A|+1}M_A$$ where  
$$\odd(B)=[n-1]\setminus (B\cup (B-1)).$$ 
Let $A=\{a_1<a_2<\cdots<a_k\}\subseteq [n-1]$ where $n-a_k$ is odd. Then there is a unique peak set $B\subseteq \{2,3,\dots,n-1\}$ such that 
$$\odd(B)=\oddd(A)$$ where 
$\oddd(A)=\{a_i:1\leq i\leq k, a_i-a_{i-1} \text{ is odd}\}$ with $a_0=0$. 
By \cite[Equation (4.9)]{A06}, we have 
\begin{align*}
\Theta_{\QSym}(M_A)
&=\begin{cases}
(-1)^{n-1-|B|-|A|} \eta_B & \text{if $n-\max(A)$ is odd,}\\
0 & \text{otherwise.}
\end{cases} 
\end{align*}

\end{enumerate}

\item Bayer and Billera \cite{BB85} generalized the Dehn-Sommerville equation and showed that if a function satisfies this equation, then it is a linear combination of flag $f$-vectors of Eulerian posets. Later, in \cite{B03}, Billera, Hsiao, and van Willigenburg showed that a function $f$ satisfies the generalized Dehn-Sommerville equation if and only if $\sum f(A)M_A$ is in $\Pi$.

\item A \emph{partition} $\lambda$ of $n$, denoted $\lambda\vdash n$, is a non-increasing composition of $n$.   The values of the descent-to-peak map at different bases of $\Sym$ are as follows. 
\begin{enumerate} 
\item For $\lambda=\lambda_1\lambda_2\dots \lambda_{\ell(\lambda)}$, 
by \cite[Table 2]{ALvW23} we have that the \emph{complete homogeneous symmetric function}  $h_\lambda$ (resp., \emph{elementary symmetric function} $e_\lambda$) is the generalized chromatic function of
$$Q_{\lambda}=Q_{\lambda_1} \sqcup Q_{\lambda_2} \sqcup \cdots \sqcup  Q_{\lambda_{\ell(\lambda)}},~~~{\text{(resp.}~ P_{\lambda}=P_{\lambda_1} \sqcup P_{\lambda_2} \sqcup \cdots \sqcup  P_{\lambda_{\ell(\lambda)}})}.$$ 
Thus by the description of the descent-to-peak map in (\ref{Def:ThetaMap})
$$\Theta_{\Sym}(h_\lambda)=\scrE_{Q_{\lambda}}~~~\text{(resp. $\Theta_{\QSym}(e_\lambda)=\scrE_{P_{\lambda}}$)}.$$

\item  For $\lambda=\lambda_1\lambda_2\cdots \lambda_{\ell(\lambda)}\vdash n$, the \emph{Schur function} $s_\lambda$ is 
$$s_\lambda=\det \left( \scrX_{Q_{\lambda_i-i+j} } \right)_{1\leq i,j\leq \ell(\lambda)}.$$ By  the description of the descent-to-peak map in (\ref{Def:ThetaMap})
\begin{align*}
\Theta_{\Sym}(s_\lambda)&=\Theta_{\Sym}(\det \left( \scrX_{Q_{\lambda_i-i+j} } \right)_{1\leq i,j\leq \ell(\lambda)})\\
&=\det \left( \scrE_{Q_{\lambda_i-i+j} } \right)_{1\leq i,j\leq \ell(\lambda)}.
\end{align*}

\item  For $\lambda=\lambda_1 \lambda_2\dots \lambda_{\ell(\lambda)}\vdash n$, we have the \emph{power sum symmetric function} $p_\lambda$ is 
$$p_\lambda=\left(\sum_{i=1} x_{i}^{\lambda_1}\right)\left(\sum_{i=1} x_{i}^{\lambda_2}\right) \cdots \left(\sum_{i=1} x_{i}^{\lambda_{\ell(\lambda)}}\right).$$
By \cite[Chapter III, Example 10]{M95}
$$ \Theta_\Sym(p_\lambda)=\begin{cases}
2^{\ell(\lambda)} p_{\lambda} & \text{if $\lambda$ is odd,}\\
0 & \text{otherwise.}
\end{cases} $$

\end{enumerate}

\item For a partition $\lambda$, the \emph{Schur $Q$-function} is 
$$q_\lambda=\scrE_{Q_{\lambda}}=\scrE_{P_{\lambda}}.$$ The Hopf algebra of \emph{Schur $Q$-functions} introduced in \cite{S11} is denoted by $\Omega$ and 
$$\Omega=\mathbb{C}\text{-span}\{q_\lambda: \lambda \text{~is an odd partition}\}.$$
Moreover, 
$$\Theta_{\QSym}(\Sym)=\Omega.$$
 By \cite[Theorem 3.8]{S97}, $\Omega$ is the intersection of the Peak algebra $\Pi$ and the Hopf algebra of symmetric functions $\Sym$, that is, 
$$\Omega=\Pi\cap \Sym.$$


 The restriction of the descent-to-peak map $\Theta_{\QSym}$ to $\Sym$ is denoted by $\Theta_\Sym$. By \cite[Theorem 3.1]{S97} the following diagram commutes.
\begin{center}
\begin{tikzpicture}
	\node(Sym) at (4,0){$\mathsf{\Sym}$};
	\node(QSym) at (8,0){$\QSym$};
	\node(Gamma) at (4,-4){$\Omega$};
	\node(Pi) at (8,-4){$\Pi$};
	
	\draw[thick,->]  (Sym)->(QSym);
	\draw[thick,->]  (Gamma)->(Pi);
	\draw[thick,->]   (Sym)->(Gamma);
	\draw[thick,->]   (QSym)->(Pi);

	\node(iota) at (6,0.2){$\iota$};
	\node(iota1) at (6,-3.8){$\iota$};
	\node(ThetaSym) at (4-0.5,-2){$\Theta_{\Sym}$};
	\node(ThetaQSym) at (8.7,-2){$\Theta_\QSym$};
	
	\node(xg) at (7.5,-1){$\scrX_G$};
	\node(eg) at (7.5,-3){$\scrE_G$};
	
	\draw[thick,|->] (xg)--(eg);
\end{tikzpicture} 
\end{center}

\item   In \cite{S11} Schocker proves that the peak algebra is a left co-ideal of $\QSym$ under internal coproduct.

\end{enumerate}

\subsection{Generalized chromatic functions in noncommuting variables}

A \emph{labelled edge-coloured digraph} is an edge-coloured digraph where its vertex set is a subset of $\mathbb{N}$. We usually denote the vertices of a labelled edge-coloured digraph by bold positive integers. 
\begin{center}
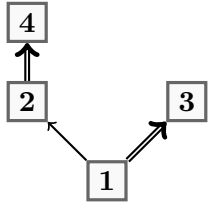
 
  \begin{tikzpicture}[scale=0.7,
roundnode/.style={circle, draw=black!60, fill=gray!10, very thick, minimum size=5mm},
squarednode/.style={rectangle, draw=black!60, fill=gray!5, very thick, minimum size=5mm},]

 \node[squarednode](a) at (0,0){$\bf 1$};
 \node[squarednode](b) at (-1.5,1.5){$\bf 2$};
 \node[squarednode](c) at (1.5,1.5){$\bf 3$};
 \node[squarednode](d) at (-1.5,3){$\bf 4$};
 
 \draw[thick,->] (a)--(b);
 \draw[thick,double,->] (a)--(c);
 \draw[thick,double,->] (b)--(d);
 \end{tikzpicture} 
   \captionof{figure}{A labelled edge-coloured digraph}\label{Fig:labelled-edge-coloured}
 \end{center} 
We usually use $\bfG$  to denote a labelled edge-coloured digraph whose underlying edge-coloured digraph is $G$. For example, if we denote the labelled edge-coloured digraph in Figure \ref{Fig:labelled-edge-coloured} by $\bfG$, then $G$ is the edge-coloured digraph in Figure \ref{Fig:edge-coloured}.

 We now define some useful labelled edge-coloured digraphs. For any set $$S=\{i_1<i_2<\dots<i_k\}$$ of positive integers, let 
 $Q_S$ (resp. $P_S$) be the labelled edge-coloured directed path with vertex set $S$ such that its coloured edges are  ${i_j}\Rightarrow {i_{j+1}}$ (resp. ${i_j}\rightarrow {i_{j+1}}$) for $1\leq j \leq k$.   
  $$
  \begin{tikzpicture}[
roundnode/.style={circle, draw=black!60, fill=gray!10, very thick, minimum size=7mm},
squarednode/.style={rectangle, draw=black!60, fill=gray!5, very thick, minimum size=5mm},]
 \node[squarednode](1) at (1.3,0){${\bf 3}$};
  \node[squarednode](2) at (2.6,0){{\bf 5}};
    \node[squarednode](3) at (3.9,0){{\bf 8}};
    \node at (2.6,-0.7){$Q_{\{3,5,8\}}$};
    \draw[thick,double,->](1)--(2);
       \draw[ thick,double,->](2)--(3);
 \end{tikzpicture}\quad \quad  \quad \quad \quad   \begin{tikzpicture}[
roundnode/.style={circle, draw=black!60, fill=gray!10, very thick, minimum size=7mm},
squarednode/.style={rectangle, draw=black!60, fill=gray!5, very thick, minimum size=5mm},]
 \node[squarednode](1) at (1.3,0){${\bf 3}$};
  \node[squarednode](2) at (2.6,0){{\bf 5}};
    \node[squarednode](3) at (3.9,0){{\bf 8}};
    \node at (2.6,-0.7){$P_{\{3,5,8\}}$};
    \draw[thick,->](1)--(2);
       \draw[ thick,->](2)--(3);
 \end{tikzpicture} 
 $$
 
 The \emph{disjoint union} of labelled edge-coloured digraphs $\bfG_1$ and $\bfG_2$ with $V(\bfG_1)\cap V(\bfG_2)=\emptyset$, is denoted by $\bfG_1\sqcup \bfG_2$. 
     $$
  \begin{tikzpicture}[
roundnode/.style={circle, draw=black!60, fill=gray!10, very thick, minimum size=7mm},
squarednode/.style={rectangle, draw=black!60, fill=gray!5, very thick, minimum size=5mm},]
 \node[squarednode](1) at (1.3,0){$\bf 3$};
  \node[squarednode](2) at (2.6,0){$\bf 5$};
    \node[squarednode](3) at (3.9,0){$\bf 8$};
    \draw[thick,double,->](1)--(2);
       \draw[ thick,double,->](2)--(3);
 \node[squarednode](11) at (4.8,0){$\bf 2$};
  \node[squarednode](21) at (6.1,0){$\bf 6$};
    \node[squarednode](31) at (7.4,0){$\bf 9$};
    \draw[thick,->](11)--(21);
       \draw[ thick,->](21)--(31);
       \node at (4.35,-0.7){$Q_{\{3,5,8\}} \sqcup P_{\{2,6,9\}}$};
 \end{tikzpicture} 
 $$

 The \emph{solid sum} of labelled edge-coloured directed paths $\bfG_1$ and $\bfG_2$, denoted $\bfG_1\rightarrow \bfG_2$, is an edge-coloured digraph obtained by connecting the last vertex of  $\bfG_1$ to the first vertex of $\bfG_2$ by a solid edge $\rightarrow$. 
    $$
  \begin{tikzpicture}[
roundnode/.style={circle, draw=black!60, fill=gray!10, very thick, minimum size=7mm},
squarednode/.style={rectangle, draw=black!60, fill=gray!5, very thick, minimum size=5mm},]
 \node[squarednode](1) at (1.3,0){${\bf 3}$};
  \node[squarednode](2) at (2.6,0){{\bf 5}};
    \node[squarednode](3) at (3.9,0){{\bf 8}};
    \node at (4.6,-0.7){$Q_{\{3,5,8\}}\rightarrow  Q_{\{2,6,9\}}$};
    \draw[thick,double,->](1)--(2);
       \draw[ thick,double,->](2)--(3);
 \node[squarednode](a) at (5.2,0){${\bf 2}$};
  \node[squarednode](b) at (6.5,0){{\bf 6}};
    \node[squarednode](c) at (7.8,0){{\bf 9}};
    \draw[thick,double,->](a)--(b);
       \draw[ thick,double,->](b)--(c);
          \draw[ thick,double,->](b)--(c);
            \draw[ thick,->](3)--(a);
 \end{tikzpicture} 
 $$

Recall that $\bbQ\langle \langle \bx_1,\bx_2,\dots\rangle\rangle$ is the algebra of formal power series in infinitely many noncommuting variables $\bx=\{\bx_1,\bx_2,\dots\}$ over $\bbQ$.

The \emph{generalized chromatic function in noncommuting variables} of a labelled edge-coloured digraph $\bfG$ with vertex set $[n]$ is 
$$\scrY_\bfG=\sum_{\kappa} \bx_{\kappa(1)} \bx_{\kappa(2)} \dots \bx_{\kappa(n)}$$   where the sum is over all proper colourings $\kappa$ of $G$. For example, if $\bfG$ is the labelled edge-coloured digraph in Figure \ref{Fig:labelled-edge-coloured}, then 
\begin{align*}
		\begin{array}{ccccccccc}
			& \begin{tikzpicture}[scale=0.4,
roundnode/.style={circle, draw=black!60, fill=gray!10, very thick, minimum size=5mm},
squarednode/.style={rectangle, draw=black!60, fill=gray!5, very thick, minimum size=5mm},]
				\draw[black,double, thick,->] (a)--(c);
				\draw[black, thick,->] (a)--(b);
				\draw[black, double, thick,->] (b)--(d);
				\node[squarednode,label=above:{$2$}](d) at (0,5) {$\bffour$};
				\node[squarednode,label=above:{$2$}](b)  at (-2,2.5) {$\bftwo$};
				\node[squarednode,label=above:{$1$}](a)  at (0,0) {$\bfone$};
				\node[squarednode, label=above:{$2$}](c)  at (2,2.5) {$\bfthree$};
			\end{tikzpicture} &
			&\begin{tikzpicture}[scale=0.4,
roundnode/.style={circle, draw=black!60, fill=gray!10, very thick, minimum size=5mm},
squarednode/.style={rectangle, draw=black!60, fill=gray!5, very thick, minimum size=5mm},]
				\draw[black,double, thick,->] (a)--(c);
				\draw[black, thick,->] (a)--(b);
				\draw[black, double, thick,->] (b)--(d);
				\node[squarednode,label=above:{$3$}](d) at (0,5) {$\bffour$};
				\node[squarednode,label=above:{$2$}](b)  at (-2,2.5) {$\bftwo$};
				\node[squarednode,label=above:{$1$}](a)  at (0,0) {$\bfone$};
				\node[squarednode, label=above:{$2$}](c)  at (2,2.5) {$\bfthree$};
			\end{tikzpicture} &
			&  \begin{tikzpicture}[scale=0.4,
roundnode/.style={circle, draw=black!60, fill=gray!10, very thick, minimum size=5mm},
squarednode/.style={rectangle, draw=black!60, fill=gray!5, very thick, minimum size=5mm},]
				\draw[black,double, thick,->] (a)--(c);
				\draw[black, thick,->] (a)--(b);
				\draw[black, double, thick,->] (b)--(d);
				\node[squarednode,label=above:{$2$}](d) at (0,5) {$\bffour$};
				\node[squarednode,label=above:{$2$}](b)  at (-2,2.5) {$\bftwo$};
				\node[squarednode,label=above:{$1$}](a)  at (0,0) {$\bfone$};
				\node[squarednode, label=above:{$3$}](c)  at (2,2.5) {$\bfthree$};
			\end{tikzpicture} & 
			& \begin{tikzpicture}[scale=0.4,
roundnode/.style={circle, draw=black!60, fill=gray!10, very thick, minimum size=5mm},
squarednode/.style={rectangle, draw=black!60, fill=gray!5, very thick, minimum size=5mm},]
				\draw[black,double, thick,->] (a)--(c);
				\draw[black, thick,->] (a)--(b);
				\draw[black, double, thick,->] (b)--(d);
				\node[squarednode,label=above:{$3$}](d) at (0,5) {$\bffour$};
				\node[squarednode,label=above:{$2$}](b)  at (-2,2.5) {$\bftwo$};
				\node[squarednode,label=above:{$1$}](a)  at (0,0) {$\bfone$};
				\node[squarednode, label=above:{$3$}](c)  at (2,2.5) {$\bfthree$};
			\end{tikzpicture}\\
			\scrY_G=& \bx_1\bx_2\bx_2\bx_2 & + & \bx_1\bx_2\bx_2\bx_3 & + & \bx_1\bx_2\bx_3\bx_2 & + & \bx_1\bx_2\bx_3\bx_3 &+\cdots.
		\end{array}
	\end{align*}
	
	\begin{Remark}
	The generalized chromatic functions in noncommuting variables first appeared in \cite{ALvW23}.
	\end{Remark}

\subsection{Enriched chromatic functions in noncommuting variables}\label{subsec:enriched-chromatic-functions} 
The \emph{enriched chromatic function in noncommuting variables} of a labelled edge-coloured digraph $\bfG$ with vertex set $[n]$ is 
$$\scrF_\bfG=\sum_{\kappa} \bx_{|\kappa(1)|} \bx_{|\kappa(2)|} \dots \bx_{|\kappa(n)|}$$  where the sum is over all enriched colouring $\kappa$ of $G$. For example, if $\bfG$ is the labelled edge-coloured digraph in Figure \ref{Fig:labelled-edge-coloured}, then 
\begin{align*}
		\begin{array}{ccccccccc}
			& \begin{tikzpicture}[scale=0.4,
roundnode/.style={circle, draw=black!60, fill=gray!10, very thick, minimum size=5mm},
squarednode/.style={rectangle, draw=black!60, fill=gray!5, very thick, minimum size=5mm},]
				\draw[black,double, thick,->] (a)--(c);
				\draw[black, thick,->] (a)--(b);
				\draw[black, double, thick,->] (b)--(d);
				\node[squarednode,label=above:{$1$}](d) at (0,5) {$\bffour$};
				\node[squarednode,label=above:{$-1$}](b)  at (-2,2.5) {$\bftwo$};
				\node[squarednode,label=above:{$-1$}](a)  at (0,0) {$\bfone$};
				\node[squarednode, label=above:{$1$}](c)  at (2,2.5) {$\bfthree$};
			\end{tikzpicture} &
			&\begin{tikzpicture}[scale=0.4,
roundnode/.style={circle, draw=black!60, fill=gray!10, very thick, minimum size=5mm},
squarednode/.style={rectangle, draw=black!60, fill=gray!5, very thick, minimum size=5mm},]
				\draw[black,double, thick,->] (a)--(c);
				\draw[black, thick,->] (a)--(b);
				\draw[black, double, thick,->] (b)--(d);
				\node[squarednode,label=above:{$1$}](d) at (0,5) {$\bffour$};
				\node[squarednode,label=above:{$1$}](b)  at (-2,2.5) {$\bftwo$};
				\node[squarednode,label=above:{$-1$}](a)  at (0,0) {$\bfone$};
				\node[squarednode, label=above:{$1$}](c)  at (2,2.5) {$\bfthree$};
			\end{tikzpicture} &
			&\begin{tikzpicture}[scale=0.4,
roundnode/.style={circle, draw=black!60, fill=gray!10, very thick, minimum size=5mm},
squarednode/.style={rectangle, draw=black!60, fill=gray!5, very thick, minimum size=5mm},]
				\draw[black,double, thick,->] (a)--(c);
				\draw[black, thick,->] (a)--(b);
				\draw[black, double, thick,->] (b)--(d);
				\node[squarednode,label=above:{$2$}](d) at (0,5) {$\bffour$};
				\node[squarednode,label=above:{$-2$}](b)  at (-2,2.5) {$\bftwo$};
				\node[squarednode,label=above:{$1$}](a)  at (0,0) {$\bfone$};
				\node[squarednode, label=above:{$1$}](c)  at (2,2.5) {$\bfthree$};
			\end{tikzpicture} & 
			& \begin{tikzpicture}[scale=0.4,
roundnode/.style={circle, draw=black!60, fill=gray!10, very thick, minimum size=5mm},
squarednode/.style={rectangle, draw=black!60, fill=gray!5, very thick, minimum size=5mm},]
				\draw[black,double, thick,->] (a)--(c);
				\draw[black, thick,->] (a)--(b);
				\draw[black, double, thick,->] (b)--(d);
				\node[squarednode,label=above:{$2$}](d) at (0,5) {$\bffour$};
				\node[squarednode,label=above:{$-2$}](b)  at (-2,2.5) {$\bftwo$};
				\node[squarednode,label=above:{$1$}](a)  at (0,0) {$\bfone$};
				\node[squarednode, label=above:{$-2$}](c)  at (2,2.5) {$\bfthree$};
			\end{tikzpicture}\\
			\scrF_\bfG=& \bx_1\bx_1\bx_1\bx_1 & + & \bx_1\bx_1\bx_1\bx_1 & + & \bx_1\bx_2\bx_1\bx_2& + & \bx_1\bx_2\bx_2\bx_2&+\cdots.
		\end{array}
	\end{align*}

\subsection{Peak algebra in noncommuting variables and labelled descent-to-peak map}\label{subsec:peak-theta-results-n} In this section, we show that all earlier definitions and results in Section \ref{subsec:peak-theta-results} can be generalized to noncommuting variables.

The \emph{peak algebra in noncommuting variables}, denoted $\NCPeak$, is the space spanned by
$$\{\scrF_\bfG: \text{$\bfG$ is a labelled edge-coloured digraph}\}.$$
The \emph{labelled descent-to-peak} map $\Theta_{\NCQSym}$ is   
 \begin{align}\label{Def:ThetaMapn}
\begin{array}{cccc}
\Theta_\NCQSym:& \NCQSym & \rightarrow & \NCPeak\\
& \scrY_\bfG & \mapsto & \scrF_\bfG.
\end{array} 
\end{align}

\begin{enumerate}[leftmargin=*]
\item A \emph{set composition} $\phi$ of $[n]$,   denoted $\phi\vDash[n]$, is a sequence of mutually disjoint nonempty sets whose union is $[n]$. An \emph{odd set composition} is a set composition whose all blocks have odd sizes. We show in Theorem \ref{Thm:NCPeakBasis} that the dimension of the space of homogeneous elements of degree $n$ of $\NCPeak$, $\NCPeak_n$, is  
\begin{align*}
\dim(\NCPeak_n)& = |\{(B,\sigma): B\subseteq \{2,3,\dots,n-1\} \text{~is a peak set and~}\Des(\sigma)\subseteq \odd(B) \}| \\
&=|\{\phi\vDash [n]: \text{$\phi$ is an odd set composition} \}|\\
&=a_n
\end{align*}
where 
$\Des(\sigma)=\{i\in [n-1]: \sigma(i)>\sigma(i+1)\},$ and $a_n$ is the sequence A006154 in the OEIS \cite{OEISA006154}.

\item  In Section \ref{Sec:NCPeak}, we study the structure of $\NCPeak$ and show that it is a Hopf algebra.

\item   In Theorem \ref{thm:theta}, we show that the labelled descent-to-peak map,
 $\Theta_{\NCQSym}$, is a surjective Hopf algebra morphism and the following diagram commutes.
 
\begin{center}
\begin{tikzpicture}
	\node(NCSym) at (4,0){$\mathsf{\NCQSym}$};
	\node(NCQSym) at (8,0){$\QSym$};
	\node(NCGamma) at (4,-4){$\NCPeak$};
	\node(NCPeak) at (8,-4){$\Pi$};
	
	\draw[thick,->]  (NCSym)->(NCQSym);
	\draw[thick,->]  (NCGamma)->(NCPeak);
	\draw[thick,->]   (NCSym)->(NCGamma);
	\draw[thick,->]   (NCQSym)->(NCPeak);
	
	\node(iota1) at (6,0.2){$\rho$};
	\node(iota2) at (6,-3.8){$\rho$};
	\node(ThetaNCSym) at (3.15,-2){$\Theta_{\NCQSym}$};
	\node(ThetaNCQSym) at (8.65,-2){$\Theta_\QSym$};
	
	\node(yg) at (4.5,-1){$\scrY_\bfG$};
	\node(fg) at (4.5,-3){$\scrF_\bfG$};
	
		\node(xg) at (7.5,-1){$\scrX_G$};
	\node(eg) at (7.5,-3){$\scrE_G$};

	
	\draw[thick,|->] (yg)--(fg);
	\draw[thick,|->] (xg)--(eg);
	\draw[thick,|->] (yg)--(xg);
	\draw[thick,|->] (fg)--(eg);
\end{tikzpicture} 
\end{center}

\item Given a subset $A=\{a_1<a_2<\cdots<a_{k}\}$ of $[n-1]$
 and $\sigma\in\mathfrak{S}_n$, we say $(A,\sigma)$ is \emph{standard} if $\Des(\sigma)\subseteq A.$ The value of the labelled descent-to-peak map at the monomial and fundamental bases elements of $\NCQSym$ are as follows. 
 \begin{enumerate} 
\item   
{\bf Fundamental basis.} 
Let $A=\{a_1<a_2<\cdots<a_{k}\}$ of $[n-1]$
 and $\sigma\in\mathfrak{S}_n$ such that $(A,\sigma)$ is standard. The \emph{fundamental} basis element $\bF_{(A,\sigma)}$ of $\NCQSym$ is the generalized chromatic function in noncommuting variables of the following labelled edge-coloured digraph,
 $$\bfG=Q_{\{\sigma(1),\dots,\sigma(a_1)\}}\rightarrow Q_{\{\sigma(a_1+1), \dots, \sigma(a_2)\}}\rightarrow \dots \rightarrow Q_{\{\sigma(a_k+1),\dots,\sigma(n)\}}.$$
 Therefore, by the description of the labelled descent-to-peak map in (\ref{Def:ThetaMapn})
 $$\Theta_{\NCQSym}(\bF_{(A,\sigma)})=\Theta_{\NCQSym}(\scrY_{\bfG})=\scrF_\bfG.$$
 
\item {\bf Monomial basis.} Let $A=\{a_1<a_2<\cdots<a_{k}\}$ of $[n-1]$
 and $\sigma\in\mathfrak{S}_n$ such that $(A,\sigma)$ is standard. The \emph{monomial} basis element $\bM_{(A,\sigma)}$ of $\NCQSym$ is
$$\bM_{(A,\sigma)}=\sum_{C\subseteq A} (-1)^{|A|-|C|} \bF_{(C,\sigma)}.$$

For each peak set $B\subseteq \{2,3,\dots, n-1\}$, the \emph{monomial peak function in noncommuting variables} $\boldeta_{(B,\sigma)}$ is 
$$\boldeta_{(B,\sigma)}=(-1)^{|B|}\sum_{A\subseteq \odd(B)} 2^{|A|+1}\bM_{(A,\sigma)}.$$

In Proposition \ref{prop:M} we show that 
\begin{align*}
\Theta_{\QSym}(\bM_{(A,\sigma)})
&=\begin{cases}
(-1)^{n-1-|B|-|A|} \boldeta_{(B,\sigma)} & \text{if $n-\max(A)$ is odd,}\\
0 & \text{otherwise.}
\end{cases} 
\end{align*}

\end{enumerate}

\item If $f$ is a rational valued function from $\sqcup_{n\geq 0}\{A:A\subseteq [n-1]\}$, then we define $f_{0\text{-}1}$ to be a rational valued function from the set of $0\text{-}1$ sequences such that $f_{0\text{-}1}(a_1a_2\cdots a_\ell(a))=f(A)$ where $A$ is a subset of $[\ell(a)]$ with $i\in A$ if and only if $a_i=1$.  In Section \ref{sec:D-S}, we show that the following are equivalent.
 \begin{enumerate}
  \item $f_{0\text{-}1}$ satisfies generalized Dehn-Sommerville equation.
  \item $f_{0\text{-}1}$ is a linear combination of flag $f$-vectors of Eulerian posets.
  \item $\sum f(A)M_A$ is in $\Pi$.
  \item For each $n$, fix $\sigma_n\in\mathfrak{S}_n$. Then $\sum_n\sum_{A\subseteq[n-1]}f(A)\bfM_{(A,\sigma_n)}$ is in $\NCPeak$.
 \end{enumerate}

\item  A \emph{set partition} $\pi$ of $[n]$ is the set of mutually disjoint non-empty subsets $\pi_1,\pi_2,\dots,\pi_{\ell(\pi)}$ of $[n]$ whose their union is $[n]$. We denote it by $\pi=\pi_1/\pi_2/\cdots/\pi_{\ell(\pi)}\vdash [n]$.
\index{$\pi$ set partition}\index{All bases of $\NCSym$ are bold.}

For a set $S$, let $\mathfrak{S}_S$ be the set of all bijections from $S$ to $S$.   Given a labelled edge-coloured digraph $\bfG$ with vertex set $S$ and $\sigma\in \mathfrak{S}_S$, define $\sigma\circ \bfG$ to be the labelled edge-coloured digraph with vertex set $S$ in which 
\begin{itemize} 
 \item $i\Rightarrow j$ in $\bfG$ if and only if ${\sigma(i)}\Rightarrow {\sigma(j)}$ in $\sigma\circ \bfG$.
\item $i\rightarrow j$ in $\bfG$ if and only if ${\sigma(i)}\rightarrow {\sigma(j)}$ in $\sigma\circ \bfG$.
\end{itemize}

 The values of the labelled descent-to-peak map at different bases of $\NCSym$ are as follows.
\begin{enumerate} 
\item For $\pi=\pi_1/\pi_2/\cdots/\pi_{\ell(\pi)}\vdash [n]$, by \cite[Section 10]{ALvW23},  we have that the \emph{complete homogeneous symmetric function in noncommuting variables} $\bfh_\pi$ (resp.,  \emph{elementary symmetric function in noncommuting variables} $\bfe_\pi$) is 
$$\bfh_\pi=\sum_{\sigma\in \frakS_{\pi}} \scrY_{\sigma\circ Q_{\pi}},~~~{\text{(resp.}}~ \bfe_\pi=\sum_{\sigma\in \frakS_{\pi}} \scrY_{\sigma\circ P_{\pi}}),$$  where 
\begin{center} $Q_\pi=Q_{\pi_1}\sqcup Q_{\pi_2}\sqcup \cdots \sqcup Q_{\pi_{\ell(\pi)}},$ (resp. $P_\pi=P_{\pi_1}\sqcup P_{\pi_2}\sqcup \cdots \sqcup P_{\pi_{\ell(\pi)}}$)
and $\mathfrak{S}_{\pi}=\mathfrak{S}_{\pi_1}\times \mathfrak{S}_{\pi_2} \times \cdots \times \mathfrak{S}_{\pi_{\ell(\pi)}}.$
\end{center} 
Thus by (\ref{Def:ThetaMapn}),
$$\Theta_{\NCSym}(\bfh_\pi)=\sum_{\sigma\in \frakS_{\pi}} \scrF_{\sigma\circ Q_{\pi}}~~~\text{(resp.}~ \Theta_{\NCSym}(\bfe_\pi)=\sum_{\sigma\in \frakS_{\pi}} \scrF_{\sigma\circ P_{\pi}}).$$ As for each $\sigma\in \mathfrak{S}_\pi$, $\scrF_{\sigma\circ Q_{\pi}}=\scrF_{\sigma\circ P_{\pi}}$, we conclude that 
\begin{equation}\label{eq:theta-h-e}
\Theta_{\NCSym}(\bfh_\pi) = \Theta_{\NCSym}(\bfe_\pi).
\end{equation}


\item For $\pi=\pi_1/\pi_2/\cdots/\pi_{\ell(\pi)}\vdash [n]$, define the permutation $\delta_\pi\in\mathfrak{S}_n$ such that  
if $\delta_\pi(i)\in \pi_s$ and $\delta_\pi(i+1)\in \pi_t$, then either $s=t$ and $\delta_\pi(i)<\delta_\pi(i+1)$ or $\min(\pi_s)<\min(\pi_t)$. For example, if $\pi=156/24/36$, then $\delta_\pi=1562436$.
The \emph{Schur function in noncommuting variable} $\bfs_\pi$ is 
 \[\bfs_\pi=\delta_{\pi}\circ \bfdet \left( \scrY_{Q_{[\lambda_i-i+j]}}  \right)_{1\leq i,j\leq \ell(\pi)}.\]
 Thus by (\ref{Def:ThetaMapn}),
$$\Theta_{\NCSym}(\bfs_\pi)=\delta_{\pi}\circ \bfdet \left( \scrF_{Q_{[\lambda_i-i+j]}}  \right)_{1\leq i,j\leq \ell(\pi)}.$$

\item  For $\pi=\pi_1/\pi_2/\cdots/\pi_{\ell(\pi)}\vdash [n]$, by Theorem \ref{thm:NCSym} we have
$$\Theta_\NCSym(\bfp_\pi)=\begin{cases}
2^{\ell(\pi)}\bfp_\pi & \text{if all blocks of $\pi$ have odd sizes,}\\
0 & \text{otherwise.}
\end{cases} $$

\end{enumerate}

\item For an odd set partition $\pi$, the \emph{Schur $Q$-function in noncommuting variables} is 
$$\mathbf{q}_\pi=\sum_{\sigma\in \frakS_{\pi}} \scrF_{\sigma\circ Q_{\pi}}=\sum_{\sigma\in \frakS_{\pi}} \scrF_{\sigma\circ P_{\pi}}.$$
The Hopf algebra of \emph{Schur Q-functions in noncommuting variables} is denoted by $\NCGamma$ and by Theorem \ref{thm:qschur},
$$\NCGamma=\mathbb{C}\text{-span}\{\mathbf{q}_\pi: \pi \text{~is an odd set partition}\}.$$ 
 and by the action of $\Theta_\NCSym$ on $\mathbf{h}_\pi$, we have $$\NCGamma=\Theta_\NCQSym(\NCSym).$$ 
By Theorem \ref{thm:NCGamma},
$$\NCGamma=\NCPeak\cap \NCSym.$$
The restriction of the labelled descent-to-peak map $\Theta_{\NCQSym}$ to $\NCSym$ is denoted by $\Theta_\NCSym$. By Corollary \ref{cor:theta-gcf} the following diagram commutes.

\begin{center}
\begin{tikzpicture}
	\node(NCSym) at (4,0){$\mathsf{\NCSym}$};
	\node(NCQSym) at (8,0){$\NCQSym$};
	\node(NCGamma) at (4,-4){$\NCGamma$};
	\node(NCPeak) at (8,-4){$\NCPeak$};
	
	\draw[thick,->]  (NCSym)->(NCQSym);
	\draw[thick,->]  (NCGamma)->(NCPeak);
	\draw[thick,->]   (NCSym)->(NCGamma);
	\draw[thick,->]   (NCQSym)->(NCPeak);
	
	\node(iota1) at (6,0.2){$\iota$};
	\node(iota2) at (6,-3.8){$\iota$};
	\node(ThetaNCSym) at (4-0.7,-2){$\Theta_{\NCSym}$};
	\node(ThetaNCQSym) at (9,-2){$\Theta_\NCQSym$};
	
	\node(yg) at (7.5,-1){$\scrY_\bfG$};
	\node(fg) at (7.5,-3){$\scrF_\bfG$};

	
	\draw[thick,|->] (yg)--(fg);
\end{tikzpicture} 
\end{center}

More generally, the following diagram commutes.

\begin{center}
\begin{tikzpicture}
	\node(Sym) at (5,0){${\Sym}$};
	\node(QSym) at (8,0){$\QSym$};
	\node(Gamma) at (5,-4){$\Omega$};
	\node(Pi) at (8,-4){$\Pi$};
	
	\node(NCSym) at (3,2){$\NCSym$};
	\node(NCQSym) at (11,2){$\NCQSym$};
	\node(NCGamma) at (3,-6){$\NCGamma$};
	\node(NCPeak) at (11,-6){$\NCPeak$};

	\draw[thick,->]  (Sym)->(QSym);
	\draw[thick,->]  (Gamma)->(Pi);
	\draw[thick,->]   (Sym)->(Gamma);
	\draw[thick,->]   (QSym)->(Pi);
	
	
	\draw[thick,->]  (NCSym)->(NCQSym);
	\draw[thick,->]  (NCGamma)->(NCPeak);
	\draw[thick,->]   (NCSym)->(NCGamma);
	\draw[thick,->]   (NCQSym)->(NCPeak);

	

	\node(iota1) at (6.5,0.2){$\iota$};
	\node(iota2) at (6.5,-3.8){$\iota$};
	\node(ThetaSym) at (5-0.5,-2){$\Theta_{\Sym}$};
	\node(ThetaQSym) at (7.3,-2){$\Theta_\QSym$};
	
        \node(iota3) at (6.5,2.2){$\iota$};
	\node(iota4) at (6.5,-3.8-2){$\iota$};
	\node(ThetaNCSym) at (2.3,-2){$\Theta_{\NCSym}$};
	\node(ThetaNCQSym) at (11.9,-2){$\Theta_\NCQSym$};
	
	\node(xg) at (8.7,-0.7){$\scrX_G$};
	\node(eg) at (8.7,-3.3){$\scrE_G$};
	\draw[thick,|->] (xg)--(eg);
	
	\node(yg) at (10.4,0.5){$\scrY_\bfG$};
	\node(fg) at (10.4,-4.5){$\scrF_\bfG$};
	
	\draw[thick,|->] (yg)--(fg);
	\draw[thick,|->] (yg)--(xg);
	\draw[thick,|->] (fg)--(eg);
	
	\draw[thick,->] (NCQSym)--(QSym);
	\draw[thick,->] (NCPeak)--(Pi);
	\draw[thick,->] (NCGamma)--(Gamma);
	\draw[thick,->] (NCSym)--(Sym);
	

	
	

\end{tikzpicture} 
\end{center}

\item   In Section \ref{Sec:com}, we show that the peak algebra in noncommuting variables is a left co-ideal of $\NCQSym$ under internal coproduct. We give a combinatorial description of the structure constants, and our result implies Schocker's result on $\QSym$.

 \end{enumerate}

\section{Preliminaries}\label{sec:per}

\subsection{Compositions}  
 Given a composition $\alpha=(\alpha_1,\alpha_2,\ldots, \alpha_k),$ each $\alpha_i$ is called a \emph{part} of $\alpha$, the \emph{size} of $\alpha$ is $|\alpha|=\alpha_1+\alpha_2+\cdots+\alpha_k$, and the \emph{length} of $\alpha$ is $k$. By convention, we denote by $\emptyset$ the unique composition of size and length zero.  For example, $(2,1,2)$ is a composition of $5$ with length $3$.

For every $A=\{a_1<a_2<\cdots<a_k\}\subseteq [n-1]$, define  
 \[ {\rm Comp}(A)=(a_1,a_2-a_1,\dots, n-a_k)\vDash [n]. \]
%
  For example,  if $A=\{2,3\}\subseteq [5]$, then \[{\rm Comp}(A)=(2,1,3).\] The map 
  \[
  \begin{array}{ccc}
  \{A: A\subseteq[n-1]\} & \rightarrow & \{\alpha: \alpha\vDash[n]\}\\
  A& \mapsto&  {\rm Comp}(A)
  \end{array} 
  \]
  is a bijection. 
  
%

\subsection{Odd compositions and peak sets}
 Recall that an odd composition is a composition whose all parts are odd.  We say a subset $A=\{a_1<a_2< \dots< a_k\}$ of $[n-1]$ is an \emph{odd set} if for each $1\leq i\leq k+1$, 
$$a_i-a_{i-1} \text{~is an odd number,}$$
where $a_0=0$ and $a_{k+1}=n.$  Therefore, for any odd set $A$, ${\rm Comp}(A)$ is an odd composition.

 As we mentioned earlier, a set $B\subseteq \{2,3,\dots,n-1\}$ is called a {peak set}  if $b\in B$ implies that $\{b-1,b+1\}\cap B=\emptyset$.  For any peak set $B\subseteq \{2,3,\dots,n-1\}$, define
$${\odd}(B)=[n-1]\setminus (B\cup (B-1))$$ where $B-1=\{b-1:b\in B\}.$ 
For example, if $B=\{2,5,7,10\}\subseteq \{2,3,\dots, 10\}$, then \[\odd(B)=\{3,8\}.\]


\begin{Proposition}\label{Pro:oddbij}
The map $$\begin{array}{cccc}
  \{B\subseteq \{2,3,\dots, n-1\}: \text{~$B$ is a peak set}\}& \rightarrow & \{A\subseteq [n-1]: \text{~$A$ is an odd set}\}\\
 B & \mapsto & \odd(B).
\end{array}$$
 is a bijection.
\end{Proposition}

\begin{proof}
Let $B\subseteq\{2,3,\dots, n-1\}$ be a peak set, and let 
$$\odd(B)=[n-1]\setminus (B\cup (B-1))=\{a_1<a_2<\dots<a_k\}.$$ Set $a_0=0$. Consider that for $1\leq i\leq k$, if $a_{i-1}\neq a_i-1$, then  $a_{i-1}+1 \in (B-1)$ and $a_{i}-1\in B$ (otherwise, $a_{i-1}\in B$ or $a_{i}\in B$).

 Since $B$ is a peak set, each of $B$ and $B-1$ does not include a pair of consecutive integers. Therefore, the sequence 
 $$a_{i-1}+1,a_{i-1}+2, \ldots, a_{i}-2,a_{i}-1$$
 is a sequence whose elements are alternatively in $B-1$ and $B$, that is 
 $$a_{i-1}+1\in (B-1),a_{i-1}+2\in B, \ldots, a_{i}-2\in (B-1) ,a_{i}-1\in B.$$
 Therefore, there are an even number of integers between $a_i$ and $a_{i-1}$, and so their difference $a_{i}-a_{i-1}$ is odd. This shows that  $\odd(B)$ is an odd set and the function $\odd$ is well-defined.

 The injectivity of the map $\odd$ is clear from the definition and now we show that $\odd$ is bijective. Given an odd set $A=\{a_1<a_2<\dots<a_k\}\subseteq [n-1]$, let
 \[B=\bigcup_{i=1}^{k+1}\{a_{i-1}+2t:t\in \mathbb{N}, a_{i}+2t\leq a_{i}-1\}\] where $a_0=0$ and $a_{k+1}=n.$ For example, if $A=\{3,8\}\subseteq [10]$, then $B=\{2,5,7,10\}.$
%
%
We have
 $$\odd(B)=A,$$
 which gives the inverse of $\odd$.
\end{proof} 

For any peak set $B$, define 
\[{\rm CompOdd}(B)={\rm Comp}({\rm Odd}(B)).\]
For example, if $B=\{2,5,7, 10\}\subseteq \{2,3,\dots,10\}$, then 
\[{\rm CompOdd}(B)={\rm Comp}(\{3,8\})=(3,5,3).\] 
Since always ${\rm CompOdd}(B)$ is an odd composition for any peak set $B$, we have the following corollary.
\begin{Corollary} 
The map 
$$\begin{array}{cccc}
  \{B\subseteq \{2,3,\dots, n-1\}: \text{~$B$ is a peak set}\}& \rightarrow & \{\alpha\vDash n: \text{~$\alpha$ is odd}\}\\
 B & \mapsto & {\rm Comp}\odd(B)
\end{array}$$
is a bijection.
\end{Corollary}

\subsection{$P$-partitions, enriched $P$-partitions, fundamental functions, and the peak algebra}\label{subsec:(enr)Ppar}

The notation and definitions in this section mainly follow from \cite{S97}.

Throughout this paper, $P=(X,\leq)$ denotes a partially ordered set of size $n$ with the ground set $X$. A \emph{linear extension} of $P$ is a total order $w=(w_1,w_2,\dots,w_n)$ of the ground set $X$ such that $w_i< w_j$ implies $i< j$. The set of linear extensions of $P$ is denoted by $\mathcal{L}(P)$. A \emph{labelled poset} is a pair $(P,\gamma)$ where $P$ is a poset with ground set $X$ and $\gamma:X\to[n]$ is a bijection.

Let $(P,\gamma)$ be a labelled poset. A \emph{$(P,\gamma)$-partition} is a map 
$$f:X\to\mathbb{N}=\{1,2,3,\dots\}$$ such that for all $a,b\in X$ with $a\leq b$, we have:
\begin{enumerate}
	\item $f(a)\leq f(b)$.
	\item If $\gamma(a)>\gamma(b)$, then $f(a)<f(b)$.
\end{enumerate}

The set of all $(P,\gamma)$-partitions is denoted by $\mathcal{A}(P,\gamma)$.

\subsection{Quasisymmetric functions}

Given $A\subseteq [n-1]$, define the \emph{monomial quasisymmetric function} $M_A$ to be 
\[
M_A=\sum x_{i_1}^{\alpha_1}x_{i_2}^{\alpha_2} \dots x_{i_k}^{\alpha_k}
\]
where $(\alpha_1,\alpha_2,\dots,\alpha_k)={\rm Comp}(A)$ and the sum is over all tuples $(i_1,i_2,\dots,i_k)$ of positive integers such that $i_1<i_2<\cdots <i_k$.

The Hopf algebra of \emph{quasisymmetric functions} is  
$$\QSym=\displaystyle\bigoplus_{n\geq 0}\QSym_n$$
where
$$\QSym_n=\mathbb{Q}\text{-span}\{M_A:A\subseteq [n-1]\}.$$
The basis $\{M_A\}$ is called the \emph{monomial basis} of the Hopf algebra of quasisymmetric functions.
%

\subsection{Fundamental functions}
Given a labelled poset $(P,\gamma)$ with ground set $X$, define the \emph{fundamental function} $F_{(P,\gamma)}$ to be
$$F_{(P,\gamma)}=\sum_{f\in\mathcal{A}(P,\gamma)}\prod_{a\in X}x_{f(a)}.$$

\begin{Example}\label{ex:fundfun}
Let $P$ be the following poset 
	 $$
 
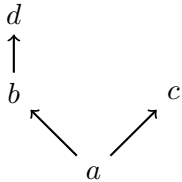
\begin{tikzpicture}[scale=0.7]
 \node(a) at (0,0){$a$};
 \node(b) at (-1.5,1.5){$b$};
 \node(c) at (1.5,1.5){$c$};
 \node(d) at (-1.5,3){$d$};

 \draw[thick,->] (a)--(b);
 \draw[thick,->] (a)--(c);
 \draw[thick,->] (b)--(d);
 \end{tikzpicture} 
 $$
 and 
 $$\gamma(a)\gamma(b)\gamma(c)\gamma(d)=2143.$$ 	Then
	\begin{align*}
		\begin{array}{ccccccccc}
			& \begin{tikzpicture}[scale=0.3]
				\draw[black, thick,->] (0.7,0.7) -- (1.3,1.3);
				\draw[black, thick,->] (-0.7,0.7) -- (-1.3,1.3);
				\draw[black, thick,->] (-2,2.8) -- (-2,4.2);
				\node at (-2,5) {$2$};
				\node at (-2,2) {$2$};
				\node at (0,0) {$1$};
				\node at (2,2) {$2$};
			\end{tikzpicture} &
			& \begin{tikzpicture}[scale=0.3]
				\draw[black, thick,->] (0.7,0.7) -- (1.3,1.3);
				\draw[black, thick,->] (-0.7,0.7) -- (-1.3,1.3);
				\draw[black, thick,->] (-2,2.8) -- (-2,4.2);
				\node at (-2,5) {$3$};
				\node at (-2,2) {$2$};
				\node at (0,0) {$1$};
				\node at (2,2) {$2$};
			\end{tikzpicture} &
			& \begin{tikzpicture}[scale=0.3]
				\draw[black, thick,->] (0.7,0.7) -- (1.3,1.3);
				\draw[black, thick,->] (-0.7,0.7) -- (-1.3,1.3);
				\draw[black, thick,->] (-2,2.8) -- (-2,4.2);
				\node at (-2,5) {$2$};
				\node at (-2,2) {$2$};
				\node at (0,0) {$1$};
				\node at (2,2) {$3$};
			\end{tikzpicture} & 
			& \begin{tikzpicture}[scale=0.3]
				\draw[black, thick,->] (0.7,0.7) -- (1.3,1.3);
				\draw[black, thick,->] (-0.7,0.7) -- (-1.3,1.3);
				\draw[black, thick,->] (-2,2.8) -- (-2,4.2);
				\node at (-2,5) {$3$};
				\node at (-2,2) {$2$};
				\node at (0,0) {$1$};
				\node at (2,2) {$3$};
			\end{tikzpicture}\\
			F_{(P,\gamma)}=& x_1x_2^3 & + & x_1x_2^2x_3 & + & x_1x_2^2x_3 & + & x_1x_2x_3^2 &+\cdots.
		\end{array}
	\end{align*}
\end{Example}

 The following is known as the fundamental theorem of $(P,\gamma)$-partitions.

\begin{Theorem}\label{thm:ppar}\cite[Lemma 3.15.3]{EC} 
	Let $(P,\gamma)$ be a labelled poset. Then
	$$\mathcal{A}(P,\gamma)=\bigsqcup_{w\in\mathcal{L}(P)}\mathcal{A}(w,\gamma).$$
	Consequently,
	$$F_{(P,\gamma)}=\sum_{w\in\mathcal{L}(P)}F_{(w,\gamma)}.$$
\end{Theorem}

Given $w=w_1w_2\dots w_n\in \mathcal{L}(P)$, 
$$\Des(w,\gamma)=\{i\in \{1,2,\dots,n-1\}:\gamma(w_i)>\gamma(w_{i+1})\}.$$ 

For $w\in\mathcal{L}(P)$, the function $F_{(w,\gamma)}$ only depends on the descent set $\Des(w,\gamma)$  (see \cite[Section 2.2]{S97}). That is, $F_{(w,\gamma)} = F_{(w,\gamma')}$ if and only if $\Des(w,\gamma)=\Des(w,\gamma')$. Given $A\subseteq [n-1]$, define $$F_A=F_{(w,\gamma)}$$ where $(w,\gamma)$ is any labelled chain  with $\Des(w,\gamma)=A$.

The set $\{F_A: A\subseteq [n-1]\}$ is a basis for $\QSym_n$, and the set $\{F_A\}$ is called the \emph{fundamental basis} of the Hopf algebra of quasisymmetric functions.  We have that 
\[ F_A=\sum_{A\subseteq C \subseteq [n-1]} M_{C}.\]

\begin{Example}\label{ex:fund-to-gcf}
Let $w=(w_1,w_2,w_3,w_4, w_5)$ be a chain, and 
$$\gamma(w_1)\gamma(w_2)\gamma(w_3)\gamma(w_4)\gamma(w_5)=34152,$$ then 
$$\Des(w,\gamma)=\{2,4\},$$ and so 
$$F_{\{2,4\}}=F_{(w,\gamma)}.$$ 
\end{Example} 


Let $(P,\gamma)$ be a labelled poset. An \emph{enriched $(P,\gamma)$-partition} is a map $$f:X\to\{ -1\prec1\prec -2\prec 2 \prec \cdots \}$$ such that for all pairs $a,b\in X$ with $a\leq b$, we have:
\begin{enumerate}
	\item $f(a)\preceq f(b)$.
	\item $f(a)=f(b)>0$ implies $\gamma(a)<\gamma(b)$.
	\item $f(a)=f(b)<0$ implies $\gamma(a)>\gamma(b)$.
\end{enumerate}
The set of all enriched $(P,\gamma)$-partitions is denoted by $\mathcal{B}(P,\gamma)$. 

Given  a labelled poset $(P,\gamma)$ with ground set $X$, define the \emph{enriched fundamental function}  $K_{(P,\gamma)}$ to be 
$$K_{(P,\gamma)}=\sum_{f\in\mathcal{B}(P,\gamma)}\prod_{a\in X}x_{|f(a)|}.$$

\begin{Example}
	Let $(P,\gamma)$ be the labelled poset in Example \ref{ex:fundfun}.
%
%
	Then
	\begin{align*}
		\begin{array}{ccccccccc}
			& \begin{tikzpicture}[scale=0.3]
				\draw[black, thick,->] (0.7,0.7) -- (1.3,1.3);
				\draw[black, thick,->] (-0.7,0.7) -- (-1.3,1.3);
				\draw[black, thick,->] (-2,2.8) -- (-2,4.2);
				\node at (-2,5) {$1$};
				\node at (-2,2) {$-1$};
				\node at (0,0) {$-1$};
				\node at (2,2) {$1$};
			\end{tikzpicture} &
			& \begin{tikzpicture}[scale=0.3]
				\draw[black, thick,->] (0.7,0.7) -- (1.3,1.3);
				\draw[black, thick,->] (-0.7,0.7) -- (-1.3,1.3);
				\draw[black, thick,->] (-2,2.8) -- (-2,4.2);
				\node at (-2,5) {$1$};
				\node at (-2,2) {$1$};
				\node at (0,0) {$-1$};
				\node at (2,2) {$1$};
			\end{tikzpicture} &
			& \begin{tikzpicture}[scale=0.3]
				\draw[black, thick,->] (0.7,0.7) -- (1.3,1.3);
				\draw[black, thick,->] (-0.7,0.7) -- (-1.3,1.3);
				\draw[black, thick,->] (-2,2.8) -- (-2,4.2);
				\node at (-2,5) {$2$};
				\node at (-2,2) {$-2$};
				\node at (0,0) {$1$};
				\node at (2,2) {$1$};
			\end{tikzpicture} & 
			& \begin{tikzpicture}[scale=0.3]
				\draw[black, thick,->] (0.7,0.7) -- (1.3,1.3);
				\draw[black, thick,->] (-0.7,0.7) -- (-1.3,1.3);
				\draw[black, thick,->] (-2,2.8) -- (-2,4.2);
				\node at (-2,5) {$2$};
				\node at (-2,2) {$-2$};
				\node at (0,0) {$1$};
				\node at (2,2) {$-2$};
			\end{tikzpicture}\\
			K_{(P,\gamma)}=& x_1^4 & + & x_1^4 & + & x_1^2x_2^2 & + & x_1x_2^3 &+\cdots.
		\end{array}
	\end{align*}
\end{Example}

The following is the analogue of the fundamental theorem for enriched $P$-partitions.

\begin{Theorem}\label{thm:eppar}\cite[Lemma 2.1]{S97}  
	Let $(P,\gamma)$ be a labelled poset. Then 
	$$\mathcal{B}(P,\gamma)=\bigsqcup_{w\in\mathcal{L}(P)}\mathcal{B}(w,\gamma).$$ Consequently,
	$$K_{(P,\gamma)}=\sum_{w\in\mathcal{L}(P)}K_{(w,\gamma)}.$$
\end{Theorem}


Given $w\in \mathcal{L}(P)$, $$\peak(w,\gamma)=\{i:\gamma(w_{i-1})<\gamma(w_i)>\gamma(w_{i+1})\}$$ is a peak set, and also any peak set is equal to some $\peak(w,\gamma)$.

For $w\in\mathcal{L}(P)$, the function $K_{(w,\gamma)}$ only depends on the peak set $\peak(w,\gamma)$ (see  \cite[Section 2.2]{S97}). That is, $K_{(w,\gamma)}=K_{(w,\gamma')}$ if and only if $\peak(w,\gamma)=\peak(w,\gamma')$. Given a peak set $B\subseteq\{2,3,\dots,n-1\}$, define $$K_B=K_{(w,\gamma)}$$ where $(w,\gamma)$ is any labelled chain with $\Peak(w,\gamma)=B$. 

The \emph{peak algebra} is 
$$\Pi=\displaystyle\bigoplus_{n\geq 0}\Pi_n$$
where
$$\Pi_n=\bbQ\text{-span}\{K_B:B\subseteq\{2,3,\dots,n-1\} \text{~is a peak set}\}.$$
The basis $\{K_B: B \text{~is a peak set}\}$ is called the {\it enriched fundamental basis} of the peak algebra. 

Given a peak set $B\subseteq \{2,3,\dots,n-1\}$, define the \emph{enriched monomial function} to be 
\[\eta_B=\sum_{C\subseteq B} (-1)^{|B|-|C|} K_{C}.\] 
 By inclusion-exclusion, we have 
 \[K_B=\sum_{C\subseteq B} \eta_C.\]

The \emph{descent-to-peak} map is 
$$\begin{array}{cccc}
	\Theta_\QSym: & \QSym & \to & \Pi\\
	& F_{(P,\gamma)} & \mapsto & K_{(P,\gamma)},
\end{array}$$
which is a surjective morphism of graded Hopf algebras (see \cite{BMSvW}).

When $(w,\gamma)$ is a labelled chain with descent set $A$ and peak set $B$, then 
$$\Theta_{\QSym}(F_{A})=\Theta_{\QSym}(F_{(w,\gamma)})=K_{(w,\gamma)}=K_B,$$ so  $\Theta_{\QSym}$ maps a fundamental basis element indexed by a descent set of a labelled chain to the enriched fundamental basis element indexed by the peak set of the same labelled chain. Consequently, the map $\Theta_{\QSym}$ is called the descent-to-peak map. 

	\begin{Lemma}\label{lem:ppar-to-gcf}
	Given a labelled poset $(P,\gamma)$, let $G$  be an edge-coloured digraph isomorphic to the Hasse diagram of $P$, and moreover,
	 \begin{enumerate}
	\item $ a\Rightarrow  b$ in $G$ if and and only if $b$ covers $a$ in $P$ and $\gamma(a)\leq \gamma(b)$.
	\item $ a\rightarrow  b$ in $G$ if and and only if $b$ covers $a$ in $P$ and $\gamma(a)> \gamma(b)$.
	\end{enumerate} 
	Then
	\begin{enumerate} 
	 \item[(i)] the set of $(P,\gamma)$-partitions and the set of proper colourings of $G$ are identical, 
	 \item[(ii)] $F_{(P,\gamma)}=\scrX_G,$
	 \item[(iii)] the set of enriched $(P,\gamma)$-partitions and the set of enriched colourings of $G$ are identical,
	  \item[(iv)] $K_{(P,\gamma)}=\scrE_G.$
	\end{enumerate} 
	\end{Lemma} 
	
	The above lemma enables us to translate the results about $(P,\gamma)$-partitions and peak algebra into the language of generalized chromatic functions.

\begin{Corollary}\label{cor:theta-gcf1}
The {descent-to-peak} map can be rewritten as 
$$\begin{array}{cccc}
	\Theta_\QSym: & \QSym & \to & \Pi\\
	& \scrX_G & \mapsto & \scrE_G.
\end{array}$$

\end{Corollary}

\begin{Example}\label{Ex:}
The representation of $F_{\{2,4\}}$ as a generalized chromatic function is as follows. 
$$\scrX_
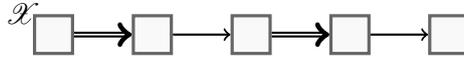
{
  \begin{tikzpicture}[
roundnode/.style={circle, draw=black!60, fill=gray!10, very thick, minimum size=7mm},
squarednode/.style={rectangle, draw=black!60, fill=gray!5, very thick, minimum size=5mm},]
 \node[squarednode](1) at (1.3,0){};
  \node[squarednode](2) at (2.6,0){};
    \node[squarednode](3) at (3.9,0){};
      \node[squarednode](4) at (5.2,0){};
        \node[squarednode](5) at (6.5,0){};
        
    \draw[thick,double,->](1)--(2);
       \draw[ thick,->](2)--(3);
          \draw[ thick,double,->](3)--(4);
             \draw[ thick,->](4)--(5);

 \end{tikzpicture}
}$$
The set $\{2,4\}$ clarifies the types of edges; the edges starting at the second and fourth vertices are solid edges, and the rest of the edges are double. 

The representation of $K_{\{2,4\}}$ as an enriched chromatic function is as follows. 
$$\scrE_{
  \begin{tikzpicture}[
roundnode/.style={circle, draw=black!60, fill=gray!10, very thick, minimum size=7mm},
squarednode/.style={rectangle, draw=black!60, fill=gray!5, very thick, minimum size=5mm},]
 \node[squarednode](1) at (1.3,0){};
  \node[squarednode](2) at (2.6,0){};
    \node[squarednode](3) at (3.9,0){};
      \node[squarednode](4) at (5.2,0){};
        \node[squarednode](5) at (6.5,0){};
        
    \draw[thick,double,->](1)--(2);
       \draw[ thick,->](2)--(3);
          \draw[ thick,double,->](3)--(4);
             \draw[ thick,->](4)--(5);

 \end{tikzpicture}
}$$

And by {Corollary \ref{cor:theta-gcf1}}
$$\Theta_{\QSym}(\scrX_{
  \begin{tikzpicture}[
roundnode/.style={circle, draw=black!60, fill=gray!10, very thick, minimum size=7mm},
squarednode/.style={rectangle, draw=black!60, fill=gray!5, very thick, minimum size=5mm},]
 \node[squarednode](1) at (1.3,0){};
  \node[squarednode](2) at (2.6,0){};
    \node[squarednode](3) at (3.9,0){};
      \node[squarednode](4) at (5.2,0){};
        \node[squarednode](5) at (6.5,0){};
        
    \draw[thick,double,->](1)--(2);
       \draw[ thick,->](2)--(3);
          \draw[ thick,double,->](3)--(4);
             \draw[ thick,->](4)--(5);

 \end{tikzpicture}
})=\scrE_{
  \begin{tikzpicture}[
roundnode/.style={circle, draw=black!60, fill=gray!10, very thick, minimum size=7mm},
squarednode/.style={rectangle, draw=black!60, fill=gray!5, very thick, minimum size=5mm},]
 \node[squarednode](1) at (1.3,0){};
  \node[squarednode](2) at (2.6,0){};
    \node[squarednode](3) at (3.9,0){};
      \node[squarednode](4) at (5.2,0){};
        \node[squarednode](5) at (6.5,0){};
        
    \draw[thick,double,->](1)--(2);
       \draw[ thick,->](2)--(3);
          \draw[ thick,double,->](3)--(4);
             \draw[ thick,->](4)--(5);

 \end{tikzpicture}.
}$$

\end{Example}

\section{Standard and enriched standard pairs}\label{sec:sten}

\subsection{Set compositions and standard pairs}
Recall that a {set composition} $\phi$ of $[n]$ is a list  of mutually disjoint  nonempty subsets  $\phi_1,\phi_2,\ldots,\phi_{\ell(\phi)}$ of $[n]$ whose union is $[n]$; this is denoted by $\phi_1|\phi_2|\cdots|\phi_{\ell(\phi)}\vDash [n]$.
 Each $\phi_i$ is called a \emph{block} of the set composition $\phi$, and the \emph{length} of $\phi$ is $\ell(\phi)$.  By convention, we denote by $\emptyset$ the unique empty set composition of $[0]=\emptyset$.  For example, $\phi=12|5|34\vDash [5]$ has length 3.


Given a pair $(A,\sigma)$, $A=\{a_1<a_2<\dots<a_k \}\subseteq [n-1]$ and $\sigma\in \frakS_n$, define 
\[\SetComp(A,\sigma)=\sigma(1)\cdots \sigma(a_1)~|~\sigma(a_1+1)\cdots \sigma(a_2)~|\cdots |~ \sigma(a_k) \cdots \sigma(n)\vDash [n].\]
For example, 
$$\SetComp(\{2,5,6\},67325841)=67|325|8|41.$$

%

	Let  $A\subseteq[n-1]$ and $\sigma\in\mathfrak{S}_n$. Remember that the pair $(A,\sigma)$ is {standard} if  $\Des(\sigma)\subseteq A$, that is $a\notin A$ implies $\sigma(a)<\sigma(a+1)$. 
	For example, when  $A=\{2,5,6\}$ and $\sigma = 37124856$, then $(A,\sigma)$ is standard and $\SetComp(A,\sigma)=37|124|8|56$. 
Note that for any pair $(A,\sigma)$, $A\subseteq [n-1]$ and $\sigma\in \frakS_n$, there is a unique $\sigma'\in\mathfrak{S}_n$ such that $(A,\sigma')$ is standard and $\SetComp(A,\sigma)=\SetComp(A,\sigma')$. We denote $(A,\sigma')$ by $\std(A,\sigma)$. 
For example,
$$\std(\{3,5\},37284516)=(\{3,5\},23748156).$$

\begin{Proposition}\label{prop:std}
The map 
$$\begin{array}{cccc}
& \{ (A,\sigma): A\subseteq [n-1], \sigma\in \frakS_n,  (A,\sigma)~ {\rm is~ standard}\}& \rightarrow & \{ \phi: \phi\vDash[n]\} \\
& (A,\sigma)  & \mapsto & \SetComp(A,\sigma)
\end{array}$$
is a bijection. 	
	
\end{Proposition}

%
%

\begin{proof} We first show that the function $\SetComp$ is surjective. Given a set composition $\phi=\phi_1|\phi_2|\dots|\phi_{\ell(\phi)}$, let $\phi_{ij}$ be the $j$th smallest element in $\phi_i$. Let $k=\ell(\phi)-1$. Let $|\phi_i|=a_i-a_{i-1}$, $1\leq i\leq k+1$, with $a_0=0$ and $a_{k+1}=n$.
Set $$\sigma=\phi_{11}\phi_{12}\dots\phi_{1a_1} \phi_{21}\phi_{22}\dots\phi_{2(a_2-a_1)} \dots \phi_{k1}\phi_{k2}\dots\phi_{k(n-a_k)}$$
and 
$$A=\{a_1,a_2,\dots, a_k\}$$ then 
$$\SetComp(A,\sigma)=\phi_{11}\phi_{12}\dots\phi_{1a_1}~|~ \phi_{21}\phi_{22}\dots\phi_{2(a_2-a_1)}~| \dots|~ \phi_{k1}\phi_{k2}\dots\phi_{k(n-a_k)}=\phi.$$
Therefore, the function $\SetComp$ is surjective.

Now we show that $\SetComp$ is an injection. If $\phi=\SetComp(A,\sigma)=\SetComp(A',\sigma')$, since $A$ and $A'$ identify the sizes of the blocks of $\phi$, we have that $A=A'$. Moreover, by the definition of the function $\SetComp$, we must have $\sigma=\sigma'$. Therefore, $\SetComp$ is injective.  

\end{proof}

\subsection{Odd set compositions and enriched standard pairs}

An \emph{odd} set composition is a set composition whose all blocks have odd sizes. We say a pair $(A,\sigma)$, $A\subseteq[n-1]$ and $\sigma\in \mathfrak{S}_n$, is \emph{odd} if $A$ is an odd set. Therefore, for any odd pair $(A,\sigma)$, ${\rm SetComp}(A,\sigma)$ is an odd set composition.

A pair $(B,\sigma)$, $B\subseteq \{2,3,\dots, n-1\}$ is a peak set and $\sigma\in \frakS_n$,  is called \emph{enriched standard} if $(\odd(B),\sigma)$ is standard, that is $\Des(\sigma)\subseteq \odd(B)$. 
For example, when  $B=\{2,5,7\}$ and $\sigma = 278134569$, then $(B,\sigma)$ is enriched standard since 
$$\Des(\sigma)=\{3\}\subseteq \odd(B)=\{3,8\}.$$

Note that for any pair $(B,\sigma)$,  $B\subseteq \{2,3,\dots, n-1\}$ is a peak set and $\sigma\in \frakS_n$, there is a unique $\sigma'\in\mathfrak{S}_n$ such that $(B,\sigma')$ is enriched standard. We denote $(B,\sigma')$ by ${\rm EStd}(B,\sigma)$.   For example, 
\[{\rm EStd}(\{2,5,7\},287134659)=(\{2,5,7\},278134569).\]

For each enriched standard pair $(B,\sigma)$, define 
\[ \OddSetComp(B,\sigma)=\SetComp(\odd(B),\sigma). \]

\begin{Proposition}\label{Prop:OddSetComp}
The map from the set of enriched standard pairs $(B,\sigma)$, $B\subseteq[n-1]$ and $\sigma\in\mathfrak{S}_n$, to the set of odd set compositions of $[n]$ given by 
\[
\begin{array}{cccc}
& (B,\sigma) & \mapsto & \OddSetComp(B,\sigma).
\end{array} 
\]
is a bijection.
\end{Proposition} 
\begin{proof}
To show that $\OddSetComp$ is a bijection, we only need to show that it is surjective. Given an odd set composition 
$$\phi=\phi_1|\phi_2|\cdots|\phi_{\ell(\phi)},$$ 
let $B$ be a peak set such that 
$$\odd(B)=\{ |\phi_1|< |\phi_1|+ |\phi_2|< \cdots<|\phi_1|+|\phi_2|+\cdots+|\phi_{\ell(\phi)-1}| \},$$ and let $\sigma\in\mathfrak{S}_n$ be the permutation for which
$$ \sigma(\{|\phi_{1}|+ \cdots+|\phi_{i-1}|+1,|\phi_{1}|+ \cdots+|\phi_{i-1}|+2, |\phi_{1}|+ \cdots+|\phi_{i-1}|+|\phi_{i}| )=\sort(\phi_i)$$ for all $i$. Then 
$$\OddSetComp(B,\sigma)=\phi.$$
\end{proof}

\section{Fundamental functions in noncommuting variables and the analogue of the fundamental theorem}\label{Sec:fundamental}

\subsection{Quasisymmetric functions in noncommuting variables}\label{subsec:NCQSym}

Recall that $\bbQ\langle \langle \bx_1,\bx_2,\dots\rangle\rangle$ is the algebra of formal power series in infinitely many noncommuting variables $\bx=\{\bx_1,\bx_2,\dots\}$ over $\bbQ$. 

Given $(A,\sigma)$, $A\subseteq [n-1]$ and $\sigma\in \mathfrak{S}_n$, define the \emph{monomial quasisymmetric function in noncommuting variables} $\bfM_{(A,\sigma)}$ to be  
$$\bfM_{(A,\sigma)}=\sum \bx_{i_1}\bx_{i_2}\cdots \bx_{i_n}$$ where the sum is over all tuples $(i_1,i_2,\dots,i_n)$ of positive integers such that 
\begin{enumerate} 
\item $i_j=i_k$ if and only if $j$ and $k$ are in the same block of $\SetComp(A,\sigma)$.
\item $i_j<i_k$ if and only if $j$ is in a block of $\SetComp(A,\sigma)$ to the left of the block containing $k$.
\end{enumerate} 

Note that for pairs $(A,\sigma_1)$ and $(A,\sigma_2)$ with $\std(A,\sigma_1)=\std(A,\sigma_2)$, we have 
 $$\bM_{(A,\sigma_1)}=\bM_{(A,\sigma_2)}.$$

The Hopf algebra of \emph{quasisymmetric functions in noncommuting variables} is  
$$\NCQSym=\displaystyle\bigoplus_{n\geq 0}\NCQSym_n$$
where
$$\NCQSym_n=\mathbb{Q}\text{-span}\{ \bM_{(A,\sigma)}: A\subseteq [n-1], \sigma\in \mathfrak{S}_n, \text{~$(A,\sigma)$ is standard}  \}.$$

The basis $\{ \bM_{(A,\sigma)}:\text{$(A,\sigma)$ is standard}\}$ is called the \emph{monomial basis} of the Hopf algebra of quasisymmetric functions in noncommuting variables.

 
 \subsection{Fundamental functions in noncommuting variables}\label{subsec:fund-fun}
Given $(P,\gamma,\sigma)$ where $(P,\gamma)$ is a labelled poset with ground set $X$, $|X|=n$, and $\sigma:X\to[n]$ is a bijection, define the \emph{fundamental function  in noncommuting variables} $\mathbf{F}_{(P,\gamma,\sigma)}$ to be
$$\mathbf{F}_{(P,\gamma,\sigma)}=\sum_{f\in\mathcal{A}(P,\gamma)}\prod_{i=1}^{n}\mathbf{x}_{f(\sigma^{-1}(i))}=\sum_{f\in\mathcal{A}(P,\gamma)}\mathbf{x}_{f(\sigma^{-1}(1))}\mathbf{x}_{f(\sigma^{-1}(2))}\cdots \mathbf{x}_{f(\sigma^{-1}(n))}.$$

\begin{Example}\label{ex:labelfund}
	Given $(P,\gamma)$  in Example \ref{ex:fundfun}
%
%
 and
 $$\sigma(a)\sigma(b)\sigma(c)\sigma(d)=1234,$$ 
 \[
		\begin{array}{ccccccccc}
			& \begin{tikzpicture}[scale=0.3]
				\draw[black, thick,->] (0.7,0.7) -- (1.3,1.3);
				\draw[black, thick,->] (-0.7,0.7) -- (-1.3,1.3);
				\draw[black, thick,->] (-2,2.8) -- (-2,4.2);
				\node at (-2,5) {$2$};
				\node at (-2,2) {$2$};
				\node at (0,0) {$1$};
				\node at (2,2) {$2$};
			\end{tikzpicture} &
			& \begin{tikzpicture}[scale=0.3]
				\draw[black, thick,->] (0.7,0.7) -- (1.3,1.3);
				\draw[black, thick,->] (-0.7,0.7) -- (-1.3,1.3);
				\draw[black, thick,->] (-2,2.8) -- (-2,4.2);
				\node at (-2,5) {$3$};
				\node at (-2,2) {$2$};
				\node at (0,0) {$1$};
				\node at (2,2) {$2$};
			\end{tikzpicture} &
			& \begin{tikzpicture}[scale=0.3]
				\draw[black, thick,->] (0.7,0.7) -- (1.3,1.3);
				\draw[black, thick,->] (-0.7,0.7) -- (-1.3,1.3);
				\draw[black, thick,->] (-2,2.8) -- (-2,4.2);
				\node at (-2,5) {$2$};
				\node at (-2,2) {$2$};
				\node at (0,0) {$1$};
				\node at (2,2) {$3$};
			\end{tikzpicture} & 
			& \begin{tikzpicture}[scale=0.3]
				\draw[black, thick,->] (0.7,0.7) -- (1.3,1.3);
				\draw[black, thick,->] (-0.7,0.7) -- (-1.3,1.3);
				\draw[black, thick,->] (-2,2.8) -- (-2,4.2);
				\node at (-2,5) {$3$};
				\node at (-2,2) {$2$};
				\node at (0,0) {$1$};
				\node at (2,2) {$3$};
			\end{tikzpicture}\\
			\bF_{(P,\gamma,\sigma)}=& \bx_1\bx_2\bx_2\bx_2 & + & \bx_1\bx_2\bx_3 & + & \bx_1\bx_2\bx_3\bx_2 & + & \bx_1\bx_2\bx_3\bx_3 &+\cdots.
		\end{array}
		\]
\end{Example}

\begin{Definition} 
Let $(P,\gamma)$ be a labelled poset with ground set $X$ and $|X|=n$. Let $\sigma:X\to[n]$ be a bijection. The labelled edge-coloured digraph $\bfG$ associated to $(P,\gamma,\sigma)$  is a labelled edge-coloured digraph $\bfG$ such that 
\begin{enumerate}
\item The vertex set of $\bfG$ is $[n]$.
\item $\sigma^{-1}(a) \Rightarrow \sigma^{-1}(b)$ in $\bfG$ if and only if $b$ covers $a$ in $P$ and $\gamma(\sigma^{-1}(a))< \gamma(\sigma^{-1}(b))$.
\item $\sigma^{-1}(a) \rightarrow \sigma^{-1}(b)$ in $\bfG$ if and only if $b$ covers $a$ in $P$ and $\gamma(\sigma^{-1}(a))> \gamma(\sigma^{-1}(b))$.
 \end{enumerate} 
\end{Definition} 

\begin{Lemma}\label{Def:PosetGraph}
 Given the labelled edge-coloured digraph $\bfG$ associated to $(P,\gamma,\sigma)$, we have 
 $$\bfF_{(P,\gamma,\sigma)}=\scrY_\bfG.$$
 \end{Lemma}

 \begin{Example}\label{ex:PosetGraph}
Given $(P,\gamma, \sigma)$ in  Example \ref{ex:labelfund},
 $$\bfF_{(P,\gamma,\sigma)}=\scrY_\bfG$$ where $\bfG$ is the following edge-coloured digraph. 
 
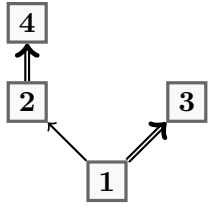
\begin{center} 
  \begin{tikzpicture}[scale=0.7,
roundnode/.style={circle, draw=black!60, fill=gray!10, very thick, minimum size=5mm},
squarednode/.style={rectangle, draw=black!60, fill=gray!5, very thick, minimum size=5mm},]

 \node[squarednode](a) at (0,0){$\bfone $};
 \node[squarednode](b) at (-1.5,1.5){$\bftwo $};
 \node[squarednode](c) at (1.5,1.5){$\bfthree $};
 \node[squarednode](d) at (-1.5,3){$\bffour $};
 
 \draw[thick,->] (a)--(b);
 \draw[thick,double,->] (a)--(c);
 \draw[thick,double,->] (b)--(d);
 \end{tikzpicture} 
  \end{center}

 \end{Example}

\subsection{The analogue of the fundamental theorem}\label{subsec:fund-thm-n}

From now on  $w=(w_1,w_2,\dots,w_n)$ is a chain with the ground set $\{w_1,w_2, \dots, w_n\}=\{1,2,\dots,n\}$. So we can look at any bijection  $\sigma: \{w_1,w_2, \dots, w_n\} \rightarrow [n]$ as a permutation in $\frakS_n$.

\begin{Proposition}\label{prop:fund1} We have that
	$$\mathbf{F}_{(P,\gamma,\sigma)}=\sum_{w\in\mathcal{L}(P)}\mathbf{F}_{(w,\gamma,\sigma)}.$$
\end{Proposition}

\begin{proof}
	By Theorem \ref{thm:ppar}, we have $\mathcal{A}(P,\gamma)=\bigsqcup_{w\in\mathcal{L}(P)}\mathcal{A}(w,\gamma)$. Hence,
	\begin{align*}
		\mathbf{F}_{(P,\gamma,\sigma)}=\sum_{f\in\mathcal{A}(P,\gamma)}\prod_{i=1}^{n}\mathbf{x}_{f(\sigma^{-1}(i))}=\sum_{w\in\mathcal{L}(P)}\left(\sum_{f\in\mathcal{A}(w,\gamma)}\prod_{i=1}^n\mathbf{x}_{f(\sigma^{-1}(i))}\right)=\sum_{w\in\mathcal{L}(P)}\mathbf{F}_{(w,\gamma,\sigma)}.
	\end{align*}
\end{proof}

\begin{Proposition}\label{prop:fund-eq}
	We have that $\mathbf{F}_{(w,\gamma,\sigma)}=\mathbf{F}_{(w,\gamma',\sigma)}$ if and only if $\Des(w,\gamma)=\Des(w,\gamma')$.
\end{Proposition}

\begin{proof}
	If $\Des(w,\gamma)=\Des(w,\gamma')$, by definition, $f\in\mathcal{A}(w,\gamma)$ implies $f\in\mathcal{A}(w,\gamma')$. Therefore, $\mathbf{F}_{(w,\gamma,\sigma)}=\mathbf{F}_{(w,\gamma',\sigma)}$.
	
	Conversely, if $\bfF_{(w,\gamma,\sigma)}=\bfF_{(w,\gamma',\sigma)}$, then commuting the variables we have $F_{(w,\gamma)}=F_{(w,\gamma)}$. Therefore, $\Des(w,\gamma)=\Des(w,\gamma')$.
\end{proof}

Given $A\subseteq[n-1]$ and $\sigma\in\mathfrak{S}_{n}$, define $$\mathbf{F}_{(A,\sigma)}=\mathbf{F}_{(w,\gamma,\sigma)}$$ where $(w,\gamma)$ is any labelled chain with $\Des(w,\gamma)=A$.

\begin{Example}\label{ex:fund-eq}
Let $w=(w_1,w_2,w_3,w_4, w_5)$ be a chain, 
$$\gamma(w_1)\gamma(w_2)\gamma(w_3)\gamma(w_4)\gamma(w_5)=34152,$$ and 
$$\sigma(w_1)\sigma(w_2)\sigma(w_3)\sigma(w_4)\sigma(w_5)=51423,$$
 then 
$$\Des(w,\gamma)=\{2,4\},$$ and so 
$$\bfF_{(\{2,4\},\sigma)}=\bfF_{(w,\gamma,\sigma)}.$$ The representation of $\bfF_{(\{2,4\},\sigma)}$ in terms of generalized chromatic functions in noncommuting variables is shown below. 
$$\scrX_
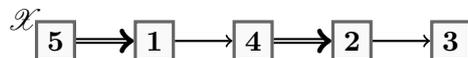
{
  \begin{tikzpicture}[
roundnode/.style={circle, draw=black!60, fill=gray!10, very thick, minimum size=7mm},
squarednode/.style={rectangle, draw=black!60, fill=gray!5, very thick, minimum size=5mm},]
 \node[squarednode](1) at (1.3,0){$\bf 5$};
  \node[squarednode](2) at (2.6,0){$\bfone$};
    \node[squarednode](3) at (3.9,0){$\bffour$};
      \node[squarednode](4) at (5.2,0){$\bftwo$};
        \node[squarednode](5) at (6.5,0){$\bfthree$};
        
    \draw[thick,double,->](1)--(2);
       \draw[ thick,->](2)--(3);
          \draw[ thick,double,->](3)--(4);
             \draw[ thick,->](4)--(5);

 \end{tikzpicture}
}$$
The set $\{2,4\}$ clarifies the types of edges and the bijection $\sigma$ provides the labels of the vertices.

\end{Example}

\subsection{Fundamental functions to monomial basis}\label{subsec:fund-fun-to-M}

\begin{Proposition}\label{prop:FtoM} Let $A\subseteq [n-1]$ and $\sigma\in \mathfrak{S}_n$. Then 
	$$ \mathbf{F}_{(A,\sigma)}=\sum_{A\subseteq C\subseteq[n-1]}\mathbf{M}_{(C,\sigma)}.$$
\end{Proposition}

\begin{proof}
	By definition, $\mathbf{F}_{(A,\sigma)}=\mathbf{F}_{(w,\gamma,\sigma)}$ where $w=(1,2,\dots,n)$ and $\Des(w,\gamma)=A$. Note the set of $(w,\gamma)$-partitions, $\mathcal{A}(w,\gamma)$, can be decomposed into
	$$\mathcal{A}(w,\gamma)=\bigsqcup_{A\subseteq C\subseteq[n-1]}\mathcal{A}_C(w,\gamma),$$
	where $\mathcal{A}_C(w,\gamma)$ is the set of $(w,\gamma)$-partitions $f$ such that $f(w_i)<f(w_{i+1})$ if and only if $i\in C$.
	Therefore,
	$$\mathbf{F}_{(A,\sigma)}=\sum_{f\in\mathcal{A}(w,\gamma)}\prod_{i=1}^{n}\mathbf{x}_{f(\sigma^{-1}(i))}=\sum_{A\subseteq C\subseteq[n-1]}\left(\sum_{f\in\mathcal{A}_C(w,\gamma)}\prod_{i=1}^{n}\mathbf{x}_{f(\sigma^{-1}(i))}\right).$$
	
	For each $A\subseteq C\subseteq[n-1]$, we have $$\sum_{f\in\mathcal{A}_C(w,\gamma)}\prod_{i=1}^{n}\mathbf{x}_{f(\sigma^{-1}(i))}=\mathbf{M}_{(C,\sigma)}.$$
	Therefore,
	$$\mathbf{F}_{(A,\sigma)}=\sum_{A\subseteq C\subseteq[n-1]}\left(\sum_{f\in\mathcal{A}_C(w,\gamma)}\prod_{i=1}^{n}\mathbf{x}_{f(\sigma^{-1}(i))}\right)=\sum_{A\subseteq C\subseteq[n-1]}\mathbf{M}_{(C,\sigma)}.$$
\end{proof}

\begin{Corollary}\label{cor:Fbasis}
	The set $\{\mathbf{F}_{(A,\sigma)}:(A,\sigma)\text{ is standard}\}$ forms a basis of $\NCQSym$.
\end{Corollary}

\begin{proof} 
By Proposition  \ref{prop:std} we have 
$$ \mathbf{F}_{(A,\sigma)}=\sum_{A\subseteq C\subseteq[n-1]}\mathbf{M}_{(C,\sigma)}.$$
Note that $A\subseteq C$ if and only if $\SetComp(A,\sigma)\leq\SetComp(C,\sigma)$. The result follows from triangularity and Proposition \ref{prop:std}.
\end{proof}

We call $\{\mathbf{F}_{(A,\sigma)}:(A,\sigma)\text{ is standard}\}$ the \emph{fundamental basis}  of $\NCQSym$.

\begin{Remark}\label{rem:fund}
	The fundamental basis also appears in \cite[Section 11]{ALvW23}.
\end{Remark}

\section{Enriched fundamental functions in noncommuting variables and the analogue of the fundamental theorem}\label{sec:peak-en}

\subsection{Enriched fundamental functions in noncommuting variables}\label{subsec:peak-n}

In this section, we present the enriched fundamental functions in noncommuting variables.

Given $(P,\gamma,\sigma)$ where $(P,\gamma)$ is a labelled poset with ground set $X$, $|X|=n$, and $\sigma:X\to[n]$ is a bijection, define the \emph{enriched fundamental function  in noncommuting variables} $\mathbf{K}_{(P,\gamma,\sigma)}$ to be	$$\mathbf{K}_{(P,\gamma,\sigma)}=\sum_{f\in\mathcal{B}(P,\gamma)}\prod_{i=1}^{n}\mathbf{x}_{|f(\sigma^{-1}(i))|}= \sum_{f\in\mathcal{B}(P,\gamma)}\mathbf{x}_{|f(\sigma^{-1}(1))|}\mathbf{x}_{|f(\sigma^{-1}(2))|}\dots \mathbf{x}_{|f(\sigma^{-1}(n))|}.$$


\begin{Example}\label{ex:peak-fun-n}
Given $(P,\gamma, \sigma)$ in  Example \ref{ex:labelfund}, 
\begin{align*}
		\begin{array}{ccccccccc}
			& \begin{tikzpicture}[scale=0.3]
				\draw[black, thick] (0.7,0.7) -- (1.3,1.3);
				\draw[black, thick] (-0.7,0.7) -- (-1.3,1.3);
				\draw[black, thick] (-2,2.8) -- (-2,4.2);
				\node at (-2,5) {$1$};
				\node at (-2,2) {$-1$};
				\node at (0,0) {$-1$};
				\node at (2,2) {$1$};
			\end{tikzpicture} &
			& \begin{tikzpicture}[scale=0.3]
				\draw[black, thick] (0.7,0.7) -- (1.3,1.3);
				\draw[black, thick] (-0.7,0.7) -- (-1.3,1.3);
				\draw[black, thick] (-2,2.8) -- (-2,4.2);
				\node at (-2,5) {$1$};
				\node at (-2,2) {$1$};
				\node at (0,0) {$-1$};
				\node at (2,2) {$1$};
			\end{tikzpicture} &
			& \begin{tikzpicture}[scale=0.3]
				\draw[black, thick] (0.7,0.7) -- (1.3,1.3);
				\draw[black, thick] (-0.7,0.7) -- (-1.3,1.3);
				\draw[black, thick] (-2,2.8) -- (-2,4.2);
				\node at (-2,5) {$2$};
				\node at (-2,2) {$-2$};
				\node at (0,0) {$1$};
				\node at (2,2) {$1$};
			\end{tikzpicture} & 
			& \begin{tikzpicture}[scale=0.3]
				\draw[black, thick] (0.7,0.7) -- (1.3,1.3);
				\draw[black, thick] (-0.7,0.7) -- (-1.3,1.3);
				\draw[black, thick] (-2,2.8) -- (-2,4.2);
				\node at (-2,5) {$2$};
				\node at (-2,2) {$-2$};
				\node at (0,0) {$1$};
				\node at (2,2) {$-2$};
			\end{tikzpicture}\\
			{\bf K}_{(P,\gamma,\sigma)}=& \bx_1\bx_1\bx_1\bx_1 & + & \bx_1\bx_1\bx_1\bx_1 & + & \bx_1\bx_2\bx_1\bx_2 & + & \bx_1\bx_2\bx_2\bx_2 &+\cdots.
		\end{array}
	\end{align*}
\end{Example} 

\begin{Lemma}  Given the labelled edge-coloured digraph ${\bf G}$ associated to $(P,\gamma,\sigma)$, we have 
	$$\bfK_{(P,\gamma,\sigma)}=\scrF_{\bf G}.$$
	\end{Lemma} 
	

\subsection{The analogue of the fundamental theorem}\label{subsec:fun-thm-en}

\begin{Proposition}\label{prop:fund2} We have that
	$$\mathbf{K}_{(P,\gamma,\sigma)}=\sum_{w\in\mathcal{L}(P)}\mathbf{K}_{(w,\gamma,\sigma)}.$$
\end{Proposition}

\begin{proof}
	By Theorem \ref{thm:eppar}, we have $\mathcal{B}(P,\gamma)=\bigsqcup_{w\in\mathcal{L}(P)}\mathcal{B}(w,\gamma)$. Hence,
	\begin{align*}
	\mathbf{K}_{(P,\gamma,\sigma)}=\sum_{f\in\mathcal{B}(P,\gamma)}\prod_{i=1}^{n}\mathbf{x}_{|f(\sigma^{-1}(i))|}=\sum_{w\in\mathcal{L}(P)}\left(\sum_{f\in\mathcal{B}(w,\gamma)}\prod_{i=1}^n\mathbf{x}_{|f(\sigma^{-1}(i))|}\right)=\sum_{w\in\mathcal{L}(P)}\mathbf{K}_{(w,\gamma,\sigma)}.
	\end{align*}
\end{proof}

\subsection{Enriched fundamental functions to monomial basis}\label{subsec:peak-to-M}

First, we expand the $\mathbf{K}$ functions in the $\mathbf{M}$ basis, showing that $\NCPeak$ is a subspace of $\NCQSym$. The idea of the proof is similar to the commutative case given in \cite{S97}.

\begin{Proposition}\label{prop:KtoM}
	Assume $w=(w_1,w_2,\dots,w_n)$ is a chain and $\sigma\in\mathfrak{S}_n$, we have
	$$\mathbf{K}_{(w,\gamma,\sigma)}=\sum_{A\subseteq[n-1] \atop \peak(w,\gamma)\subseteq A\cup (A+1)} 2^{|A|+1}\mathbf{M}_{(A,\sigma)}.$$
	\end{Proposition}

\begin{proof}
Recall that 
$$\bfK_{(w,\gamma,\sigma)}=\sum_{f\in\mathcal{B}(w,\gamma)}\prod_{i=1}^n\mathbf{x}_{|f(\sigma^{-1}(i))|}.$$

Fix $A=\{a_1<a_2<\dots<a_{k}\}\subseteq[n-1]$ with $0=a_0$ and $a_{k+1}=n$. Define 
$$[A]=1 2 \dots  a_1 ~|~ (a_1+1) (a_1+2) \dots a_2~| \cdots |(a_k+1)(a_{k}+2)  \cdots  a_n.$$ 

 Let  $$\mathcal{B}_A(w,\gamma)$$ be the set of all $f\in \mathcal{B}(w,\gamma)$ such that
\begin{enumerate}
\item $|f(w_i)|=|f(w_j)|$ if $i$ and $j$ are in the same block of $[A]$. 
\item $|f(w_i)|<|f(w_j)|$ if $i$ is in a block of $[A]$ to the left of the block containing $j$.
\end{enumerate} 
Then 
$$\bfK_{(w,\gamma,\sigma)}=\sum_{f\in\mathcal{B}(w,\gamma)}\prod_{i=1}^n\mathbf{x}_{|f(\sigma^{-1}(i))|}=\sum_{A\subseteq [n-1]} \left(\sum_{f\in\mathcal{B}_A(w,\gamma)}\prod_{i=1}^n\mathbf{x}_{|f(\sigma^{-1}(i))|}\right).$$
We will show the following claims are true. 
\begin{itemize}
\item[(I)] If  $\peak(w,\gamma)\not\subseteq A\cup (A+1)$, then  $\calB_A(w,\gamma)$ is empty.

\item[(II)] If  $\peak(w,\gamma)\subseteq A\cup (A+1)$, then 
$$\sum_{f\in\mathcal{B}_A(w,\gamma)}\prod_{i=1}^n\mathbf{x}_{|f(\sigma^{-1}(i))|}=2^{|A|+1}\bfM_{(A,\sigma)}.$$
\end{itemize} 
Therefore,  
$$\bfK_{(w,\gamma,\sigma)}=\sum_{A\subseteq [n-1]} \left(\sum_{f\in\mathcal{B}_A(w,\gamma)}\prod_{i=1}^n\mathbf{x}_{|f(\sigma^{-1}(i))|}\right)=\sum_{A\subseteq[n-1] \atop \peak(w,\gamma)\subseteq A\cup (A+1)} 2^{|A|+1}\mathbf{M}_{(A,\sigma)},$$ as desired. 
\\

{\it Proof of }(I). Let $b\in \peak(w,\gamma)\setminus A\cup (A+1)$. Since 
$$ \peak(w,\gamma)\setminus A\cup (A+1)\subseteq \{2,3,\dots,n-1\}\setminus A\cup (A+1)=$$$$\{2,3, \dots, a_{1}-1,a_{1}+2,a_{1}+3,\dots,a_{2}-1,\dots, a_{k}+2,a_{k}+3,\dots,n\},$$
we have that $b-1,b,b+1$ are in the same block of $[A]$. Therefore, if $f\in \calB_A(w,\gamma)\neq\emptyset$, we must have $$|f(w_{b-1})|=|f(w_b)|=|f(w_{b+1})|.$$ 

Note that $b$ is in $\peak(w,\gamma)$, and so $$\gamma(w_{b-1})<\gamma(w_b)>\gamma(w_{b+1}).$$ Since $f$ is an enriched $(P,\gamma)$-partition,   
$$f(w_{b-1})\preceq f(w_b)\preceq f(w_{b+1}).$$

If  $|f(w_{b-1})|=|f(w_b)|$ then $f(w_{b-1})=f(b)>0$, and if $|f(w_{b})|=|f(w_{b+1})|$ then $f(w_{b-1})=f(b)<0$. Therefore, we can not have $|f(w_{b-1})|=|f(w_b)|=|f(w_{b+1})|,$ a contradiction. 
Therefore, $ \calB_A(w,\gamma)=\emptyset$.\\

{\it Proof of }(II). If $\peak(w,\gamma)\subseteq A\cup (A+1)$, then for each $1\leq j\leq \ell-1$, we have
 $$\gamma(w_{a_{j}+1})\leq \gamma(w_{a_{j}+2})\leq \dots\leq \gamma(w_{a_{j}+t})\geq \dots \geq\gamma(w_{a_{j+1}-1})\geq\gamma(w_{a_{j+1}})$$ (note that we may have $t=1$ or $t=(a_{j+1}-a_j)$, which in these cases $\gamma$ is decreasing or increasing, respectively).
 Then there are two types of $f\in \calB_A(P,\gamma)$: 
 \begin{itemize}
 \item $f(w_{a_j+t})<0$, $f(w_{a_j+t})=f(w_{a_j+s})$ for $1\leq s\leq t$, and $f(w_{a_j+t})=-f(w_{a_j+s})$ for $t+1\leq s\leq a_{j+1}-a_j$.
  \item $f(w_{a_j+t})>0$, $f(w_{a_j+t})=f(w_{a_j+s})$ for $1\leq s\leq t$, and $f(w_{a_j+t})=-f(w_{a_j+s})$ for $t+1\leq s\leq a_{j+1}-a_j$.
 \end{itemize} 
 Therefore, there are $2^{|A|+1}$ possible choices of $f\in \calB_{A}(P,\gamma)$ such that they produce the monomial 
 $$\bx_{|f(\sigma^{-1}(1))|}\bx_{|f(\sigma^{-1}(2))|}\dots \bx_{|f(\sigma^{-1}(n))|}$$ such that 
 \begin{enumerate}
\item $|f(\sigma^{-1}(i))|=|f(\sigma^{-1}(j))|$ if $i$ and $j$ are in the same block of $[A]$. 
\item $|f(\sigma^{-1}(i))|<|f(\sigma^{-1}(j))|$ if $i$ is in a block of $[A]$ to the left of the block containing $j$.
\end{enumerate} 
Therefore, $$\sum_{f\in\mathcal{B}_A(w,\gamma)}\prod_{i=1}^n\mathbf{x}_{|f(\sigma^{-1}(i))|}=2^{|A|+1}\bfM_{(A,\sigma)},$$ as desired.
\end{proof}

\begin{Corollary}\label{Coro:PeakSigma} We have that 
$\mathbf{K}_{(w,\gamma,\sigma)}=\mathbf{K}_{(w,\gamma',\sigma)}$ if and only if $\peak(w,\gamma)=\peak(w,\gamma')$.
	
	\end{Corollary} 

%
%
%
%
%
%

Given a peak set $B\subseteq\{2,3,\dots,n-1\}$ and $\sigma\in\mathfrak{S}_n$, define  $$\mathbf{K}_{(B,\sigma)}=\mathbf{K}_{(w,\gamma,\sigma)}$$ where $(w,\gamma)$ is any labelled chain with $\peak(w,\gamma)=B$.\\

\section{Peak algebra in noncommuting variables and its dimension}\label{sec:peak-n}

\begin{Definition}\label{Def:NCPeak}
	Let $\NCPeak$ be the space spanned by all enriched fundamental functions in noncommuting variables $\bfK_{(P,\gamma,\sigma)}$, or equivalently, the space spanned by the following set 
	$$\{\scrF_{\bf G}: {\bf G} \text{~is a labelled edge-coloured digraph}\}.$$
	The space $\NCPeak$ is called the \emph{peak algebra in noncommuting variables}.
	\end{Definition}

\subsection{Enriched monomial functions}\label{subsec:monpeak-n}

Given $(B,\sigma)$ where $B\subseteq\{2,3,\dots,n-1\}$ is a peak set and $\sigma\in\mathfrak{S}_n$, define the \emph{enriched monomial function in noncommuting variables} $\boldeta_{(B,\sigma)}$ to be
$$\boldeta_{(B,\sigma)}=\sum_{C\subseteq B}(-1)^{|B|-|C|}\mathbf{K}_{(C,\sigma)}.$$

By inclusion-exclusion, we have
$$\mathbf{K}_{(B,\sigma)}=\sum_{C\subseteq B}\boldeta_{(C,\sigma)}.$$


\begin{Corollary}\label{cor:ncpeak}
 Both sets $$\{\bfK_{(B,\sigma)}: \text{~$B\subseteq \{2,3,\dots,n-1\}$ is a peak set, $\sigma\in \frakS_n$}\}$$ and 
	$$\{\boldeta_{(B,\sigma)}: \text{~$B\subseteq \{2,3,\dots,n-1\}$ is a peak set, $\sigma\in \frakS_n$}\}$$ 
	span $\NCPeak_n$.

\end{Corollary}

\begin{proof} By definition 
$\NCPeak$ is spanned by $\{\bfK_{(P,\gamma,\sigma)}\}$, and
 by Corollary \ref{Coro:PeakSigma}, if $\Peak(P,\gamma)=\Peak(P,\gamma')$, then $\bfK_{(P,\gamma,\sigma)}=\bfK_{(P,\gamma',\sigma)}$. Therefore, the set 
 $$\{\bfK_{(B,\sigma)}: \text{~$B\subseteq \{2,3,\dots,n-1\}$ is a peak set, $\sigma\in \frakS_n$}\}$$ spans $\NCPeak_n$, and by triangulation we have $$\{\boldeta_{(B,\sigma)}: \text{~$B\subseteq \{2,3,\dots,n-1\}$ is a peak set, $\sigma\in \frakS_n$}\}$$
	spans $\NCPeak_n$.
 
\end{proof}

\subsection{Enriched monomial functions to monomials}\label{subsec:mpeak-to-M} 
In order to show that the dimension of $\NCPeak_n$ is the same as the number of odd set compositions of $[n]$, we need the following lemma.

\begin{Lemma}\label{lem:N}
	For any peak set $B\subseteq[n-1]$ and $\sigma\in\mathfrak{S}_n$, we have
	$$\boldeta_{(B,\sigma)}=(-1)^{|B|}\sum_{A\subseteq \odd(B)}2^{|A|+1}\mathbf{M}_{(A,\sigma)}.$$
\end{Lemma}

\begin{proof}
	It follows from Proposition \ref{prop:KtoM} that
	\begin{align*}
		\boldeta_{(B,\sigma)}&=\sum_{C\subseteq B}(-1)^{|B|-|C|}\sum_{A\subseteq[n-1] \atop C\subseteq A\cup(A+1)}2^{|A|+1}\mathbf{M}_{(A,\sigma)}\\
		&=\sum_{A\subseteq[n-1]}\left(\sum_{C\subseteq (A\cup(A+1))\cap B}(-1)^{|B|-|C|}\right)2^{|A|+1}\mathbf{M}_{(A,\sigma)}\\
		&=(-1)^{|B|}\sum_{A\subseteq [n-1]\setminus(B\cup (B-1))}2^{|A|+1}\mathbf{M}_{(A,\sigma)}.
	\end{align*}
	The last equality follows from the observation that the sum inside the bracket vanishes unless $(A\cup(A+1))\cap B=\emptyset$ which is equivalent to $A\subseteq[n-1]\setminus(B\cup(B-1))$. Note that $\odd(B)=[n-1]\setminus (B\cup (B-1))$, thus 
	$$\boldeta_{(B,\sigma)}=(-1)^{|B|}\sum_{A\subseteq \odd(B)}2^{|A|+1}\mathbf{M}_{(A,\sigma)}.$$
\end{proof}

\begin{Lemma}\label{lem:eq-eta}
Let $B\subseteq[n-1]$ be a peak set and $\sigma\in \mathfrak{S}_n$. Let $\sigma'\in \mathfrak{S}_n$. If   
$$\std(\odd(B),\sigma)=\std(\odd(B),\sigma'),$$
 then 
$$\boldeta_{(B,\sigma)}=\boldeta_{(B,\sigma')}.$$
\end{Lemma}

\begin{proof}
Note that $$\bfM_{(\odd(B),\sigma)}=\bfM_{(\odd(B),\sigma')}.$$ Thus, for any $A\subseteq \odd(B)$, 
$$\bfM_{(A,\sigma)}=\bfM_{(A,\sigma')}.$$
Therefore,
$$\boldeta_{(B,\sigma)}=(-1)^{|B|}\sum_{A\subseteq \odd(B)}2^{|A|+1}\mathbf{M}_{(A,\sigma)}=(-1)^{|B|}\sum_{A\subseteq \odd(B)}2^{|A|+1}\mathbf{M}_{(A,\sigma')}=\boldeta_{(B,\sigma')}.$$

\end{proof} 

\subsection{Bases of the peak algebra in noncommuting variables and its dimension}\label{subsec:bases-peak}

Recall that if $B\subseteq [n-1]$ is a peak set and $\sigma\in \mathfrak{S}_n$, we say $(B,\sigma)$ is  enriched standard if $(\odd(B),\sigma)$ is standard. 

Our main result of this section is the following. 

\begin{Theorem}\label{Thm:NCPeakBasis}
The sets 
$$\{\boldeta_{(B,\sigma)}:B\subseteq \{2,3,\dots,n-1\}, \sigma\in\mathfrak{S}_n, (B,\sigma) \text{~is  enriched standard}   \}$$ and 
 $$\{\bfK_{(B,\sigma)}:B\subseteq \{2,3,\dots,n-1\}, \sigma\in\mathfrak{S}_n, (B,\sigma) \text{~is  enriched standard}   \}$$
  are bases for $\NCPeak_n.$ Moreover, 
$$\dim (\NCPeak_n)=|\{ \phi\vDash [n]: \phi \text{~is odd}\}|.$$
\end{Theorem} 

\begin{proof}
Given any pair $(B,\sigma)$ where $B\subseteq \{2,3,\dots,n-1\}$ is a peak set and $\sigma\in \frakS_n$, let $\sigma'$ be the unique permutation in $\mathfrak{S}_n$ such that 
$$\std(\odd(B),\sigma)=(\odd(B),\sigma').$$
By Lemma \ref{lem:eq-eta}, 
$$\boldeta_{(B,\sigma)}=\boldeta_{(B,\sigma')}.$$
Moreover,  $$\{\boldeta_{(B,\sigma)}: B\subseteq \{2,3,\dots,n-1\}, \sigma\in\mathfrak{S}_n, (B,\sigma) \text{~is enriched standard}  \}$$ is linearly independent by triangularity of the transition matrix from $\boldeta$ to $\bfM$. Therefore, 
this set is a basis for $\NCPeak_n$. 

Observe that if $(B,\sigma)$ is enriched standard, then for any $C\subseteq B$, $(C,\sigma)$ is also enriched standard. Therefore, $$\{\bfK_{(B,\sigma)}:B\subseteq \{2,3,\dots,n-1\}, \sigma\in\mathfrak{S}_n, (B,\sigma) \text{~is  enriched standard}   \}$$ also forms a basis for $\NCPeak_n$.
\end{proof}

\section{Hopf structure of the peak algebra in noncommuting variables}\label{Sec:NCPeak}

The goal of this section is to prove that $\NCPeak$ is a Hopf subalgebra of $\NCQSym$.

Recall that the product in $\NCQSym$ is the regular product of power series. The coproduct in $\NCQSym$ is defined as follows (see \cite{B09}). By introducing another set of noncommuting variables $\mathbf{y}=\mathbf{y}_1,\mathbf{y}_2,\dots$ with linear order $\mathbf{x}_1<\mathbf{x}_2<\cdots<\mathbf{y}_1<\mathbf{y}_2<\cdots$, we have the following map composition
$$\begin{array}{ccccccc}
	\NCQSym(\mathbf{x}) & \cong & \NCQSym(\mathbf{x},\mathbf{y}) & \to & \NCQSym(\mathbf{x},\mathbf{y})/\sim & \to & \NCQSym(\mathbf{x})\otimes\NCQSym(\mathbf{y})\\
	f(\mathbf{x}) & \mapsto & f(\mathbf{x},\mathbf{y}) & \mapsto & \bar{f}(\mathbf{x},\mathbf{y}) & \mapsto & \sum \bar{f_1}(\mathbf{x})\otimes \bar{f_2}(\mathbf{y})
\end{array}$$
where $\sim$ denotes the relations $\mathbf{x}_i\mathbf{y}_j=\mathbf{y}_j\mathbf{x}_i$ for all $i,j$. Then, the coproduct  is given by
$$\begin{array}{cccccc}
	\Delta: & \NCQSym(\mathbf{x}) & \to & \NCQSym(\mathbf{x})\otimes\NCQSym(\mathbf{y}) & \cong & \NCQSym(\mathbf{x})\otimes\NCQSym(\mathbf{x})\\
	& f(\mathbf{x}) & \mapsto & \sum \bar{f_1}(\mathbf{x})\otimes \bar{f_2}(\mathbf{y}) & \mapsto & \sum \bar{f_1}(\mathbf{x})\otimes \bar{f_2}(\mathbf{x})
\end{array}$$

Let $\bf G$ be a labelled edge-coloured digraph, and let $n$ be a positive integer. Then ${\bf G}+n$ is the labelled edge-coloured digraph obtained by replacing each vertex $a$ of $\bf G$ by $a+n$. 
%

Given labelled edge-coloured digraphs ${\bf G}$ and ${\bf H}$, by \cite[Proposition 10.1]{ALvW23} 
 $$\scrY_{\bf G} \scrY_{\bf H}=\scrY_{{\bf G}| {\bf H}}$$ where ${\bf G}|{\bf H}$ is the digraph of disjoint union of ${\bf G}$ and  ${\bf H}+|V({\bf G})|$. 
  $$
  \begin{tikzpicture}[
roundnode/.style={circle, draw=black!60, fill=gray!10, very thick, minimum size=7mm},
squarednode/.style={rectangle, draw=black!60, fill=gray!5, very thick, minimum size=5mm},]
 \node[squarednode](1) at (1.3,0){$\bf 1$};
  \node[squarednode](2) at (2.6,0){$\bf 2$};
    \node[squarednode](3) at (3.9,0){$\bf 3$};
    \node at (4.6,-0.7){$Q_{\{1,2,3\}}| Q_{\{1,2,3\}}$};
    \draw[thick,double,->](1)--(2);
       \draw[ thick,double,->](2)--(3);
 \node[squarednode](a) at (4.8,0){${\bf 4}$};
  \node[squarednode](b) at (6.1,0){$\bf 5$};
    \node[squarednode](c) at (7.4,0){$\bf 6$};
    \draw[thick,double,->](a)--(b);
       \draw[ thick,double,->](b)--(c);
          \draw[ thick,double,->](b)--(c);
 \end{tikzpicture} 
 $$
Let $\bf G$ be a labelled edge-coloured digraph. The \emph{standardization} of $\bf G$, denoted $\std({\bf G})$, is the labelled edge-coloured digraph obtained by replacing the $i$th smallest element of $V({\bf G})$ by $i$.
For a subset $A$ of $V({\bf G})$, let ${\bf G}|_A$ be the induced labelled edge-coloured subdigraph of ${\bf G}$ with vertex set $A$.  
   An edge-coloured subdigraph ${\bf F}$  of ${\bf G}$ is called a \emph{$\{\rightarrow,\Rightarrow\}$-induced subdigraph} of ${\bf G}$ 
  if ${\bf F}$ is an induced subdigraph of ${\bf G}$ and if $a\in V({\bf F})$ and either $a\rightarrow b$ or $a\Rightarrow b$ in ${\bf G}$, then $b\in V({\bf F})$.  Then by \cite[Proposition 10.2]{ALvW23} 
$$\Delta(\scrY_{\bf G})=\sum_{{\bf F}} \scrY_{\std({\bf G}|_{V({\bf G})\setminus V({\bf F})})} \otimes \scrY_{\std({\bf F})}$$ where the sum is over all $\{\rightarrow, \Rightarrow\}$-induced subdigraphs ${\bf F}$ of ${\bf G}$.\\

With a similar proof for the product and coproduct of enriched generalized chromatic functions. We have 
$$\scrF_{\bf G} \scrF_{\bf H}=\scrF_{{\bf G} | {\bf H}},$$
and
$$\Delta(\scrF_{\bf G})=\sum_{{\bf F}} \scrF_{\std({\bf G}|_{V({\bf G})\setminus V({\bf F})})} \otimes \scrF_{\std({\bf F})}$$ where the sum is over all $\{\rightarrow, \Rightarrow\}$-induced subdigraphs ${\bf F}$ of ${\bf G}$.

Therefore $\NCPeak$ is isomorphic to a Hopf subalgebra of $\NCQSym$.

\section{The labelled descent-to-peak map}\label{sec:thetamap-n}

	The \emph{labelled descent-to-peak map} from $\NCQSym$ to $\NCPeak$, denoted $\Theta_{\NCQSym}$, is the linear map
	\begin{equation}\label{def:theta-non}
	\begin{array}{cccc}
		\Theta_\NCQSym: & \NCQSym & \to & \NCPeak\\
		& \mathbf{F}_{(A,\sigma)} & \mapsto & \mathbf{K}_{(A\setminus((A+1)\cup\{1\}),\sigma)}
	\end{array}
	\end{equation}
	for all standard pairs $(A,\sigma)$.

Let $A=\{a_1<a_2<\dots<a_k\}\subseteq[n-1]$ with $n-a_k$ is odd. We define $$\oddd(A)=\{a_i:1\leq i\leq k, a_i-a_{i-1} \text{ is odd}\}$$ where $a_0=0$. For example, if $A=\{1,3,4,8\}\subseteq[9]$, then $\oddd(A)=\{1,4\}$. 

Remark that $\oddd(A)$ is an odd set, and there is a unique peak set $B$ such that $$\odd(B)=\oddd(A).$$

The following lemma will be useful in the study of the labelled descent-to-peak map.
\begin{Lemma}\label{lem:M}
	Let $A\subseteq[n-1]$ and $\sigma\in\mathfrak{S}_n$. Then
	$$\sum_{A\subseteq C\subseteq[n-1]}(-1)^{|C|}\mathbf{K}_{(C\setminus((C+1)\cup\{1\}),\sigma)}=\begin{cases}
		(-1)^{n-1-|B|}\boldeta_{(B,\sigma)} & \text{ if }n-\max (A)\text{ is odd,}\\
		0 & \text{ otherwise,}
	\end{cases}$$
	where $B$ is the unique peak set such that $\odd(B)=\oddd(A)$.
	
	\end{Lemma}

\begin{proof}

Consider the set $\{C\subseteq [n-1]: A\subseteq C\}$. Let the sign of $C$ to be $(-1)^{|C|}$. 

We first show that  if $n-\max(A)$ is even, then
$$\sum_{A\subseteq C\subseteq[n-1]}(-1)^{|C|}\mathbf{K}_{(C\setminus((C+1)\cup\{1\}),\sigma)}=0$$ using a sign reversing involution $\imath$ with no fixed point defined as follows. 

Let $C_{\geq \max(A)}$ be the set of the elements of $C$ that are bigger than or equal to $\max(A)$. Let 
$$(C_{\geq \max(A)}\setminus \{1\})\cup \{\max(A),n+1\}=\{\max(A)=d_1<d_2<\cdots<d_k=n+1\}.$$ 

The set $\{1\leq i\leq k-1: d_{i+1}-d_{i} \text{~is odd} \}$ is non-empty since $d_{k}-d_1=n+1-\max(A)$ is odd. Let $j=\min(\{1\leq i\leq k-1: d_{i+1}-d_{i} \text{~is odd}\})$; that is $j$  is the smallest integer such that $d_{j+1}-d_j$ is odd. If $ C\cap \{d_j+1, d_j+2,\dots, d_{j+1}\}=\emptyset$, set $m_C=0$, otherwise, let $m_C$ be the largest integer such that 
$$d_{j}+1,d_j+2,\ldots,d_j+m_C \in C\cap \{d_j+1, d_j+2,\dots, d_{j+1}\}.$$
Now define 
$$
\begin{array}{cccc}
\imath: & \{ C\subseteq [n-1]: A\subseteq C\}&\rightarrow &\{ C\subseteq [n-1]: A\subseteq C\}\\
& C& \mapsto & \begin{cases} 
C\setminus \{d_j+m_C\} & \text{if ~} m_C \text{~is odd,}\\
C\cup \{d_j+m_C+1\} & \text{if ~} m_C \text{~is even.}
\end{cases} 
\end{array}
$$
Then $\imath$ is a sign reversing involution with no fixed point. Moreover, for any $C$, we have
$$\imath(C)\setminus((\imath(C)+1)\cup\{1\})=C\setminus((C+1)\cup\{1\}),$$ and so 
$$\mathbf{K}_{(C\setminus((C+1)\cup\{1\}),\sigma)}=\mathbf{K}_{\imath(C)\setminus((\imath(C)+1)\cup\{1\}),\sigma)}.$$
Consequently, 
$$\sum_{A\subseteq C\subseteq[n-1]}(-1)^{|C|}\mathbf{K}_{(C\setminus((C+1)\cup\{1\}),\sigma)}=0.$$

Let $n-\max(A)$ be odd and let $B$ be the unique peak set such that $\odd(B)=\oddd(A)$. We want to show that 
$$\sum_{A\subseteq C\subseteq[n-1]}(-1)^{|C|}\mathbf{K}_{(C\setminus((C+1)\cup\{1\}),\sigma)}=(-1)^{n-1-|B|}\boldeta_{(B,\sigma)}$$ using a sing reversing involution $\jmath$ defined as follows. \\

Define 
$$
\begin{array}{cccc}
\jmath:& \{C\subseteq [n-1]: A\subseteq C\} & \rightarrow& \{C\subseteq [n-1]: A\subseteq C\}\\
& C & \mapsto & 
\begin{cases} 
C & \text{if~}B\subseteq C,\\
 C\setminus\{\ (\min(B)\setminus C)-1 \} & \text{if~}  B\not\subseteq C, (\min(B)\setminus C)-1\in C,\\
 C\cup\{\ (\min(B)\setminus C)-1 \}  & \text{if~}  B\not\subseteq C, (\min(B)\setminus C)-1\not\in C.\\
\end{cases} 
\end{array} 
$$
Note that the set of fixed points of $\jmath$ is 
$$\{C\subseteq [n-1]: A,B\subseteq C\}.$$ 
We have that
$$\{C\setminus ( (C+1)\cup \{1\} ): A,B\subseteq C\subseteq [n-1]\}=\{ D: D\subseteq B \}.$$
The reason for this equality is that if $A,B\subseteq C\subseteq [n-1]$, then 
$$C=[n-1]\setminus \{d-1:d\in D\}$$ for some $D\subseteq B$.
Note that  $|D|=n-1-|C|$. 
Therefore, 
\begin{align*} 
\sum_{A\subseteq C\subseteq[n-1]}(-1)^{|C|}\mathbf{K}_{(C\setminus((C+1)\cup\{1\}),\sigma)}&=\sum_{A,B\subseteq C\subseteq[n-1]}(-1)^{|C|}\mathbf{K}_{(C\setminus((C+1)\cup\{1\}),\sigma)}\\
&=\sum_{D\subseteq B}(-1)^{n-1-|D|}\mathbf{K}_{(D,\sigma)}\\
&=(-1)^{n-1-|B|}\boldeta_{(B,\sigma)}.
\end{align*}

\end{proof}

\begin{Proposition}\label{prop:M} 
Let $A\subseteq [n-1]$ and $\sigma\in \mathfrak{S}_n$.  
	Let $(A,\sigma)$ be standard such that $n-\max(A)$ is odd. Then
	$$\Theta_{\NCQSym}(\mathbf{M}_{(A,\sigma)})=\begin{cases}
			(-1)^{n-1-|B|-|A|}\boldeta_{(B,\sigma)} & \text{ if }n-\max(A)\text{ is odd,}\\
			0 & \text{ otherwise,}
		\end{cases}
		$$
		where $B$ is the unique peak set such that $\odd(B)=\oddd(A)$.
\end{Proposition}

\begin{proof}
	By Proposition \ref{prop:FtoM} and inclusion-exclusion, we have
	$$\Theta_{\NCQSym}(\mathbf{M}_{(A,\sigma)})=\sum_{A\subseteq C}(-1)^{|C|-|A|}\Theta_{\NCQSym}(\mathbf{F}_{(C,\sigma)}).$$
	
	For any $A\subseteq C$, we must have $(C,\sigma)$ is standard. Hence by definition, $\Theta_{\NCQSym}(\mathbf{F}_{(C,\sigma)})=\mathbf{K}_{(C\setminus((C+1)\cup\{1\}),\sigma)}$. Let $B$ be the unique peak set such that $\odd(B)=\oddd(A)$. Then, by Lemma \ref{lem:M}, we have
	\begin{align*}
		\Theta_{\NCQSym}(\mathbf{M}_{(A,\sigma)})&=\sum_{A\subseteq C\subseteq[n-1]}(-1)^{|C|-|A|}\mathbf{K}_{(C\setminus((C+1)\cup\{1\}),\sigma)}\\
		&=\begin{cases}
			(-1)^{n-1-|B|-|A|}\boldeta_{(B,\sigma)} & \text{ if }n-\max(A)\text{ is odd,}\\
			0 & \text{ otherwise.}
		\end{cases}
	\end{align*}
	
\end{proof}

\begin{Remark}\label{rem:comm-m}
	A commutative version of Proposition \ref{prop:M} appears in \cite[Theorem 2.4]{H07}. So this result is an extension of \cite[Theorem 2.4]{H07}.
\end{Remark}

\begin{Theorem}\label{thm:theta} 
Let $(P,\gamma)$ be a labelled poset with ground set $X$ and $|X|=n$. Let $\sigma:X\to[n]$ be a bijection. 
	We have that:
	\begin{enumerate}
		\item $\Theta_{\NCQSym}(\mathbf{F}_{(A,\sigma)})=\mathbf{K}_{(A\setminus((A+1)\cup\{1\}),\sigma)}$ for all $A\subseteq[n-1]$, $\sigma\in\mathfrak{S}_n$ ($(A,\sigma)$ is not necessarily standard).
		\item $\Theta_{\NCQSym}(\mathbf{F}_{(P,\gamma,\sigma)})=\mathbf{K}_{(P,\gamma,\sigma)}$.
		\item The following diagram commutes.
		\begin{center}
			\begin{tikzpicture}
				\node(TV) at (0,0){$\NCQSym$};
				\node(Sym) at (4,0){$\QSym$};
				\node(OV) at (0,-2){$\NCPeak$};
				\node(Omega) at (4,-2){$\Pi$};
				
				\draw[thick,->]  (TV)->(Sym);
				\draw[thick,->]  (OV)->(Omega);
				\draw[thick,->]   (TV)->(OV);
				\draw[thick,->]   (Sym)->(Omega);
				
				\node(Phi) at (2,0.2){$\rho$};
				\node(Phi) at (2,-1.8){$\rho$};
				\node(thetaV) at (-0.9,-1){$\Theta_{\NCQSym}$};
				\node(thetaV) at (4.7,-1){$\Theta_\QSym$};
			\end{tikzpicture} 
		\end{center}
	\end{enumerate}
\end{Theorem}

\begin{proof} Recall that 
$\Theta_{\NCQSym}$ is the linear map
	$$\begin{array}{cccc}
		\Theta_\NCQSym: & \NCQSym & \to & \NCPeak\\
		& \mathbf{F}_{(A,\sigma)} & \mapsto & \mathbf{K}_{(A\setminus((A+1)\cup\{1\}),\sigma)}
	\end{array}$$ where $(A,\sigma)$ is standard.

	(1) We prove by induction on $|A|$. If $|A|=n-1$, then $(A,\sigma)$ must be standard and we are done. 
	
	Assume the statement is true for all $n-1 \geq |A|>k$. We want to show that for any $\sigma\in\mathfrak{S}_n$ and $A\subseteq[n-1]$ with $|A|=k$, $\Theta_{\NCQSym}( \mathbf{F}_{(A,\sigma)})= \mathbf{K}_{(A\setminus((A+1)\cup\{1\}),\sigma)}$. 
	Let $\mathrm{std}(A,\sigma)=(A,\sigma')$. By inclusion-exclusion, we have
	$$\sum_{A\subseteq C\subseteq [n-1]}(-1)^{|C|-|A|}\mathbf{F}_{(C,\sigma)}=\mathbf{M}_{(A,\sigma)}=\mathbf{M}_{(A,\sigma')}=\sum_{A\subseteq C\subseteq [n-1]}(-1)^{|C|-|A|}\mathbf{F}_{(C,\sigma')}.$$

Let $B$ be the peak set such that $\odd(B)=\oddd(A)$. Since $\oddd(A)\subseteq A$, we have $\std(\odd(A),\sigma)=\std(\odd(A),\sigma')$. Therefore,  $\std(\odd(B),\sigma)=\std(\odd(B),\sigma')$, and so by Lemma \ref{lem:eq-eta}, we have $\boldeta_{(B,\sigma)}=\boldeta_{(B,\sigma')}$. Then, by Lemma \ref{lem:M},
	$$\sum_{A\subseteq C\subseteq [n-1]}(-1)^{|C|-|A|}\mathbf{K}_{(C\setminus((C+1)\cup\{1\}),\sigma)}=\sum_{A\subseteq C\subseteq [n-1]}(-1)^{|C|-|A|}\mathbf{K}_{(C\setminus((C+1)\cup\{1\}),\sigma')}.$$
	Therefore, by induction hypothesis, we have
	\begin{align*}
		\Theta_{\NCQSym}(\mathbf{F}_{(A,\sigma)})&=\Theta_{\NCQSym}\left( \sum_{A\subseteq C\subseteq [n-1]}(-1)^{|C|-|A|}\mathbf{F}_{(C,\sigma')} - \sum_{A\subsetneq C\subseteq [n-1]}(-1)^{|C|-|A|}\mathbf{F}_{(C,\sigma)}\right)\\
		&=\sum_{A\subseteq C\subseteq[n-1]}(-1)^{|C|-|A|}\mathbf{K}_{(C\setminus((C+1)\cup\{1\}),\sigma')}-\sum_{A\subsetneq C\subseteq[n-1]}(-1)^{|C|-|A|}\mathbf{K}_{(C\setminus((C+1)\cup\{1\}),\sigma)}\\
		&=\mathbf{K}_{(A\setminus((A+1)\cup\{1\}),\sigma)}
	\end{align*}
	as desired.
	
	Part (2) follows from (1) by Propositions \ref{prop:fund1} and \ref{prop:fund2}, and (3) follows from the fact that the commutation map $\rho$ sends $\mathbf{F}_{(P,\gamma,\sigma)}$ and $\mathbf{K}_{(P,\gamma,\sigma)}$ to $F_{(P,\gamma)}$ and $K_{(P,\gamma)}$, respectively.
\end{proof}

\begin{Corollary}\label{cor:theta-gcf}
We have that
$$
\begin{array}{cccc}
\Theta_{\NCQSym}:& \NCQSym & \rightarrow & \NCPeak\\
& \scrY_\bfG& \mapsto& \scrF_\bfG.
\end{array} 
$$
Moreover, the following diagram commutes.

\begin{center}
\begin{tikzpicture}
	\node(NCSym) at (4,0){$\mathsf{\NCQSym}$};
	\node(NCQSym) at (8,0){$\QSym$};
	\node(NCGamma) at (4,-4){$\NCPeak$};
	\node(NCPeak) at (8,-4){$\Pi$};
	
	\draw[thick,->]  (NCSym)->(NCQSym);
	\draw[thick,->]  (NCGamma)->(NCPeak);
	\draw[thick,->]   (NCSym)->(NCGamma);
	\draw[thick,->]   (NCQSym)->(NCPeak);
	
	\node(iota1) at (6,0.2){$\rho$};
	\node(iota2) at (6,-3.8){$\rho$};
	\node(ThetaNCSym) at (3.15,-2){$\Theta_{\NCQSym}$};
	\node(ThetaNCQSym) at (8.65,-2){$\Theta_\QSym$};
	
	\node(yg) at (4.5,-1){$\scrY_\bfG$};
	\node(fg) at (4.5,-3){$\scrF_\bfG$};
	
		\node(xg) at (7.5,-1){$\scrX_G$};
	\node(eg) at (7.5,-3){$\scrE_G$};

	
	\draw[thick,|->] (yg)--(fg);
	\draw[thick,|->] (xg)--(eg);
	\draw[thick,|->] (yg)--(xg);
	\draw[thick,|->] (fg)--(eg);
\end{tikzpicture} 
\end{center}
\end{Corollary}
Comparing the product and coproduct formulas for generalized chromatic functions and enriched generalized chromatic functions in noncommuting variables in Section \ref{Sec:NCPeak} yield the following result.
\begin{Theorem}\label{theta-hopf}
	The labelled descent-to-peak map $\Theta_{\NCQSym}$ is a Hopf morphism. Hence, $\NCPeak$ is both a Hopf subalgebra and a quotient of $\NCQSym$.
\end{Theorem}

\section{The generalized Dehn-Sommerville equation}\label{sec:D-S}

In this section, we study the relation between the functions that satisfy the generalized Dehn-Sommerville equation and the elements of the peak algebra in noncommuting variables. \\

A function $f$ from the set of $0\text{-}1$ sequences to $\mathbb{Q}$ 
satisfies \emph{generalized Dehn-Sommerville equation} if for all $0\text{-}1$ sequences $a=a_1a_2\cdots a_{\ell(a)-1}1$ or $()$, $b=1b_2b_3\cdots b_{\ell(b)}$ or $()$, and $j\geq 1$, 
$$\sum_{i=1}^j(-1)^{i-1}f({a0^{i-1}10^{j-i}b})=\begin{cases}
2f({a 0^j b})& \text{if $j$ is odd,}\\
0& \text{otherwise.}
\end{cases}$$


Let $P$ be an Eulerian poset of rank $n+1$ with the minimum $\hat{0}$ and maximum $\hat{1}$. Here we always assume that a chain $\mathcal{C}$ of length $k$ in $P$ does not contain $\hat{0}$ and $\hat{1}$. For such a chain, define a $0\text{-}1$ sequence ${\rm seq}_\mathcal{C}$ of length $k$ such that 
the $i$th component of ${\rm seq}_\mathcal{C}$  is $1$ if and only if there exists an element in $\mathcal{C}$ with rank $i$ in $P$. A \emph{flag $f$-vector} of $P$ is function $f$ from the set of  $0\text{-}1$ sequences to $\mathbb{Q}$ such that   

$$f(a) = |\{\text{Chains $\mathcal{C}$ in $P$ with ${\rm seq}_{\mathcal{C}}=a$}\}|.$$

To each $0\text{-}1$ sequence $a=a_1a_2\cdots a_{n-1}$ correspond a subset $A_a$ of $[n-1]$ such that $i$ is in $A_a$ if and only if $a_i=1$. 
For any function $$f:\bigsqcup_{n\geq 0}\{A: A\subseteq[n-1]\}\to \mathbb{Q},$$ define 
$$f_{0\text{-}1}: \{0\text{-}1~\text{sequences}\} \rightarrow \mathbb{Q}$$ such that 
$$f_{0\text{-}1}(a) = f(A_a).$$

\begin{Theorem}\label{thm:D-S}
 Let $$f:\bigsqcup_{n\geq 0}\{A: A\subseteq[n-1]\}\to \mathbb{Q}.$$ Then the following are equivalent.
 \begin{enumerate}
  \item $f_{0\text{-}1}$ satisfies generalized Dehn-Sommerville equation.
  \item $f_{0\text{-}1}$ is a linear combination of flag $f$-vectors of Eulerian posets.
  \item $\sum f(A)M_A$ is in $\Pi$.
  \item For each $n$, fix $\sigma_n\in\mathfrak{S}_n$. Then $\sum_n\sum_{A\subseteq[n-1]}f(A)\bfM_{(A,\sigma_n)}$ is in $\NCPeak$.
 \end{enumerate}
\end{Theorem}

\begin{proof}
The equivalence of (1) and (2) is established in \cite{BB85} (see also \cite[Theorem 2.2]{B14}), and that of (1) and (3) in \cite[Proposition 1.3]{B03}. If $$\sum_n\sum_{A\subseteq[n-1]}f(A)\bfM_{(A,\sigma_n)} \in \NCPeak,$$ then  commuting the variables we have, $\sum f(A)M_A\in \Pi.$  

Now let $\sum f(A)M_A$ in $\Pi$. Since $\Pi$ is a graded vector space, we have 
$$\sum f(A)M_A=\sum_n \sum_{A\subseteq [n-1]} f(A)M_A.$$

 Now for each $\sigma_n$, consider the linear map 
 $$\begin{array}{cccc}L_n: &\QSym_n &\rightarrow & \NCQSym_n\\
 & M_A & \mapsto & \bfM_{(A,\sigma_n)}.
 \end{array}$$
 We will show that $L_n(\Pi_n)\subseteq \NCPeak_n$. As a result of \cite[Proposition 2.2]{S97}, $\{K_B: B~\text{is a peak set}\}$ is a basis for $\Pi$ where 
 $$K_B=\sum_{B\subseteq A\cup(A-1)}2^{|A|+1}M_A.$$ Note that 
 $$L_n(K_B)=\sum_{B\subseteq A\cup(A-1)}2^{|A|+1}\bfM_{(A,\sigma_n)}= \bfK_{(B,\sigma_n)} \in \NCPeak_n.$$ Therefore, 
 $$L_n(\Pi_n)\subseteq \NCPeak_n.$$
 Consequently, if $$\sum_{A\subseteq[n-1]}f(A)M_A\in\Pi_n,$$ then $$\sum_{A\subseteq[n-1]}f(A)\bfM_{(A,\sigma_n)} =L_n\left(\sum_{A\subseteq[n-1]}f(A)M_A\right)  \in \NCPeak_n,$$ and so  
 $$\sum_n\sum_{A\subseteq[n-1]}f(A)\bfM_{(A,\sigma_n)} \in \NCPeak.$$
\end{proof}

\section{Schur $Q$-functions in noncommuting variables}\label{Sec:NCSym} In this section we define the space of Schur $Q$-functions in noncommuting variables, which is a Hopf subalgebra of symmetric functions in noncommuting variables $\NCSym$. We will follow notation from \cite{R04}.

\subsection{Symmetric functions in noncommuting variables}\label{sec:ncsym} For any set composition $\phi$, define $\pi(\phi)$ to be the set partition with the same blacks as $\phi$. For example, 
$$\pi(256|14|389)=256/389/14.$$
Given two set partitions $\pi=\pi_1/\pi_2/\cdots/\pi_{\ell(\pi)}\vdash [n]$ and $\tau=\tau_1/\tau_2/\cdots/\tau_{\ell(\tau)}\vdash [m]$, the \emph{shifted concatenation} of $\pi$ and $\tau$ is  
$$\pi|\tau=\pi_1/\pi_2/\cdots/\pi_{\ell(\pi)}/\tau_1+n/\tau_2+n/\cdots/\tau_{\ell(\tau)}+n.$$  

Recall that $\{\pi\vdash[n]\}$ forms a lattice defined as $\pi\leq\tau$ if all blocks of $\tau$ is contained in some block of $\pi$. 
Throughout this section, we use $\mu(\pi,\tau)$ to denote the M\"{o}bius function of $\pi$ and $\tau$ in this lattice. 
Let $\hat{0}$ and $\hat{1}$ denote the minimal element $12\cdots n$ and the maximal element $1/2/\cdots/n$, respectively. It is well-known that $\mu(\hat{0},\pi)=\prod_i(-1)^{\lambda_i-1}(\lambda_i-1)!$ where $(\lambda_1,\dots,\lambda_\ell)=\lambda(\pi)$.

For any set partition $\pi$, define 
$$\bfm_\pi=\sum_{\phi\vDash [n]: \atop \pi(\pi)=\pi} \bfM_\phi.$$
Then the \emph{Hopf algebra of symmetric functions in noncommuting variables} is 
$$\NCSym=\bigoplus_{n\geq 0} \NCSym_n$$ where 
$$\NCSym_n=\mathbb{Q}\text{-span}\{ \bfm_\pi: \pi\vdash [n]\}.$$
The set $\{\bfm_\pi\}$ is called the \emph{monomial basis} of $\NCSym$. Moreover, define 
$$\bfp_\pi=\sum_{\tau\leq \pi} \bfm_\tau.$$ The set $\{\bfp_\pi\}$ is a basis for $\NCSym$ and it is called the  \emph{power sum basis} of $\NCSym$. 


There is also a fundamental involution 
$$\begin{array}{cccc}
\omega:&\NCSym&\to&\NCSym\\
& \bfh_\pi &\mapsto & \bfe_\pi.
\end{array}$$
Moreover, 
$$\omega(\bfp_\pi)=(-1)^{n-\ell(\pi)}\bfp_\pi.$$
We will also use the following result.

\begin{Theorem}\label{thm:htop}{\cite[Theorems 3.4 and 3.5]{R04}} 
	For any set partition $\pi$, we have the following. 
	\begin{enumerate} 
		\item $\mathbf{e}_\pi=\displaystyle\sum_{\tau\leq \pi}\mu(\hat{0},\tau)\mathbf{p}_{\tau}$,
		\item  $\mathbf{h}_\pi=\displaystyle\sum_{\tau\leq \pi}|\mu(\hat{0},\tau)|\mathbf{p}_{\tau}$.
	\end{enumerate} 
\end{Theorem}

\subsection{Theta map for symmetric functions in noncommuting variables}\label{subsec:theta-ncsym}
	Define 
	$$\Theta_{\NCSym}=\Theta_{\NCQSym}|_{\NCSym}.$$

\begin{Theorem}\label{thm:NCSym}
	Suppose $\pi=\pi_1/\pi_2/\cdots/\pi_{\ell(\pi)}\vdash[n]$ is a set partition with $n\geq 1$, then
	$$\Theta_{\NCSym}(\mathbf{p}_\pi)=\begin{cases}
		2^{\ell(\pi)}\mathbf{p}_\pi & \text{ if }\pi\text{ is odd set partition,}\\
		0 & \text{ otherwise.}
	\end{cases}$$
	Consequently, $\Theta_\NCSym$ commutes with $\omega$, that is, the following diagram is commutative.
	\begin{center}
		\begin{tikzpicture}
			\node(TV) at (0,0){$\NCSym$};
			\node(Sym) at (4,0){$\NCSym$};
			\node(OV) at (0,-2){$\NCGamma$};
			\node(Omega) at (4,-2){$\NCGamma$};
			
			\draw[thick,->]  (TV)->(Sym);
			\draw[thick,->]  (OV)->(Omega);
			\draw[thick,->]   (TV)->(OV);
			\draw[thick,->]   (Sym)->(Omega);
			
			\node(Phi) at (2,0.2){};
			\node(Phi) at (2,-1.8){};
			\node(thetaV) at (-0.8,-1){$\Theta_{\NCSym}$};
			\node(thetaV) at (4.8,-1){$\Theta_\NCSym$};
			
			\node(omega1) at (2,0.3){$\omega$};
			\node(omega2) at (2,-2+.3){$\omega$};
		\end{tikzpicture} 
	\end{center}
\end{Theorem}

To prove the above theorem, first we need the following lemma. For set compositions $\phi$ and $\psi$, define $\psi\leq \phi$ if  each part of $\psi$ is a union of consecutive parts of $\phi$. Let $\odd(\phi)$ be a set composition that union the maximal segment whose the sizes of the blocks are maximal segments of the form $({\rm even},\dots, {\rm even}, {\rm odd})$ (for example, $\odd(1| 58| 69| 4| 23|7) = 1| 45689| 237$).

\begin{Lemma}\label{lem:odd-set}
Let $\pi$ be an odd set partition. For any $\psi=\psi_1|\psi_2|\cdots|\psi_{\ell(\psi)}$ with $\pi(\psi)\leq \pi$, we have that $$\displaystyle\sum_{\phi\vDash[n]: \atop{ \pi(\phi)\leq \pi,\odd(\phi)\geq\psi \atop |\phi_{\ell(\phi)}|\text{is odd}}}(-1)^{\ell(\phi)}=2^{\ell(\pi)-\ell(\psi)}(-1)^{\ell(\pi)}.$$
\end{Lemma} 

\begin{proof}
We call $\phi\vDash[n]$ \emph{suitable} if  it appears in the above summation, that is, $\pi(\phi)\leq \pi,\odd(\phi)\geq\psi$, and  $|\phi_{\ell(\phi)}|$. Consider that $\phi\vDash[n]$ is suitable  if and only if $\phi=\phi_1|\phi_2|\cdots|\phi_{\ell(\phi)}\leq \pi$ such that for all $j$, $\psi_j=\phi_{i_{j-1}+1}\cup \cdots \cup \phi_{i_j}$ for some $i_j$,  and $|\phi_{i_j}|$ is odd. In particular, $i_0=0$ and $i_{\ell(\psi)}=\ell(\phi)$. 

With the notation above, let $I$ be the collection of all possible tuples $(\phi_{i_j}:1\leq j\leq \ell(\psi))$, that is,
	$$I=\{(\phi_{i_1},\phi_{i_2},\dots,\phi_{i_{\ell(\psi)}}):\pi(\phi)\leq \pi,\text{ for all ~}j, \psi_j=\phi_{i_{j-1}+1}\cup\cdots\cup \phi_{i_j}\text{ and }|\phi_{i_j}|\text{ is odd}\}.$$
	For each $D=(D_1,\dots,D_{\ell(\psi)})\in I$, let $I_D$ be the collection of all suitable $\phi$ with the given tuple $D$, that is, 
	$$I_D=\{\phi\vDash [n]:\pi(\phi)\leq \pi \text{~ and for all ~}j, \psi_j=\phi_{i_{j-1}+1}\cup\cdots\cup \phi_{i_j}\text{ with }\phi_{i_j}=D_j\}.$$
	Then, we have
	$$\sum_{\phi\in Q_\pi,\odd(\phi)\geq\phi' \atop \left|\phi_{\ell(\phi)}\right|\text{ is odd}}(-1)^{\ell(\phi)}=\sum_{D\in I}\sum_{\phi\in I_D}(-1)^{\ell(\phi)}.$$	
	Fix $D=(D_1,\dots,D_{\ell(\psi)})\in I$. For each $j$, let $\psi_j\setminus D_j$ be the union of $a_j$ blocks of $\pi$. Let $D_j$ be the union of $b_j$ blocks of $\pi$. Since all blocks of $\pi$ are odd and $|D_j|$ is odd, we must have $b_j$ is odd. Also,
	$$\sum_{j=1}^{\ell(\psi)}(a_j+b_j)=\ell(\pi).$$
	
	Note that $I_D$ consists of all $\phi\vDash[n]$ such that $\psi_j\setminus D_j=\phi_{i_{j-1}+1}\cup\cdots\cup \phi_{i_j-1}$, and each $\phi_p\subseteq \psi_j\setminus D_j$ is a union of some of the $a_j$ blocks of $\pi$. Therefore, each tuple of blocks $(\phi_{i_{j-1}+1},\dots, \phi_{i_j-1})$ corresponds to a set composition $\varphi$ of $[a_j]$. Thus, we have
	$$\sum_{\phi\in I_D}(-1)^{\ell(\phi)}=\prod_{j=1}^{\ell(\psi)}\sum_{\varphi\models[a_j]}(-1)^{\ell(\varphi)+1}=\prod_{j=1}^{\ell(\psi)}(-1)^{a_j+1}=\prod_{j=1}^{\ell(\psi)}(-1)^{a_j+b_j}=(-1)^{\sum_{j=1}^{\ell(\psi)}a_j+b_j}=(-1)^{\ell(\pi)},$$ where
	the second equality follows from the well-known fact that $\sum_{\varphi\vDash[a]}(-1)^{\ell(\varphi)}=(-1)^{a}$ for all $a$. 
		
	Suppose $\psi_j$ is a union of $c_j$ blocks of $\pi$ with $\sum_{j=1}^{\ell(\psi)} c_j=\ell(\pi)$. Note that $I$ consists of all tuples $(D_1,\dots,D_{\ell(\psi)})$ such that $D_j\subseteq \psi_j$ is a union of odd many blocks of $\pi$. Then,
	$$|I|=\prod_{j=1}^{\ell(\psi)}\left(\sum_{i\text{ is odd}}\binom{c_j}{i}\right)=\prod_{j=1}^{\ell(\psi)}2^{c_j-1}=2^{\ell(\pi)-\ell(\psi)}.$$
	Therefore, 
	$$\displaystyle\sum_{\phi\vDash[n]: \atop{ \pi(\phi)\leq \pi,\odd(\phi)\geq\psi \atop |\phi_{\ell(\phi)}|\text{is odd}}}(-1)^{\ell(\phi)}=\sum_{D\in I}\sum_{\phi\in I_D}(-1)^{\ell(\phi)}=2^{\ell(\pi)-\ell(\psi)}(-1)^{\ell(\pi)}.$$
\end{proof}

\begin{proof}[Proof of Theorem \ref{thm:NCSym}]
	We have that
	$$\displaystyle\mathbf{p}_{\pi}=\sum_{\tau\leq \pi}\mathbf{m}_{\tau}=\sum_{\phi\vDash [n]:\atop  \pi(\phi)\leq \pi}\mathbf{M}_{\phi}.$$
	
	We show that if  there is a block $\pi_j$ in $\pi$ such that $|\pi_j|$ is even, then 
	$$\Theta_{\NCSym}(\mathbf{p}_{\pi})=0.$$ 
	
Define
	$$\begin{array}{cccc}
	\imath: & \{\phi\vDash [n]: \pi(\phi)\leq \pi\} & \rightarrow &  \{\phi\vDash [n]: \pi(\phi)\leq \pi\}
	\end{array} 
	$$
	where 
	$$\imath(\phi)=
	\begin{cases} 
	\phi & \text{if~}\phi_{\ell(\phi)}=\pi_j\\
	\phi_1|\phi_2|\cdots|\phi_{m-1}|\phi_m\cup \phi_{m+1}|\phi_{m+1}|\cdots|\phi_{\ell(\phi)} & \text{if ~} \phi_m=\pi_j, m<\ell(\phi)\\
	\phi_1|\phi_2|\cdots|\phi_{m-1}|\pi_j|\phi_m\setminus \pi_j|\phi_{m+1}|\cdots|\phi_{\ell(\phi)} & \text{if~} \pi_j\subsetneq \phi_m \text{~for some~} m.
	\end{cases} 
	$$
	Consider $\{\phi\vDash [n]: \pi(\phi)\leq \pi\}$ as a sign set where the sign of $\phi$ is $(-1)^{\ell(\phi)}$. Then $\imath$ is a sign reversing involution. \\
	
By Proposition \ref{prop:M}, we have
\begin{equation}\label{eq:theta-M}
\Theta_{\NCQSym}(\mathbf{M}_{\phi})
		=\begin{cases}
		(-1)^{\ell(\phi)+n}\displaystyle\sum_{\psi\leq \odd(\phi)}2^{\ell(\psi)}\mathbf{M}_{\psi} & \text{ if $|\phi_{\ell(\phi)}|$}  \text{~ is odd,}\\
		0 & \text{ otherwise.}
	\end{cases}
	\end{equation}
		
	 Note that if $\imath(\phi)\neq \phi$ and $|\phi_{\ell(\phi)}|$ is odd, then $|\iota(\phi)_{\ell(\iota(\phi))}|$ is odd, moreover,
	$$\odd(\phi)=\odd(\imath(\phi)) \quad \text{and}\quad (-1)^{\ell(\phi)}=(-1)^{\ell(\imath(\phi))\pm 1}.$$  Consequently, if $\iota(\phi)\neq \phi$, 
	$$\Theta_{\NCQSym}(\bfM_\phi)=-\Theta_{\NCQSym}(\bfM_{\imath(\phi)}).$$
	Therefore, 
	\begin{align*}
	\Theta_{\NCSym}(\mathbf{p}_\pi)&=\Theta_{\NCQSym}\left(\sum_{\phi\vDash [n]:\atop  \pi(\phi)\leq \pi}\mathbf{M}_{\phi} \right)\\
	&=\sum_{\phi\vDash [n]:\atop  \pi(\phi)\leq \pi}\Theta_{\NCQSym}\left(\mathbf{M}_{\phi}\right)\\
	&=\sum_{\phi\vDash[n]:\atop \pi(\phi)\leq \pi, \imath(\phi)=\phi}\Theta_{\NCQSym}\left(\mathbf{M}_{\phi}\right)=0,
	\end{align*}
	where the last equality is because if $\imath(\phi)=\phi$, then $|\phi_{\ell(\phi)}|$ is even, and so by Equation \ref{eq:theta-M}, we have $\Theta_{\NCQSym}\left(\mathbf{M}_{\phi}\right)=0$.\\

	Assume $\pi$ is an odd set partition, then since all blocks of $\pi$ have odd sizes, $(-1)^n=(-1)^{\ell(\pi)}$. Using Theorem \ref{thm:htop}
	\begin{align*}
	\Theta_{\NCSym}(\mathbf{p}_\pi)&=\Theta_{\NCQSym}\left(\sum_{\phi\vDash [n]:\atop  \pi(\phi)\leq \pi}\mathbf{M}_{\phi} \right)\\
	&=\sum_{\phi\vDash [n]:|\phi_{\ell(\phi)}| \text{~is odd,}\atop  \pi(\phi)\leq \pi}\Theta_{\NCQSym}\left(\mathbf{M}_{\phi}\right)\\
	(\text{by Equation \ref{eq:theta-M}})&=\sum_{\phi\vDash [n]:|\phi_{\ell(\phi)}| \text{~is odd,}\atop  \pi(\phi)\leq \pi} \left((-1)^{\ell(\phi)+n}\displaystyle\sum_{\psi\leq \odd(\phi)}2^{\ell(\psi)}\mathbf{M}_{\psi}\right)\\
	(\text{by reordering the sums})&=\sum_{\phi\vDash[n]:\atop \pi(\phi)\leq \pi} \left( 
	2^{\ell(\psi)}(-1)^{\ell(\pi)} \left( 
	\sum_{\phi\vDash [n]: \pi(\phi)\leq \pi \atop \odd(\phi)\geq \psi, |\phi_{\ell(\phi)}|\text{~is odd}} (-1)^{\ell(\phi)}
	\right) 
	\right)\bfM_{\psi}\\
	(\text{by Lemma \ref{lem:odd-set}})&=\sum_{\psi\vDash[n]:\atop \pi(\psi)\leq \pi} \left( 
	2^{\ell(\psi)}(-1)^{\ell(\pi)} 2^{\ell(\pi)-\ell(\psi)}(-1)^{\ell(\pi)}
	\right)\bfM_{\psi}\\
	&=\sum_{\psi\vDash[n]:\atop \pi(\phi)\leq \pi} 2^{\ell(\pi)} \bfM_{\psi}\\
	&=2^{\ell(\pi)}\mathbf{p}_\pi.
	\end{align*}

\end{proof}

Recall that for a set partition $\pi$, the \emph{Schur $Q$-function in noncommuting variables} is 
$$\mathbf{q}_\pi=\sum_{\sigma\in \frakS_{\pi}} \scrF_{\sigma\circ Q_{\pi}}=\sum_{\sigma\in \frakS_{\pi}} \scrF_{\sigma\circ P_{\pi}}.$$
Note that  $$\mathbf{q}_\pi=\Theta_{\NCQSym}(\mathbf{h}_\pi)=\Theta_{\NCQSym}(\mathbf{e}_\pi).$$ 
The space of the Schur $Q$-functions is defined to be the image of $\Theta_\Sym$, similarly, we define the \emph{space of Schur $Q$-functions in noncommuting variables}, $\NCGamma$, to be the image of $\Theta_{\NCSym}$ i.e., 
	$$\NCGamma=\mathrm{Img}(\NCSym).$$
	The following corollary is a result of Theorem \ref{thm:NCSym}.

\begin{Corollary}
	For $\pi\vdash[n]$, we have
	$$\Theta_{\NCSym}(\mathbf{m}_\pi)=(-1)^{\ell(\pi)+n}\sum_{\sigma\vdash[n]}C_\pi^\sigma2^{\ell(\psi)}\mathbf{m}_\sigma$$
	where $C_\pi^\sigma=|\{\phi\models[n]:\phi_{\ell(\phi)}\text{ is odd}, B_1|\cdots|B_k\leq\mathrm{Odd}(\phi)\}|$ for $\sigma=B_1/\cdots/B_k$.
\end{Corollary}

\begin{proof}
	By Proposition \ref{prop:M}, We have
	\begin{align*}
		\Theta_\NCSym(\mathbf{m}_\pi)&=\Theta_\NCQSym\left(\sum_{\phi\models[n] \atop \pi(\phi)=\pi}\mathbf{M}_\phi\right)\\
		&=\sum_{\phi\models [n]:|\phi_{\ell(\phi)}|\text{ is odd} \atop \pi(\phi)=\pi}\left((-1)^{\ell(\phi)+n}\sum_{\psi\leq\mathrm{Odd}(\phi)}2^{\ell(\psi)}\mathbf{M}_\psi\right)\\
		&=(-1)^{\ell(\pi)+n}\sum_{\psi}C_\pi^\psi2^{\ell(\psi)}\mathbf{M}_\psi.
	\end{align*}
	where $C_\pi^\psi=|\{\phi\models[n]:|\phi_{\ell(\phi)}|\text{ is odd},\pi(\phi)=\pi,\psi\leq\mathrm{Odd}(\phi)\}|$.
	
	It then follows from Theorem \ref{thm:NCSym} that $\Theta_{\NCSym}(\mathbf{m}_\pi)\in\NCSym$.
\end{proof}

\begin{Corollary}\label{cor:p}
	The set $\{p_\pi:\pi\vdash[n],\pi\text{ is odd set partition}\}$ forms a linear basis of $\NCGamma_n$.
\end{Corollary}
The main goal of this section is to show that 
$$\NCGamma=\NCSym\cap\NCPeak$$
 and also $$\{\mathbf{q}_\pi: \pi \text{~is an odd set partition}\}$$ is a basis for $\NCGamma_n$, so the dimension of  $\NCGamma_n$ is equal to the number of odd set partitions.

Recall from Lemma \ref{lem:N} that
$$\boldeta_{(B,\sigma)}=(-1)^{|B|}\sum_{A\subseteq\odd(B)}2^{|A|+1}\mathbf{M}_{(A,\sigma)}.$$
For any odd set composition $\phi$, we have $$\boldeta_\phi=(-1)^{(n-\ell(\phi))/2}\sum_{\psi\leq \phi} 2^{\ell(\psi)}\bfM_\psi$$
under the identification ${\rm SetCompOdd}:(B,\sigma)\mapsto\phi$. 

For any odd set partition $\pi$ define 
$$\mathbf{n}_\pi=(-1)^{(n-\ell(\pi))/2}\sum_{\pi(\phi)=\pi} \boldeta_\phi.$$

\begin{Theorem}\label{thm:NCGamma}
	 We have that $$\{\mathbf{n}_\pi: \pi\vdash[n], \pi\text{ is an odd set partition}\}$$ is a basis for $\NCSym\cap\NCPeak$.
\end{Theorem}

\begin{proof}
	Clearly, $\{\mathbf{n}_\pi: \pi\vdash[n], \pi\text{ is an odd set partition}\}$ is linearly independent and $\mathbf{n}_\pi\in\NCPeak$.
	
For any odd set composition $\pi\vdash[n]$, we have
	$$\mathbf{n}_\pi=(-1)^{(n-\ell(\pi))/2}\sum_{\pi(\phi)=\pi}\boldeta_\phi=\sum_{\pi(\phi)=\pi}\sum_{\psi\leq\phi}2^{\ell(\psi)}\mathbf{M}_\psi=\sum_{\psi\models [n]}\left(\sum_{\phi\geq\psi \atop \pi(\phi)=\pi}1\right)2^{\ell(\psi)}\mathbf{M}_\psi.$$ 
	
	Suppose $\pi(\psi)=\pi(\varphi)$. We define a bijection $$g:\{\phi\geq\psi:\pi(\phi)=\pi\}\to\{\phi\geq\varphi:\pi(\phi)=\pi\}$$ as follows: if $\psi$ is obtained from $\phi$ by replacing some consecutive blocks $\phi_t|\phi_{t+1}|\cdots|\phi_{t+k}$ of $\phi$ by $\psi_i$, then since $\psi_i=\varphi_j$, we obtain $g(\phi)$ by replacing $\varphi_j$ by $\phi_t|\phi_{t+1}|\cdots|\phi_{t+k}$. Then $g$ is a bijection, and so  $\mathbf{n}_\pi\in\NCSym$.
	
	Suppose $f\in \NCSym\cap\NCPeak_n$. We show that 
	$f$ is a linear combination of the elements in the set $\{\mathbf{n}_\pi: \pi~ \text{is an odd set partition}\}$. 
	
Let $f=\displaystyle\sum_{\phi\models [n] \atop \phi\text{ is odd }}c_\phi\boldeta_\phi$. We have
	$$f=\sum_{\phi\models [n] \atop \phi\text{ is odd }}(-1)^{(n-\ell(\phi))/2}\sum_{\psi\leq\phi}2^{\ell(\psi)}c_\phi\mathbf{M}_\psi=\sum_{\psi\models[n]}\left(\sum_{\psi\leq\phi \atop \phi\text{ is odd}}(-1)^{(n-\ell(\phi))/2}c_\phi\right)2^{\ell(\psi)}\mathbf{M}_\psi.$$
	Since $f\in\NCSym$, we must have
	$$\sum_{\psi\leq\phi \atop \phi\text{ is odd}}(-1)^{(n-\ell(\phi))/2}c_\phi=\sum_{\varphi\leq\phi \atop \phi\text{ is odd}}(-1)^{(n-\ell(\phi))/2}c_\phi$$
	whenever $\pi(\psi)=\pi(\varphi)$. We prove $\pi(\psi)=\pi(\varphi)$ implies $c_\psi=c_{\varphi}$ by induction on $\ell(\psi)$. If $\psi,\varphi\vDash[n]$ and $\ell(\psi)=\ell(\varphi)=n$, the equality reduces to $c_\psi=c_{\varphi}$ and we are done. Assume $\pi(\psi)=\pi(\varphi)$ implies $c_\psi=c_{\varphi}$ for all odd set compositions $\psi,\varphi$ with $\ell(\psi)>\ell$. Take odd set compositions $\psi$ and $\varphi$ with $\pi(\psi)=\pi(\varphi)$ and $\ell(\psi)=\ell$. We define a bijection $$\mathsf{g}:\{\phi\geq\psi:\phi\text{ is odd}\}\to\{\phi\geq\varphi:\phi\text{ is odd}\}$$ as follows: if $\psi$ is obtained from $\phi$ by replacing some consecutive blocks $\phi_t|\phi_{t+1}|\cdots|\phi_{t+k}$ of $\phi$ by $\psi_i$, then since $\psi_i=\varphi_j$, we obtain $\mathsf{g}(\phi)$ by replacing $\varphi_j$ by $\phi_t|\phi_{t+1}|\cdots|\phi_{t+k}$. Then $\mathsf{g}$ is a bijection. Then, $\pi(\phi)=\pi(\mathsf{g}(\phi))$ and by induction hypothesis, we have
	\begin{align*}
		(-1)^{(n-\ell)/2}(c_\psi-c_{\varphi})&=\sum_{\varphi<\phi \atop \phi\text{ is odd }}(-1)^{(n-\ell(\phi))/2}c_\phi-\sum_{\psi<\phi \atop \phi\text{ is odd}}(-1)^{(n-\ell(\phi))/2}c_\phi\\
		&=\sum_{\psi<\phi \atop \phi\text{ is odd}}(-1)^{(n-\ell(\phi))/2}(c_{\mathsf{g}(\phi)}-c_\phi)=0.
	\end{align*}
	Therefore, $\NCSym\cap\NCPeak$ is spaned by $\{\mathbf{n}_\pi:\pi\text{ is odd}\}$, and since this is a linearly independent set, we have $\{\mathbf{n}_\pi:\pi\text{ is odd}\}$ is a basis for $\NCSym\cap\NCPeak$. 
\end{proof}

\begin{Corollary}\label{coro:NCGamma}
	We have $\NCGamma=\NCSym\cap \NCPeak$. In particular, $\NCGamma$ is both a subalgebra and a quotient algebra of $\NCSym$.
\end{Corollary}

\begin{proof}
	By Theorem \ref{thm:NCSym}, $\NCGamma\subseteq\NCSym\cap\NCPeak$. Corollary \ref{cor:p} and Theorem \ref{thm:NCGamma} implies $\NCGamma$ and $\NCSym\cap\NCPeak$ have the same dimension in each degree.
\end{proof}

\begin{Theorem}\label{thm:qschur} We have the following. 
	\begin{enumerate}
		\item  For each set partition $\pi$, we have $\omega(\mathbf{q}_\pi)=\mathbf{q}_\pi$. 
		\item The set $\{\mathbf{q}_\pi:\pi\text{ is odd set partition}\}$ forms a linear basis of $\NCGamma$.
		\item As an algebra, $\NCGamma$ is freely generated by the set $\{\mathbf{p}_\pi:\pi\text{ is connected odd set partition}\}$ and the set $\{\mathbf{q}_\pi:\pi\text{ is connected odd set partition}\}$.
	\end{enumerate}
\end{Theorem}

\begin{proof} 
(1) Note that $$\omega(\mathbf{q}_\pi) =\omega(\Theta_{\NCSym}(\bfh_\pi)) =\Theta_{\NCSym}(\omega(\bfh_\pi))=\Theta_{\NCSym}(\bfe_\pi)=\mathbf{q}_\pi.$$

(2) 
By Theorems \ref{thm:NCSym} and \ref{thm:htop},
$$\mathbf{q}_\pi=\sum_{\tau\leq\pi \atop \tau\text{ is odd set partition}}|\mu(\hat{0},\tau)|2^{\ell(\tau)}\mathbf{p}_{\tau}.$$
Observe that for any odd set partition $\tau$, $\mu(\hat{0},\tau)$ is always positive. Hence,
$$\mathbf{q}_\pi=\sum_{\tau\leq\pi \atop \tau\text{ is odd }}\mu(\hat{0},\tau)2^{\ell(\tau)}\mathbf{p}_{\tau}=\Theta_{\NCQSym}(\mathbf{e}_\pi).$$
	 By triangularity we have  the set $\{\mathbf{q}_\pi:\pi\text{ is odd}\}$ forms a linear basis of $\NCGamma$.
	\\
	
(3)	We have that $\{\mathbf{h}_\pi:\pi\text{ is connected set partition}\}$ generates $\NCSym$. Also, $\mathbf{h}_\pi\cdot \mathbf{h}_{\tau}=\mathbf{h}_{\pi|\tau}$. Note that $\pi|\tau$ is an odd set partition if and only if $\pi$ and $\tau$ are odd set partitions. Consequently, the set $\{\mathbf{h}_\pi:\pi\text{ is odd set partition}\}$ can be generated by $\{\mathbf{h}_\pi:\pi\text{ is connected odd set partition}\}$. Since $\Theta_{\NCQSym}$ is an algebra epimorphism, the set $\{\mathbf{q}_\pi:\pi\text{ is connected odd set partition}\}$ generates $\{\mathbf{q}_\pi:\pi\text{ is odd set partition}\}$ and $\NCGamma$.
	
	The set $\{\mathbf{q}_\pi:\pi\text{ is connected odd set partition}\}$ must be algebraically independent because $\{\mathbf{p}_\pi:\pi\text{ is connected odd set partition}\}$ is algebraically independent and these two sets have the same cardinality in every degree. This proves (4).
\end{proof}

\begin{Remark}
The commutative version of Theorem \ref{thm:qschur} is a classical result that can be found in section 3.8 of \cite{M95}.
\end{Remark}

\subsection{Theta maps for noncommutative symmetric functions}\label{sec:NSym}
We define the noncommutative symmetric functions $\NSym$ to be the Hopf subalgebra of $\NCSym$ generated by $H_{n}$ for all $n\geq 1$, where $H_\alpha$ is the dual of $M_\alpha$. Thus we have
$$\NSym=\langle H_n: n\geq 1\rangle.$$ 
Thus, the Hopf algebra $\NSym$ is the dual of $\QSym$. In \cite[Section 3]{ALi}, the authors defined the shuffle basis $\{S_\alpha\}$ for $\QSym$. Let $\mathbf{p}_\alpha$ be the dual of $S_\alpha$. The peak algebra for $\NSym$ is $$\Pi^*=\langle \mathbf{p}_{n}: n\geq 1, \text{$n$ is odd} \rangle,$$ and the  theta map for $\NSym$, $\Theta_\NSym=\Theta_{\QSym}^*$, by \cite[Theorem 3.8]{ALi} is 
$$\begin{array}{cccc}\Theta_{\NSym}: &\NSym &\rightarrow& \Pi^* \\ & \mathbf{p}_{\alpha} & \mapsto & \begin{cases}
2 \mathbf{p}_{\alpha} & \text{if $\alpha$ is odd},\\ 0 & \text{otherwise.}
\end{cases} \end{array}$$

\section{Combinatorial identities and conjectures}\label{Sec:combid}

This section is a preparation for our main result on internal coproduct formula, and it has its own combinatorial interests.

Suppose $\phi=\phi_1|\phi_2|\cdots|\phi_{\ell(\phi)}$ and $\psi=\psi_1|\psi_2|\cdots|\psi_{\ell(\psi)}$. Let $\phi\land\psi$ be the set composition obtained from $(\phi_1\cap \psi_1,\phi_1\cap \psi_2,\dots,\phi_1\cap \psi_{\ell(\psi)},\phi_2\cap \psi_1,\dots,\phi_{\ell(\phi)}\cap \psi_{\ell(\psi)})$ with empty blocks removed.
For two set compositions $\phi,\psi\models[n]$, recall the \emph{refinement order} is given by $\psi\leq\phi$ if the blocks of $\psi$ are union of consecutive blocks of $\phi$.

%
%

We now consider the sub-poset $\{\phi\models[n]:\phi\text{ is odd set composition}\}$, with inherited refinement order $\leq$. In this section, for odd set compositions $\phi,\psi\models[n]$, we let $\mu(\phi,\psi)$ denote the M\"{o}bius function of $(\phi,\psi)$ on this sub-poset.

\begin{Lemma}\label{lem:poset}
	If $\phi\vDash[n]$ is an odd set composition, then the sub-posets $P=\{\psi\leq\phi:\psi\text{ is odd set composition}\}$ is isomorphic to the poset $P'=\{\alpha\vDash \ell(\phi):\alpha\text{ is odd composition}\}$.
\end{Lemma}

\begin{proof}
	We construct a map $f:P\to P'$ such that $\ell(\psi)=\ell(f(\psi))$ recursively as follows. First, define $f(\phi)=\hat{1}=(1,1,\dots,1)$. Suppose we have already defined $f(\psi)=\alpha$ with $\psi=\psi_1|\psi_2|\cdots|\psi_{\ell(\psi)}$ and $\alpha=\alpha_1\alpha_2\cdots\alpha_{\ell(\psi)}$. If $\varphi\in P$ is obtained from $\psi$ by merging $\psi_i,\psi_{i+1},\psi_{i+2}$, then we define $f(\varphi)$ to be obtained from $\alpha$ by merging $\alpha_i,\alpha_{i+1},\alpha_{i+2}$. Then $f(\varphi)\in P'$ and $\ell(\psi)=\ell(f(\psi))$. Clearly, $f$ is a bijection that preserves the partial order.
\end{proof}

From the above lemma, we can conclude that $$\displaystyle \sum_{\psi\leq\phi \atop \psi\text{ is odd}}\mu(\psi,\phi)$$ depends only on the length of $\phi$. We define $\mu_\ell$ to be the sum above for any $\phi$ with $\ell(\phi)=\ell$. 
If $\ell$ is odd, then there is a unique minimum element among the summands, namely $\{[\ell]\}$. Hence, $\mu_1=1$ and $\mu_\ell=0$ for $\ell>1$. For any odd composition $\beta\leq\alpha$, it is also consistent to write $\mu(\beta,\alpha)=\mu(\psi,\phi)$ when $\alpha(\psi)=\beta$, $\alpha(\phi)=\alpha$ and $\psi\leq\phi$.

\begin{Remark}
	It has been shown in \cite[Corollary 1.4 and Theorem 2.5]{C15}  that $\mu(\hat{0},\alpha)$ and $\mu_\ell$ are signed Catalan numbers. More precisely, let $\Cat_n=\frac{1}{n+1}\binom{2n}{n}$ be the $n$-th Catalan number. If $\alpha\models[n]$ is odd and $n$ is odd, we have
	$$\mu(\hat{0},\alpha)=(-1)^{(\ell(\alpha)-1)/2}\Cat_{(\ell(\alpha)-1)/2}$$
	Moreover, if $\ell$ is even, then $\mu_\ell=(-1)^{\ell/2-1}\Cat_{\ell/2-1}$.
\end{Remark}

\begin{Lemma}\label{lem:case1}
	For any $n$, we have
	$$\sum_{\alpha\vDash n \atop \alpha\text{ is odd}}\mu_{\ell(\alpha)}=1.$$
\end{Lemma}

\begin{proof}
	By definition, we can rewrite the left-hand side as
	$$\sum_{\alpha\vDash  n \atop \alpha\text{ is odd}}\sum_{\beta\leq\alpha \atop \beta\text{ is odd}}\mu(\beta,\alpha)=\sum_{\beta\vDash n \atop \beta\text{ is odd}}\sum_{\beta\leq\alpha \atop \alpha\text{ is odd}}\mu(\beta,\alpha).$$
	Since there is a unique maximal element in the poset of odd compositions of $n$, namely $\hat{1}=(1,1,\dots,1)$, for any $\beta\neq \hat{1}$,
	$$\sum_{\beta\leq\alpha \atop \beta\text{ is odd}}\mu(\beta,\alpha)=0.$$
	Hence, the expression can be further reduced to $\mu(\hat{1},\hat{1})=1$.
\end{proof}

For a set composition $\phi=\phi_1|\phi_2|\cdots|\phi_{\ell(\phi)}$ and a set composition $\psi\leq\phi$, we define $$\mathscr{C}_\phi^\psi=\{\phi_{\tau(1)}|\phi_{\tau(2)}|\cdots|\phi_{\tau(\ell(\phi))}:\tau\in\mathfrak{S}_{\ell(\phi)},\psi\land \left(\phi_{\tau(1)}|\phi_{\tau(2)}|\cdots|\phi_{\tau(\ell(\phi))}\right)=\phi\}.$$
For example, if $\phi=3|1|4|2$ and $\psi=3|124$, then $\mathscr{C}_\phi^\psi=\{3|1|4|2,1|3|4|2,1|4|3|2,1|4|2|3\}$.

\begin{Proposition}\label{lem:int1}
	Let $\phi$ be a set composition and $\psi\leq\phi$. Then for a set compostition $\varphi$, we have $\psi\leq\psi\land\varphi\leq\phi$ if and only if $\varphi\leq\xi$ for some $\xi\in\mathscr{C}_\phi^\psi$.
\end{Proposition}

\begin{proof}
	The fact that $\psi\leq\psi\land\varphi$ is obvious. Let $\phi=\phi_1|\phi_2|\cdots|\phi_\ell(\phi)$, $\psi\land\varphi=C_1|C_2|\cdots|C_k$ and $\varphi=\varphi_1|\varphi_2|\cdots|\varphi_{\ell(\varphi)}$.

	If $\psi\land\varphi\leq\phi$, then each block $\phi_i$ of $\phi$ is contained in some block $C_j$ of $\psi\land\varphi$ which is contained in some block $\psi_m$ of $\varphi$. Let $\xi\geq\varphi$ obtained from $\varphi$ by decomposing each block $\psi_m$ into $\phi_{m_1}|\phi_{m_2}|\cdots|\phi_{m_q}$ such that $m_1<m_2<\cdots<m_q$.
	
	Then, $\psi\land\xi$ is obtained by decomposing each block $C_j=A\cap \psi_m$ for some block $A$ of $\psi$ into $A\cap \phi_{m_1}|A\cap \phi_{m_2}|\cdots|A\cap \phi_{m_q}$. Each block $A\cap \phi_{m_i}$ is either $\phi_{m_i}$ or $\emptyset$.
	
	Assume that $\psi\land\xi=\phi_{\tau(1)}|\phi_{\tau(2)}|\cdots|\phi_{\tau(\ell(\phi))}.$
	\begin{itemize}
		\item If $\phi_{\tau(i)}\subseteq C_j$, $\phi_{\tau(i+1)}\subseteq C_{j+1}$ for some $j$, then $\tau(i)<\tau(i+1)$ since $\psi\land\varphi\leq\phi$.
		\item If $\phi_{\tau(i)},\phi_{\tau(i+1)}\subset C_j$ for some $j$, then $\tau(i)=m_s$, $\tau(i+1)=m_t$ where $\phi_{m_s},\phi_{m_t}\subset \psi_m$ for some $m$ and $s<t$. By construction, $m_s<m_t$ and $\tau(i)<\tau(i+1)$.
	\end{itemize}
	Therefore, $\tau$ is the identity permutation and $\psi\land\xi=\phi$, and so $\xi\in\mathscr{C}_\phi^\psi$.
	
	The converse can be seen from covering relations. If $\varphi$ is obtained from $\varphi'$ by merging two blocks $A_1$ and $A_2$ of $\varphi'$, then $\psi\land\varphi$ is obtained from $\psi\land\varphi'$ by merging $A\cap A_1$ and $A\cap A_2$ for each block $A$ of $\psi$. Hence, $\varphi\leq\xi$ implies $\psi\land\varphi\leq\psi\land\xi=\phi$.
\end{proof}

Let $$\mathscr{D}_\phi^\psi = \{\chi\leq\xi:\xi\in\mathscr{C}_\phi^\psi,\chi\text{ is odd set composition}\}=\{\chi:\psi\land\chi\leq\phi,\chi\text{ is odd set composition}\}.$$

For example, take $\phi=3|1|4|2$ and $\psi=3|124$, then $\mathscr{C}_\phi^\psi=\{3|1|4|2,1|3|4|2,1|4|3|2,1|4|2|3\}$ and $\mathscr{D}_\phi^\psi=\{3|1|4|2,1|3|4|2,1|4|3|2,1|4|2|3,134|2,3|124,1|234,124|3\}$.

The following theorem is the most technical part of this section.

\begin{Theorem}\label{lem:main}
	Let $\phi\vDash[n]$ be an odd set composition and $\psi\leq\phi$. Then,
	\begin{equation}\label{eq:1}
		\sum_{\chi\in\mathscr{D}_\phi^\psi}2^{\ell(\psi\land\chi)-\ell(\chi)}\mu_{\ell(\chi)}=2^{\ell(\psi)-1}.
	\end{equation}
\end{Theorem}

\begin{proof}
	First observe that for any set composition appearing on the left-hand side, each of its blocks is a union of blocks of $\phi$. Therefore, the exact contents of the blocks in $\phi$ play no role, as long as the sizes are odd. We only need to prove the case  $\phi=1|2|\cdots|n$.
	
	If $n$ is odd, then $\ell(\chi)$ is odd for all $\chi\in\mathscr{D}_\phi^\psi$. Hence, $\mu_{\ell(\chi)}=0$ unless $\chi=\{[n]\}$. The equality holds since $\psi\land\{[n]\}=\psi$.
	
	Assume $n$ is even. We begin with a special case. If $\psi=\{[n]\}$, then $\mathscr{C}_\phi^\psi=\{\phi\}$ and the poset $\mathscr{D}_\phi^\psi$ is isomorphic to the poset $\{\alpha\vDash n:\alpha\text{ is odd composition}\}$. Furthermore, $\psi\land\chi=\chi$ for all $\chi\vDash[n]$. Hence, the equality holds by Lemma \ref{lem:case1}.

	We prove Equation \ref{eq:1} by induction on $n$. The base case $n=2$ can be easily checked and we omit it. Assume the equality holds for $n<\ell$. Let $\phi=1|2|\cdots|\ell$. 
	
	Let $\psi=\psi_1|\psi_2|\cdots|\psi_{\ell(\psi)}$ with $\psi_i=\{j+1,j+2,\dots,j+m\}$ and $\psi_{i+1}=\{j+m+1\}$, and let $\psi'=\psi_1|\cdots|\psi_{i-1}|\psi_i\cup \psi_{i+1}|\psi_{i+2}|\cdots|\psi_{\ell(\psi)}$. That is, $\psi'$ is obtained from $\psi$ by merging a block of size $m$ and a block of size $1$. Then, it suffices to show that $$\displaystyle\sum_{{\chi\in\mathscr{D}_\phi^\psi}}2^{\ell(\psi\land\chi)-\ell(\chi)}\mu_{\ell(\chi)}=2^{\ell(\psi)-1} \text{~if and only if~} \displaystyle\sum_{{\chi\in\mathscr{D}_\phi^{\psi'}}}2^{\ell(\psi'\land\chi)-\ell(\chi)}\mu_{\ell(\chi)}=2^{\ell(\psi')-1}$$ since all set compositions $\varphi\leq\phi$ can be obtained by a series of such processes from the base case $\{[n]\}$.
	
	Suppose $\chi=\chi_1|\cdots|\chi_{\ell(\chi)}\in\mathscr{D}_\phi^\psi$ and $\sigma\in\mathfrak{S}_{\ell(\chi)}$. We now define $\sigma(\chi)$ to be the set composition obtained from $\chi_{\sigma(1)}|\cdots|\chi_{\sigma({\ell(\chi)})}$ by \emph{renumbering} each block $\psi_r$ of $\psi$ increasingly, that is, if $\psi_r=\{a_1<\cdots<a_p\}$ and $\chi_{\sigma(1)}|\cdots|\chi_{\sigma({\ell(\chi)})}=\SetComp(A,\tau)$, then $\sigma(\chi)=\SetComp(A,\tau')$ where $\tau'$ is the permutation with $\tau'^{-1}(a_1)<\cdots<\tau'^{-1}(a_p)$, and $\tau^{-1}(q)=\tau'^{-1}(q)$ for $q\notin\psi_r$. For example, if $\psi=146|35|2$, $\chi=134|256$ and $\sigma=21$, then $\chi_{\sigma(1)}|\chi_{\sigma(2)}=256|134$, and we need to renumber $\{1,4,6\}$ and $\{3,5\}$ increasingly. Hence, $\sigma(\chi)=123|456$.
	
	Clearly, $\sigma(\chi)$ is an odd set composition. The renumbering ensures that $\psi\land\sigma(\chi)\leq\phi$. Therefore, $\sigma(\chi)\in\mathscr{D}_\phi^\psi$. Moreover, $\ell(\psi\land\chi)-\ell(\chi)=\ell(\psi\land\sigma(\chi))-\ell(\sigma(\chi))$ because 
	$$\ell(\psi\land\chi)=\sum_{1\leq p\leq {\ell(\chi)}}|\{1\leq r\leq {\ell(\psi)}:\chi_p\cap \psi_r\neq\emptyset\}|$$
	and $\{1\leq r\leq {\ell(\psi)}:\chi_p\cap \psi_r\neq\emptyset\}=\{1\leq r\leq {\ell(\psi)}:\chi_{\sigma(p)}\cap \psi_r\neq\emptyset\}$. For simplicity, we write $c_\chi^\psi=2^{\ell(\psi\land\chi)-\ell(\chi)}\mu_{\ell(\chi)}$. Then we have $c_\chi^\psi=c_{\sigma(\chi)}^\psi$.
	
	We partition $\{\chi\in\mathscr{D}_\phi^\psi\}$ into classes according to $\chi|_{\psi_i}$ and $\chi|_{\psi_i\cup \psi_{i+1}}$. For any $1\leq t\leq m$, let $A^t=\{\chi\in\mathscr{D}_\phi^\psi:\ell(\chi|_{\psi_i})=\ell(\chi|_{\psi_i\cup \psi_{i+1}})-1=t\}$ and $B^t=\{\chi\in\mathscr{D}_\phi^\psi:\ell(\chi|_{\psi_i})=\ell(\chi|_{\psi_i\cup \psi_{i+1}})=t\}$.
	
	Assume $\chi=\chi_1|\cdots|\chi_{\ell(\chi)}\in\mathscr{D}_\phi^\psi$ and let $\{1\leq p\leq {\ell(\chi)}:\chi_p\cap \psi_i\neq\emptyset\}=\{p_1<\dots<p_t\}$. If $\chi\in A^t$, let $\psi_{i+1}\subseteq \chi_{p_{t+1}}$, then there are $t+1$ permutations in the subgroup of $\mathfrak{S}_{\ell(\chi)}$ generated by the $(t+1)$-cycle $(p_1,p_2,\dots,p_{t+1})$. Moreover, for each such permutation $\sigma$, we have $c_{\sigma(\chi)}^\psi=c_\chi^\psi=c_\chi^{\psi'}$. Among these cyclic permutations, there is a unique permutation $\sigma$ making $\sigma(p_{t+1})=\max\{p_1,p_2,\dots,p_{t+1}\}$, in which case $\sigma(\chi)\in  \mathscr{D}_\phi^{\psi'}$. Therefore, $$\displaystyle\sum_{\chi\in A^t}c_\chi^\psi=(t+1)\sum_{\chi\in A^t\cap\mathscr{D}_\phi^{\psi'}}c_\chi^{\psi'}.$$
	
	Similarly, if $\chi\in B^t$, then there are $t$ permutations in the subgroup of $\mathfrak{S}_{\ell(\chi)}$ generated by the $t$-cycle $(p_1,p_2,\dots,p_t)$. Moreover, for each such permutation $\sigma$, we have $c_{\sigma(\chi)}^\psi=c_\chi^\psi=2c_\chi^{\psi'}$. Among these cyclic permutaitons,  there is a unique permutation $\sigma$ making $\psi_{i+1}\subseteq \chi_{p_t}$, in which case $\sigma(\chi)\in \mathscr{D}_\phi^{\psi'}$. Therefore, $$\displaystyle\sum_{\chi\in B^t}c_\chi^\psi=2t\sum_{\chi\in B^t\cap\mathscr{D}_\phi^{\psi'}}c_\chi^{\psi'}.$$ Combining these two cases yields
	$$\sum_{\chi\in\mathscr{D}_\phi^\psi}c_\chi^\psi=\sum_{\alpha\vDash m}\left((\ell(\alpha)+1)\sum_{\chi\in\mathscr{D}_\phi^{\psi'} \atop \alpha\left(\chi|_{\psi_i\cup \psi_{i+1}}\right)=\alpha\cdot 1}c_\chi^{\psi'}+2\ell(\alpha)\sum_{\chi\in\mathscr{D}_\phi^{\psi'} \atop \alpha\left(\chi|_{\psi_i\cup \psi_{i+1}}\right)=\alpha\odot 1}c_\chi^{\psi'}\right).$$
	
	Note that for $\chi\in\mathscr{D}_\phi^{\psi'}$, the order of appearance of $\psi_i\cup \psi_{i+1}$ is fixed. Hence,
	$$\sum_{\chi\in\mathscr{D}_\phi^{\psi'}}c_\chi^{\psi'}=\sum_{\alpha\vDash m}\left(\sum_{\chi\in\mathscr{D}_\phi^{\psi'} \atop \alpha\left(\chi|_{\psi_i\cup \psi_{i+1}}\right)=\alpha\cdot 1}c_\chi^{\psi'}+\sum_{\chi\in\mathscr{D}_\phi^{\psi'} \atop \alpha\left(\chi|_{\psi_i\cup \psi_{i+1}}\right)=\alpha\odot 1}c_\chi^{\psi'}\right).$$
	
	If $m=1$, then $\{\alpha\vDash m\}=\{1\}$ and $$\displaystyle\sum_{\chi\in\mathscr{D}_\phi^{\psi}}c_\chi^{\psi} = 2\sum_{\chi\in\mathscr{D}_\phi^{\psi'}}c_\phi^{\psi'},$$ so we are done.
	
	Assume that $m\geq2$, let $\phi'$ be obtained from $\phi$ by merging $\{j+1\},\{j+2\}$, and $\{j+3\}$. We use again the fact that the contents of blocks in $\phi'$ play no role and the induction hypothesis to get
	$$\sum_{\alpha\vDash m-1}\sum_{\chi\in\mathscr{D}_\phi^{\psi'} \atop \alpha\left(\chi|_{\psi_i\cup \psi_{i+1}}\right)=2\odot\alpha}c_\chi^{\psi'}=\sum_{\chi\in\mathscr{D}_{\phi'}^{\psi'}}c_\chi^{\psi'}=2^{\ell(\psi')-1}.$$
	
	We now claim that
	$$\sum_{\chi\in\mathscr{D}_\phi^\psi}c_\chi^\psi+(m-1)\left(\sum_{\alpha\vDash m-1}\sum_{\chi\in\mathscr{D}_\phi^{\psi'} \atop \alpha\left(\chi|_{\psi_i\cup \psi_{i+1}}\right)=2\odot\alpha}c_\chi^{\psi'}\right)=(m+1)\sum_{\chi\in\mathscr{D}_\phi^{\psi'}}c_\chi^{\psi'}.$$
	
	If the claim were true, we have
	$$\sum_{\chi\in\mathscr{D}_\phi^\psi}c_\chi^\psi+(m-1)2^{\ell(\psi')-1}=(m+1)\sum_{\chi\in\mathscr{D}_\phi^{\psi'}}c_\chi^{\psi'}.$$
	It is clear that  $$\displaystyle\sum_{\chi\in\mathscr{D}_\phi^\psi}c_\chi^\psi=2^{\ell(\psi)-1}\text{~if and only if~}\displaystyle\sum_{\chi\in\mathscr{D}_\phi^{\psi'}}c_\chi^{\psi'}=2^{\ell(\psi')-1}$$ and we are done.
	
	To prove our claim, we define an equivalence relation on $\{\alpha\vDash m+1\}\cup\{\alpha\vDash m-1\}$ by $\alpha$ is equivalent to $\beta=(\beta_1,\dots,\beta_t)$ if
	\begin{enumerate}
		\item $\alpha=(\beta_{\sigma(1)},\dots,\beta_{\sigma(t)})$ for some $\sigma\in\mathfrak{S}_t$, or
		\item $\beta\vDash m-1$ and $\alpha=(\beta_1+2,\beta_2,\dots,\beta_t)$,
	\end{enumerate}
	and take the transitive closure. For example, take $m=6$, then the equivalence class containing $32$ is $\{32,23,52,25,43,34,41,14,61,16\}$. Equivalently, suppose $\alpha$ consists of $a$ odd parts and $b$ even parts, then the equivalence class of $\alpha$ is the set of $\beta\vDash m+1$ or $m-1$ that consists of $a$ odd parts and $b$ even parts. We use $E(a,b)$ to denote such an equivalence class.
	
	If $\alpha\vDash m+1$ is equivalent to $\beta\vDash m+1$ satisfying (1), then by the group action,
	$$\sum_{\chi\in\mathscr{D}_\phi^{\psi'} \atop \alpha\left(\chi|_{\psi_i\cup \psi_{i+1}}\right)=\alpha}c_\chi^{\psi'}=\sum_{\chi\in\mathscr{D}_\phi^{\psi'} \atop \alpha\left(\chi|_{\psi_i\cup \psi_{i+1}}\right)=\beta}c_\chi^{\psi'}.$$
	
	If $\alpha\vDash m-1$ is equivalent to $\beta\vDash m-1$ satisfying (1), we claim
	$$\sum_{\chi\in\mathscr{D}_\phi^{\psi'} \atop \alpha\left(\chi|_{\psi_i\cup \psi_{i+1}}\right)=2\odot\alpha}c_\chi^{\psi'}=\sum_{\chi\in\mathscr{D}_\phi^{\psi'} \atop \alpha\left(\chi|_{\psi_i\cup \psi_{i+1}}\right)=2\odot\beta}c_\chi^{\psi'}.$$
	For all compositions $\chi$ on both sides, $j+1,j+2$ and $j+3$ appear in the same block of $\chi$. Hence, we can view $\{j+1,j+2,j+3\}$ as a single element and apply the group action.
	
	If $\alpha\vDash m+1$ is equivalent to $\beta\vDash m-1$ satisfying (2), then obviously
	$$\sum_{\chi\in\mathscr{D}_\phi^{\psi'} \atop \alpha\left(\chi|_{\psi_i\cup \psi_{i+1}}\right)=\alpha}c_\chi^{\psi'}=\sum_{\chi\in\mathscr{D}_\phi^{\psi'} \atop \alpha\left(\chi|_{\psi_i\cup \psi_{i+1}}\right)=2\odot\beta}c_\chi^{\psi'}.$$
	
	For each equivalence class $E(a,b)$, let $\alpha\in E(a,b)$ with $\alpha\vDash m+1$, and let $$c(a,b)=\displaystyle\sum_{\chi\in\mathscr{D}_\phi^{\psi'} \atop \alpha\left(\chi|_{\psi_i\cup \psi_{i+1}}\right)=\alpha}c_\chi^{\psi'}.$$ We examine the coefficients of $c(a,b)$ on both sides of the equation.
	
	On the left-hand side, we have
	
	\begin{align*}
		\sum_{a,b}\left(\sum_{\alpha\in E(a,b),\alpha\vDash m+1 \atop \alpha_{a+b}=1}(a+b)+\sum_{\alpha\in E(a,b),\alpha\vDash m+1 \atop \alpha_{a+b}>1}2(a+b)+\sum_{\alpha\in E(a,b) \atop \alpha\vDash m-1}(m-1)\right)c(a,b)
	\end{align*}
	and on the right-hand side, we have
	$$\sum_{a,b}\left(\sum_{\alpha\in E(a,b) \atop \alpha\vDash m+1}(m+1)\right)c(a,b).$$
	
	Recall that the number of compositions of $n$ into $a$ odd parts and $b$ even parts is $\displaystyle\binom{\frac{n+a-2}{2}}{a+b-1}\binom{a+b}{a}$.
	
	\begin{align*}
		&\sum_{\alpha\in E(a,b),\alpha\vDash m+1 \atop \alpha_{a+b}=1}(a+b)+\sum_{\alpha\in E(a,b),\alpha\vDash m+1 \atop \alpha_{a+b}>1}2(a+b)+\sum_{\alpha\in E(a,b) \atop \alpha\vDash m-1}(m-1)\\
		=&\sum_{\alpha\in E(a,b)\atop \alpha\vDash m+1}2(a+b)-\sum_{\alpha\in E(a,b),\alpha\vDash m+1 \atop \alpha_{a+b}=1}(a+b)+\sum_{\alpha\in E(a,b) \atop \alpha\vDash m-1}(m-1)\\
		=&\binom{\frac{m+a-1}{2}}{a+b-1}\binom{a+b}{a}2(a+b)-\binom{\frac{m+a-3}{2}}{a+b-2}\binom{a+b-1}{a-1}(a+b)+\binom{\frac{m+a-3}{2}}{a+b-1}\binom{a+b}{a}(m-1)\\
		=&\binom{\frac{m+a-1}{2}}{a+b-1}\binom{a+b}{a}\left(2(a+b)-\frac{2a(a+b-1)}{m+a-1}+\frac{(m-a-2b+1)(m-1)}{m+a-1}\right)\\
		=&\binom{\frac{m+a-1}{2}}{a+b-1}\binom{a+b}{a}(m+1)=\sum_{\alpha\in E(a,b) \atop \alpha\vDash m+1}(m+1).
	\end{align*}
\end{proof}

From now on for any odd set composition $\phi\vDash [n]$, let $p(\phi)=(n-\ell(\phi))/2$. 

Let $\phi$ be an odd set composition and $\psi\leq\phi$. For each odd set composition $\varphi$ such that $\varphi\leq\xi$ for some $\xi\in\mathscr{C}_\phi^\psi$, we define the coefficient $C_\varphi^{\phi\psi}$ recursively as follows:
$$\sum_{\varphi\leq\zeta\leq\xi\text{ for some }\xi\in\mathscr{C}_\phi^\psi \atop \zeta\text{ is odd}}(-1)^{p(\zeta)}C_\zeta^{\phi\psi}=(-1)^{p(\phi)}2^{\ell(\psi\land\varphi)-\ell(\varphi)}.$$ 
By inclusion-exclusion, we can see that
$$C_\varphi^{\phi\psi}=(-1)^{p(\varphi)+p(\phi)}\sum_{\varphi\leq\zeta\leq\xi\text{ for some }\xi\in\mathscr{C}_\phi^\psi \atop \zeta\text{ is odd}}\mu(\varphi,\zeta)2^{\ell(\psi\land\zeta)-\ell(\zeta)}.$$
In particular, $\varphi\in\mathscr{C}_\phi^\psi$, implies $C_\varphi^{\phi\psi}=1$.

\begin{Corollary}\label{lem:int2}
	Let $\phi$ be an odd set composition and $\psi\leq\phi$. Then for any set composition $\varphi$, not necessarily odd, such that $\varphi\leq\xi$ for some $\xi\in\mathscr{C}_\phi^\psi$, we have
	$$\sum_{\varphi\leq\zeta\leq\xi\text{ for some }\xi\in\mathscr{C}_\phi^\psi \atop \zeta\text{ is odd}}(-1)^{p(\zeta)}C_\zeta^{\phi\psi}=(-1)^{p(\phi)}2^{\ell(\psi\land\varphi)-\ell(\varphi)}.$$
\end{Corollary}

\begin{proof}
	For simplicity, we write $\mathscr{C}=\mathscr{C}_\phi^{\psi}$ and $C_\zeta=C_\zeta^{\phi\psi}$. Assume $\varphi\leq \xi$ for some $\xi\in\mathscr{C}$. Then,
	\begin{align*}
		\sum_{\varphi\leq\zeta\leq\xi\text{ for some }\xi\in\mathscr{C} \atop \zeta\text{ is odd }}(-1)^{p(\zeta)}C_\zeta&=\sum_{\varphi\leq\zeta\leq\xi\text{ for some }\xi\in\mathscr{C} \atop \zeta\text{ is odd}}(-1)^{p(\phi)}\left(\sum_{\zeta\leq\chi\leq\xi\text{ for some }\xi\in\mathscr{C} \atop \chi\text{ is odd}}\mu(\zeta,\chi)2^{\ell(\psi\land\chi)-\ell(\chi)}\right)\\
		&=(-1)^{p(\phi)}\sum_{\varphi\leq\chi\leq\xi\text{ for some }\xi\in\mathscr{C} \atop \chi\text{ is odd}}2^{\ell(\psi\land\chi)-\ell(\chi)}\left(\sum_{\varphi\leq\zeta\leq\chi \atop \zeta\text{ is odd}}\mu(\zeta,\chi)\right).
	\end{align*}
	This can be computed on each block of $\varphi$ separately. Let $\varphi=\varphi_1|\varphi_2|\cdots|\varphi_{\ell(\varphi)}\vDash[n]$ and $\psi=\psi_1|\psi_2|\cdots|\psi_k\vDash [n]$. When $\varphi\leq\chi\leq\xi$ for some $\xi\in\mathscr{C}$ and $\chi$ is odd set composition, then $\chi=\chi_1|\cdots|\chi_{\ell(\chi)}$ such that for all $j$, $\varphi_j=\chi_{i_j+1}\cup \chi_{i_j+2}\cup\cdots\cup \chi_{i_{j+1}}$ for some $i_j,i_{j+1}$.  We have the following.
		\begin{itemize}
		\item Let $a_i=\ell(\psi|_{\chi_i})$. Then, $$\displaystyle\ell(\psi\land\chi)=\sum_{i=1}^{\ell(\chi)} a_i=\sum_{j=1}^{\ell(\varphi)}\sum_{i=i_j+1}^{i_{j+1}}a_i.$$
		\item We have $$\displaystyle\sum_{\varphi\leq\zeta\leq\chi \atop \zeta\text{ is odd}}\mu(\zeta,\chi)=\prod_{j=1}^{\ell(\varphi)}\sum_{\beta\vDash i_{j+1}-i_j \atop \beta\text{ is odd composition}}\mu(\beta,\hat{1})$$ where $\hat{1}=(1,1,\dots,1)$ is the maximal composition.
	\end{itemize}
	
	Now, let $\phi_j=\mathrm{std}(\phi|_{\varphi_j})$ and $\psi_j=\mathrm{std}(\psi|_{\varphi_j})$ where $\mathrm{std}(-)$ gives the unique set composition that preserves the relative order of elements in the blocks. We have
	
	\begin{align*}
		\sum_{\varphi\leq\chi\leq\xi\text{ for some }\xi\in\mathscr{C} \atop \chi\text{ is odd}}2^{\ell(\psi\land\chi)-\ell(\chi)}\left(\sum_{\varphi\leq\zeta\leq\chi \atop \zeta\text{ is odd}}\mu(\zeta,\chi)\right)
		=&\prod_{j=1}^{\ell(\varphi)}\left(\sum_{\chi\leq\xi\text{ for some }\xi\in\mathscr{C}_{\phi_j}^{\psi_j} \atop \chi\text{ is odd}}2^{\ell(\psi_j\land\chi)-\ell(\chi)}\left(\sum_{\zeta\leq\chi \atop \zeta\text{ is odd}}\mu(\zeta,\chi)\right)\right)\\
		=&\prod_{j=1}^{\ell(\varphi)} 2^{\ell(\psi_j)-1}=2^{\ell(\psi\land\varphi)-\ell(\varphi)}.
	\end{align*}
	The second equality follows from Lemma \ref{lem:main}.
\end{proof}

In our computation of the special case $\phi=\psi$, we find the following identity that relates Catalan numbers with Eulerian numbers that we cannot prove. This problem is irrelevant to the main result of this section.

\begin{Conjecture}\label{conj:euler}
	If $n$ is odd, then
	$$\sum_{\phi\vDash[n] \atop \phi\text{ is odd}}2^{n-\ell(\phi)}\mu_{\ell(\phi)+1}=\sum_{0\leq k\leq n-1}(-1)^kE(n,k)$$
	where $E(n,k)$ are the Eulerian numbers i.e. number of permutations of $n$ with $k$ descents. (OEIS A000182 with alternating sign)
\end{Conjecture}

\begin{Corollary}\label{cor:main}
	Let $\phi$ be an odd set composition and $\psi\leq\phi$. Let $\varphi$ be an odd set composition  such that $\varphi\leq\xi$ for some $\xi\in\mathscr{C}_\phi^\psi$. Let $\varphi=\varphi_1|\varphi_2|\cdots|\varphi_{\ell(\varphi)}\vDash[n]$. Let $\phi_j=\mathrm{std}(\phi|_{\varphi_j})$ and $\psi_j=\mathrm{std}(\psi|_{\varphi_j})$ where $\mathrm{std}(-)$ gives the unique set composition that preserves the relative order of elements in the block. We have
	$$C_\varphi^{\phi\psi}=(-1)^{p(\varphi)+p(\phi)}\prod_{j=1}^{\ell(\varphi)}\left(\sum_{\zeta\leq\xi\text{ for some }\xi\in\mathscr{C}_{\phi_j}^{\psi_j} \atop \zeta\text{ is odd}}2^{\ell(\psi_j\land\zeta)-\ell(\zeta)}\mu_{\ell(\zeta)+1}\right).$$
\end{Corollary}

\begin{proof}
	Recall that we have
	$$C_\varphi^{\phi\psi}=(-1)^{p(\varphi)+p(\phi)}\sum_{\varphi\leq\zeta\leq\xi\text{ for some }\xi\in\mathscr{C}_\phi^\psi \atop \zeta\text{ is odd}}\mu(\varphi,\zeta)2^{\ell(\psi\land\zeta)-\ell(\zeta)}.$$
	This can be compute on each block of $\varphi$ separately and we have
	\begin{align*}
		\sum_{\varphi\leq\zeta\leq\xi\text{ for some }\xi\in\mathscr{C}_\phi^\psi \atop \zeta\text{ is odd}}\mu(\varphi,\zeta)2^{\ell(\psi\land\zeta)-\ell(\zeta)}=\prod_{j=1}^{\ell(\varphi)}\left(\sum_{\zeta\leq\xi\text{ for some }\zeta\in\mathscr{C}_{\phi_j}^{\psi_j} \atop \zeta\text{ is odd}}\mu(\hat{0},\zeta)2^{\ell(\psi_i\land\zeta)-\ell(\zeta)}\right).
	\end{align*}
	Finally, $\mu(\hat{0},\zeta)=\mu_{\ell(\zeta)+1}=(-1)^{(\ell(\zeta)-1)/2}\Cat_{(\ell(\zeta)-1)/2}$ yields the result.
\end{proof}

\begin{Conjecture}
	The coefficients $C_\varphi^{\phi\psi}$ are non-negative. 
\end{Conjecture}

\begin{Remark}
	From Corollary \ref{cor:main}, to understand the coefficients $C_\varphi^{\phi\psi}$, it suffices to understand the case where $\phi=1|2|\cdots|n$. In this case, $\psi=1\dots \alpha_1|(\alpha_1+1)\dots \alpha_2|\cdots|(\alpha_{\ell(\psi)-1}+1)\cdots n$. Therefore, we can write the coefficients as $C^\alpha_\varphi$ where $\alpha\models n$ is a composition.
	
	If every block of $\varphi$ is contained in some block of $\psi$, then $\ell(\psi\land\zeta)=\ell(\zeta)$ for all $\zeta\geq\varphi$. In this case, it can be deduced that $C_\varphi^{\phi\psi}=0$ unless $\varphi\in\mathscr{C}_\phi^\psi$. If Problem \ref{conj:euler} is true, then $\psi=\phi$ implies $C_\varphi^{\phi\psi}$ is non-negative. A cancellation-free formula for $C_\varphi^{\phi\psi}$ remains open.
\end{Remark}

\begin{Conjecture}\label{conj:euler2}
	By taking the special case of $\psi=\phi$ in Lemma \ref{lem:main}, we have that
	if the Conjecture \ref{conj:euler} is true and if $n$ is even, then
	$$\sum_{k=1}^{n/2}\binom{n}{2k-1}2^{n-2k}\left(\sum_{m=0}^{2k-1}(-1)^mE(2k-1,m)\right)=2^{n-1}.$$
\end{Conjecture}

We would like to remark that Conjecture \ref{conj:euler} has recently been confirmed by Zhao, Lin and Zang \cite{ZLZ} using the method of generating functions. In addition, they prove a $q$-analogue of Conjecture \ref{conj:euler2}. They also provide two interesting and unexpected applications in enumerative combinatorics. The first one is that Conjecture \ref{conj:euler2} leads to a simple proof of that the Genocchi numbers are odd. This non-trivial fact was traditionally proved by several sophisticated methods, and no simple proof was known. As the second application, an alternative proof of Foata's divisibility property of $q$-tangent numbers is obtained as a consequence of the $q$-analogue of Conjecture \ref{conj:euler2}.

\section{The peak algebra as a left co-ideal of $\NCQSym$ under the internal coproduct}\label{Sec:com}
The goal of this section is to prove that $\NCPeak$ is a left co-ideal of $\NCQSym$ under the internal coproduct. 
In NCQSym, one can define (cf. \cite{BHRZ06}) an internal coproduct $\Delta^\circ:\NCQSym\to\NCQSym\otimes\NCQSym$ via the substitution $f(\mathbf{x})\mapsto f(\mathbf{xy})\mapsto f(\mathbf{x})\otimes f(\mathbf{y})$ with $\{\mathbf{x}_i\mathbf{y}_j\}$ ordered lexicographically. Together with the regular product, $m$, of power series, $(\NCQSym,m,\Delta^\circ)$ forms an ungraded bialgebra.

The internal coproduct is given by
$$\Delta^\circ(\mathbf{M}_\phi)=\sum_{\psi\land\varphi=\phi}\mathbf{M}_\psi\otimes \mathbf{M}_\varphi.$$

For example, $\Delta^\circ(\mathbf{M}_{1|2})=\mathbf{M}_{12}\otimes\mathbf{M}_{1|2}+\mathbf{M}_{1|2}\otimes(\mathbf{M}_{12}+\mathbf{M}_{1|2}+\mathbf{M}_{2|1})$. Note that $\varphi\land\psi\neq\psi\land\varphi$ in general, so $\Delta^\circ$ is not cocommutative. Since $\psi\land\varphi=\phi$ implies $\psi\leq\phi$, we will use the following refined formula
$$\Delta^\circ(\mathbf{M}_\phi)=\sum_{\psi\leq\phi \atop \psi\land\varphi=\phi}\mathbf{M}_\psi\otimes\mathbf{M}_\varphi.$$

We are ready to prove the main theorem of this section.

\begin{Theorem}\label{thm:internal} Let $\phi$ be an odd set composition. We have that
	$$\Delta^\circ(\boldeta_\phi)=\sum_{\psi\leq\phi}\mathbf{M}_{\psi}\otimes\left(\sum_{\varphi\leq\xi\text{ for some }\xi\in\mathscr{C}_\phi^\psi \atop \varphi\text{ is odd }}C_\varphi^{\phi\psi}\boldeta_\varphi\right)$$
	In particular, the set $\NCPeak$ forms a left co-ideal in $\NCQSym$ under the internal coproduct i.e. $\Delta^{\circ}(\NCPeak)\subseteq\NCQSym\otimes\NCPeak$.
\end{Theorem}

\begin{proof}
	Fix an odd set composition $\phi$, Recall that $\boldeta_\phi=(-1)^{p(\phi)}\sum_{\chi\leq\phi}2^{\ell(\chi)}\mathbf{M}_\chi$. We have that 
	\begin{align*}
		\Delta^\circ(\boldeta_\phi)&=(-1)^{p(\phi)}\sum_{\chi\leq\phi}2^{\ell(\chi)}\Delta^\circ(\mathbf{M}_\chi)\\
		&=(-1)^{p(\phi)}\sum_{\chi\leq\phi}2^{\ell(\chi)}\left(\sum_{\psi\leq\chi \atop \psi\land\epsilon=\chi}\mathbf{M}_\psi\otimes\mathbf{M}_\epsilon\right)\\
		&=\sum_{\psi\leq\phi}\mathbf{M}_{\psi}\otimes\left((-1)^{p(\phi)}\sum_{\psi\leq\chi\leq\phi \atop \psi\land\epsilon=\chi}2^{\ell(\chi)}\mathbf{M}_\epsilon\right)
	\end{align*}
	Now we fix $\psi\leq\phi$ and for simplicity, we write $\mathscr{C}=\mathscr{C}_\phi^\psi$ and $C_{\varphi}=C_{\varphi}^{\phi\psi}$. By Lemma \ref{lem:int1} and Corollary \ref{lem:int2},
	\begin{align*}
		(-1)^{p(\phi)}\sum_{\psi\leq\chi\leq\phi \atop \psi\land\varphi=\chi}2^{\ell(\chi)}\mathbf{M}_\epsilon&=
		(-1)^{p(\phi)}\sum_{\epsilon\leq\xi ~\text{for some }\xi\in\mathscr{C}}2^{\ell(\psi\land\epsilon)}\mathbf{M}_\epsilon\\
		&=\sum_{\epsilon\leq\xi ~ \text{for some }\xi\in\mathscr{C}}\left(2^{\ell(\epsilon)}\sum_{\varphi\leq{\varphi}\leq\xi\text{ for some }\xi\in\mathscr{C} \atop {\varphi}\text{ is odd}}(-1)^{p({\varphi})}C_{\varphi}\right)\mathbf{M}_\epsilon\\
		&=\sum_{{\varphi}\leq\xi\text{ for some }\xi\in\mathscr{C} \atop {\varphi}\text{ is odd}}(-1)^{p(\varphi)}C_{\varphi}\left(\sum_{\epsilon\leq{\varphi}}2^{\ell(\epsilon)}\mathbf{M}_\epsilon\right)\\
		&=\sum_{{\varphi}\leq\xi\text{ for some }\xi\in\mathscr{C} \atop {\varphi}\text{ is odd}}C_{\varphi}\boldeta_{\varphi}.
	\end{align*}
\end{proof}

\begin{Example}
	Suppose $\phi=3|1|4|2$ and $\psi=3|124$, then $\mathscr{C}_\phi^\psi=\{3|1|4|2,1|3|4|2,1|4|3|2,1|4|2|3\}$. For $\zeta=134|2$, since $3|124\land 134|2=3|14|2$ and $|\{\xi\in\mathscr{C}_\phi^\psi:\zeta\leq\xi\}|=3$, we have $3-C_\zeta^{\phi\psi}=2^{\ell(3|14|2)-\ell(3|124)}$ i.e. $C_\zeta^{\phi\psi}=1$. Similarly, $C_{124|3}^{\phi\psi}=0$, $C_{3|124}^{\phi\psi}=0$ and $C_{1|234}^{\phi\psi}=1$. Hence,
	$$\Delta^\circ(\boldeta_{3|1|4|2})=\mathbf{M}_{3|124}\otimes(\boldeta_{3|1|4|2}+\boldeta_{1|3|4|2}+\boldeta_{1|4|3|2}+\boldeta_{1|4|2|3}+\boldeta_{134|2}+\boldeta_{1|234})+\cdots.$$
\end{Example}

\begin{Remark}
	The set $\NCPeak$ fails to be a right co-ideal in $\NCQSym$. For example,
	$$\Delta^\circ(\mathbf{K}_{(\emptyset,12)})=\mathbf{K}_{(\emptyset,12)}\otimes\mathbf{M}_{12}+4(\mathbf{M}_{12}+\mathbf{M}_{1|2})\otimes\mathbf{M}_{1|2}+4\mathbf{M}_{1|2}\otimes\mathbf{M}_{2|1}$$
	while $\mathbf{M}_{1|2}\notin\NCPeak$.
\end{Remark}

\section{Representation-theoretic interpretations of peak algebras and theta maps}\label{Sec:rep}
In this section, we find representation-theoretic interpretations of the theta maps for the Hopf algebra of symmetric functions and noncommutative symmetric functions. 
\subsection{Symmetric functions}

Let ${\rm Irr}(G)$ be the set of irreducible characters of the finite group $G$ and ${\rm cf}(G)$ denote the space of the class functions of $G$.

The space $${\rm cf}(\mathfrak{S}_\bullet)=\bigoplus_{n\geq 0}{\rm cf}(\mathfrak{S}_n)$$  is a Hopf algebra with the product and coproduct
$$f . g = {\rm Ind}(f.g) = {\rm Ind}_{\mathfrak{S}_n\times \mathfrak{S}_m}^{\mathfrak{S}_{n+m}} (f\otimes g)$$ and $${\rm Res}(f) = \bigoplus_{i+j=n}{\rm Res}_{\mathfrak{S}_i\otimes \mathfrak{S}_j}^{\frakS_n} (f),$$ respectively, where $f\in\frakS_n$ and $g\in \frakS_m$. This Hopf algebra is isomorphic to the Hopf algebra of symmetric functions $\Sym$ using the Frebenius characteristic map 
$${\rm ch}_{\Sym}: \cf(\frakS_\bullet) \rightarrow \Sym$$ give by 
$$z_\lambda \delta_\lambda \mapsto p_\lambda$$ where $z_\lambda = 1^{m_1}m_1! 2^{m_2} m_2! 3^{m_3} m_3! \dots n^{m_n}m_n!$ such that $\lambda$ has $m_i$ parts equal to $i$.

For $f,g\in \cf(G)$, define $$f\odot g$$ give by $$(f \cdot g)(a)=f(a)g(a)$$ for all $a\in G$. Also, for any $f \in \cf(G)$, define $$(f\odot -):\cf(G)\rightarrow \cf(G)$$ given by 
$$(f\odot -)(g) = (f\odot g).$$ In the following theorem, we show that the following functor describes $\Theta_\Sym$,
$$\Theta_{\cf(\frakS_\bullet)}|_{\cf(\frakS_n)}=  {\rm Ind} \circ ((\sgn\odot -) \otimes (\mathbbm{1}\odot -)) \circ {\rm Res}$$ where $\bbmone$ is the trivial character and $\sgn$ is the sign character of $\frakS_n$.

\begin{Theorem}
The following diagram 
\begin{center}
\begin{tikzpicture}
	\node(NCSym) at (4,0){$\mathsf{\cf(\frakS_\bullet)}$};
	\node(NCQSym) at (8,0){$\Sym$};
	\node(NCGamma) at (4,-4){${\rm Img}(\Theta_{\cf(\frakS_\bullet)})$};
	\node(NCPeak) at (8,-4){$\Omega$};
	
	\draw[thick,->]  (NCSym)->(NCQSym);
	\draw[thick,->]  (NCGamma)->(NCPeak);
	\draw[thick,->]   (NCSym)->(NCGamma);
	\draw[thick,->]   (NCQSym)->(NCPeak);
	
	\node(iota1) at (6,0.25){${\rm ch}_\Sym$};
	\node(iota2) at (6,-3.75){${\rm ch}_\Sym$};
	\node(ThetaNCSym) at (4-0.7,-2){$\Theta_{\cf(\frakS_\bullet)}$};
	\node(ThetaNCQSym) at (8.5,-2){$\Theta_\Sym$};
	

	
\end{tikzpicture} 
\end{center}
is commutative.
\end{Theorem}

\begin{proof}
Let $\delta_\lambda$ be the class function that is $1$ over all permutations with cycle type $\lambda$ and $0$ otherwise.   For $I\subseteq [\ell(\lambda)]$, let $\lambda_I$ be the partition whose parts are $\lambda_i, i\in I$. 
We have that 
$${\rm Res} (z_\lambda \delta_{\lambda})=\sum_{I \subseteq [\ell(\lambda)]}  z_{\lambda_I} \delta_{\lambda_I}\otimes z_{\lambda_{[\ell(\lambda)] \setminus I}} \delta_{[\ell(\lambda)] \setminus I}.$$
Then $$(\sgn\odot -)\otimes (\bbmone\odot -) \left(\sum_{I \subseteq [\ell(\lambda)]}  z_{\lambda_I} \delta_{\lambda_I}\otimes z_{\lambda_{[\ell(\lambda)] \setminus I}} \delta_{[\ell(\lambda)] \setminus I}\right) = $$
$$\sum_{I \subseteq [\ell(\lambda)]}  z_{\lambda_I} (\sgn\odot \delta_{\lambda_I})\otimes z_{\lambda_{[\ell(\lambda)] \setminus I}} (\bbmone\odot \delta_{[\ell(\lambda)] \setminus I}) =
$$
$$\sum_{I \subseteq [\ell(\lambda)]}  z_{\lambda_I} ((-1)^{\prod_{i\in I} \lambda_i} \delta_{\lambda_I})\otimes z_{\lambda_{[\ell(\lambda)] \setminus I}} (\delta_{[\ell(\lambda)] \setminus I}) .
$$
Moreover, $${\rm Ind}\left(\sum_{I \subseteq [\ell(\lambda)]}  z_{\lambda_I} ((-1)^{\prod_{i\in I} \lambda_i} \delta_{\lambda_I})\otimes z_{\lambda_{[\ell(\lambda)] \setminus I}} (\delta_{[\ell(\lambda)] \setminus I})\right) = \begin{cases} 2^{\ell(\lambda)}z_\lambda \delta_{\lambda} & \text{if $\lambda$ is odd,} \\ 0, & \text{otherwise.} \end{cases}$$
Therefore, 
$$\Theta_{\cf(\frakS_\bullet)}(\delta_{\lambda})=\begin{cases} 2 \delta_{\lambda} & \text{if $\lambda$ is odd,} \\ 0 & \text{otherwise.} \end{cases}$$
It is now straightforward to check that the diagram is commutative. 
\end{proof}

In Aguiar-Bergeron-Sottile's character theory of Hopf algebras \cite{A06}, we have that the theta map of $\Sym$ is the unique Hopf algebra morphism corresponding to the character $\overline{\zeta^{-1}}\zeta$ which takes $m_{\lambda}$ to $1$ when $\lambda=()$ or $(n)$ and $0$ otherwise. It is interesting to see that this character of the Hopf algebra $\Sym$ is equivalent to the character $(\sgn \odot \bbmone) \circ {\rm Res}$. So both the sign character $\sgn$ and the trivial character $\bbmone$ play an important role in defining the theta map and the character of $\Sym$ corresponding to the theta map. 

\subsection{Noncommutative symmetric functions} In this section, we first use the framework in \cite{AT} to find a representation-theoretic interpretation of the Hopf algebra of Noncommutative symmetric functions $\NSym$, and then we describe the representation-theoretic interpretation of $\Theta_\NSym$, which is identical to the dual of $\Theta_\QSym$.

Let $$\frakS^n = \frakS_2\times \frakS_2\times \cdots \times \frakS_2\times \frakS_1.$$ The group $\frakS_1$ has one ireducible character that we denote it by $1$. Consider that for each compostion $\alpha=(\alpha_1,\alpha_2, \dots, \alpha_{\ell})\vDash n$, 
$$(\sgn)_\alpha^{\bbmone}=\sgn^{\otimes(\alpha_1-1)} \otimes \bbmone \otimes \sgn^{\otimes(\alpha_2-1)} \otimes \bbmone \otimes  \cdots \otimes \bbmone \otimes \sgn^{\otimes(\alpha_{\ell}-1)} \otimes 1$$ is an irreducible character of $\frakS^\bullet$, and we have that 
$${\rm Irr}(\frakS^n)=\{(\sgn)_\alpha^{\bbmone}: \alpha\vDash [n]\}.$$

Then 
$$\cf(\frakS^\bullet)=\bigoplus_{n\geq 0} \cf(\frakS^n)$$ is a Hopf algebra by the representation theoretic functors ${\rm Inf}$ and ${\rm Dn}$ in \cite[(3.1) and (3.2)]{AT} that we describe them combinatorially. Give class functions $f=f_1\otimes f_2\otimes \cdots \otimes f_{n-1} \otimes 1\in \cf(\frakS^n)$ and $g=g_1\otimes g_2\otimes \cdots \otimes g_{m-1} \otimes 1\in \cf(\frakS^m)$, define
$${\rm Inf}(f\otimes g)=f_1\otimes f_2\otimes \cdots \otimes f_{n-1} \otimes \bbmone \otimes g_1\otimes g_2\otimes \cdots \otimes g_{m-1} \otimes 1\in \cf(\frakS^{n+m}),$$
where $\reg$ is the regular character of $\frakS_2$. Let $I=\subseteq [n]$ and $\overline{I}= [n]\setminus I$. Define for each $i\in I$, 
$$\left({\rm Dn}_{(I,\overline{I})}^{[n]}f\right)_i=\begin{cases}
\langle f_i, \bbmone \rangle 1 & \text{if $i=\max(I)$},\\
f_i & \text{if $i,i+1\in I$},\\
\langle f_i, \bbmone \rangle \sgn & \text{otherwise;} 
\end{cases}$$
and for each $i\in \overline{I}=[n]\setminus I$, define
$$\left({\rm Dn}_{(I,\overline{I})}^{[n]}f\right)_i=\begin{cases}
1 & \text{if $i=\max(\overline{I})$},\\
f_i & \text{if $i,i+1\in \overline{I}$},\\
\langle f_i, \bbmone \rangle \sgn & \text{otherwise.} 
\end{cases}$$
Then define 
$${\rm Dn}(f)=\left(\bigotimes_{i\in I} \left({\rm Dn}_{(I,\overline{I})}^{[n]} f\right)_i \right) \otimes \left(\bigotimes_{i\in \overline{I}} \left({\rm Dn}_{(I,\overline{I})}^{[n]} f\right)_i\right)\in \cf(\frakS^{|I|})\otimes \cf(\frakS^{n-|I|}).$$ 
For example, if $I=\{2,3,6,7,8\}\subseteq [10]$, then 
$${\rm Dn}_{(I,\overline{I})}^{[n]} f = \left(f_2\otimes \langle f_3, \bbmone\rangle \sgn \otimes f_6 \otimes f_7 \otimes \langle f_8, \bbmone\rangle 1\right) \otimes \left(\langle f_1, \bbmone \rangle \sgn \otimes f_4\otimes \langle f_5, \bbmone \rangle \sgn \otimes f_9\otimes 1 \right).$$

By \cite[Lemma 4.4]{AT}, we have that 
$${\rm ch}: \cf(\frakS^\bullet) \rightarrow \NSym$$ is a Hopf algebra morphism given by 
$$(\sgn)_{\alpha}^{\bbmone} \mapsto \mathbf{p}_{\alpha},$$ where 
$$(\sgn)_\alpha^{\bbmone}=\sgn^{\otimes(\alpha_1-1)} \otimes \bbmone \otimes \sgn^{\otimes(\alpha_2-1)} \otimes \bbmone \otimes  \cdots \otimes \bbmone \otimes \sgn^{\otimes(\alpha_{\ell}-1)} \otimes 1.$$ 

With respect to the Inner product of characters, since for each composition $\alpha$, $(\sgn)_\alpha^{\bbmone}$ is an irreducible character, we have $$\langle (\sgn)_\alpha^{\bbmone}, (\sgn)_\beta^{\bbmone}  \rangle=\begin{cases}
1 & \text{if $\alpha=\beta$}\\
0 & \text{otherwise.}
\end{cases}$$ 
Therefore, $\cf(\frakS^\bullet)$ with the dual of the functors {\rm Inf} and ${\rm Dn}$ is a Hopf algebra isomorphic to $\QSym$ where $(\sgn)_\alpha^{\bbmone}$ mapps to a shuffle basis of $\QSym$. 

In the following theorem, we show that for each $\alpha\vDash n$, the representation-theoretic interpretation of $\Theta_\NSym$ is defined by 
$$\Theta_{\cf(\frakS^\bullet)}((\sgn)_\alpha^{\bbmone})=\begin{cases}
2 (\sgn)_\alpha^{\bbmone} & \text{if $\alpha$ is odd,}\\
0 & \text{otherwise.}
\end{cases}$$
This suggests that there may be a functor in representation theory that gives us the theta map of $\cf(\frakS^\bullet)$ as the characters map to characters. 

\begin{Theorem}
The following diagram 
\begin{center}
\begin{tikzpicture}
	\node(NCSym) at (4,0){$\mathsf{\cf(\frakS^\bullet)}$};
	\node(NCQSym) at (8,0){$\NSym$};
	\node(NCGamma) at (4,-4){${\rm Img}(\Theta_{\cf(\frakS^\bullet)})$};
	\node(NCPeak) at (8,-4){$\Pi^*$};
	
	\draw[thick,->]  (NCSym)->(NCQSym);
	\draw[thick,->]  (NCGamma)->(NCPeak);
	\draw[thick,->]   (NCSym)->(NCGamma);
	\draw[thick,->]   (NCQSym)->(NCPeak);
	
	\node(iota1) at (6,0.25){${\rm ch}_{\NSym}$};
	\node(iota2) at (6,-3.75){${\rm ch}_{\NSym}$};
	\node(ThetaNCSym) at (4-0.7,-2){$\Theta_{\cf(\frakS^\bullet)}$};
	\node(ThetaNCQSym) at (8.6,-2){$\Theta_\NSym$};
	

	
\end{tikzpicture} 
\end{center}
is commutative.
\end{Theorem}

\begin{remark} A family of supercharacter theories for the group of unipotent upper-triangular matrices, with inflation and restriction functors, provides a representation-theoretic interpretation of $\NCSym$ \cite{ABT}. It would be interesting to find a functor in representation theory that is equivalent to the theta map of $\NCSym$. 
\end{remark}

\section{Acknowledgement}
The authors would like to thank Nantel Bergeron for introducing the problem and to thank Zhicong Lin and Stephanie van Willigenburg for helpful conversations.

\begin{thebibliography}{10}

\bibitem{ABT} M. Aguiar, C. Andé, C. Benedetti, N. Bergeron, Z. Chen, P. Diaconis, A. Hendrickson, S. Hsiao, M. Isaacs, A. Jedwab, K. Johnson, G. Karaali, A. Lauve, T. Le, S. Lewis, H. Li, K. Magaard, E. Marberg, J-C. Novelli, A. Pang, F. Saliola, L. Tevlin, J-Y. Thibon, N. Thiem, V. Venkateswaran, C. R. Vinroot, N. Yan, M. Zabrocki, 
\newblock \textit{Supercharacters, symmetric functions in noncommuting variables, and related Hopf algebras},
\newblock Advances in Mathematics, 229: 2310--2337, 2012.

\bibitem{A06} M. Aguiar, N. Bergeron, F. Sottile,
\newblock \textit{Combinatorial Hopf algebras and generalized Dehn-Sommerville relations},
\newblock Compositio Mathematica, 142: 1--30, 2006. 

\bibitem{A04} M. Aguiar, N. Bergeron, K. Nyman,
\newblock \textit{The peak algebra and the descent algebras of types B and D},
\newblock Transactions of the American Mathematical Society, 356: 2781--2824, 2004.

\bibitem{ALi} F. Aliniaeifard, S. X. Li, Peak algebras in combinatorial Hopf algebras, 
\href{https://arxiv.org/abs/2110.05648v1}{arXiv:2110.05648v1}, 2021. 

\bibitem{ALvW23} F. Aliniaeifard, S. X. Li, S. van Willigenburg, 
\newblock \textit{Generalized chromatic functions},
\newblock International Mathematics Research Notices, 5: 4456--4500, 2024.

\bibitem{ALvW} F. Aliniaeifard, S. X. Li, S. van Willigenburg, 
\newblock \textit{Schur functions in noncommuting variables},
\newblock Advances in Mathematics, 406:108536, 2022.

\bibitem{AT} F. Aliniaeifard, N. Thiem, 
\newblock \textit{Hopf structures in the representation theory of direct products}, 
\newblock The Electronic Journal of Combinatorics, 29: 4, 2022.

\bibitem{BB85} M. Bayer, L. Billera,
\newblock\textit{Generalized Dehn-Sommerville relations for polytopes, spheres and Eulerian partially ordered sets},
\newblock Inventiones Mathematicae, 79: 143--157, 1985.

\bibitem{BB81} A. Beilinson, J. Bernstein, 
\newblock\textit{Localisation de g-modules},
\newblock Comptes Rendus. Mathématique. Académie des Sciences, Paris,  292: 15--18, 1981.


\bibitem{B06} N. Bergeron, C. Hohlweg, 
\newblock \textit{Coloured peak algebras and Hopf algebras},
\newblock Journal of Algebraic Combinatorics, 24: 299--330, 2006.

\bibitem{BHRZ06} N. Bergeron, C. Hohlweg, M. Rosas, M. Zabrocki,
\newblock \textit{Grothendieck bialgebras, partition lattices, and symmetric functions in non-commutative variables},
\newblock The Electronic Journal of Combinatorics, 13: 1, 2006.

\bibitem{BMSvW} N. Bergeron, S. Mykytiuk, F. Sottile, S. van Willigenburg,
\newblock \textit{Shifted quasisymmetric functions and the Hopf algebra of peak functions},
\newblock Discrete Mathematics, 246: 57--66, 2002.

\bibitem{BMSvW0} N. Bergeron, S. Mykytiuk, F. Sottile, S. van Willigenburg,
 \newblock \textit{Noncommutative Pieri operators on posets. In memory of Gian-Carlo Rota}, Journal of Combinatorial Theory, Series A, 91:  84--110, 2000.

\bibitem{B04} N. Bergeron, F. Hivert, J-Y. Thibon,
\newblock \textit{The peak algebra and the Hecke-Clifford algebras at q=0},
\newblock Journal of Combinatorial Theory, Series A, 107: 1--19, 2004.

\bibitem{B09} N. Bergeron, M. Zabrocki, 
\newblock \textit{The Hopf algebras of symmetric functions and quasi-symmetric functions in non-commutative variables are free and cofree},
\newblock Journal of Algebra and its Applications, 8: 581--600, 2009.

\bibitem{BH}  S. Billey, M. Haiman, \textit{Schubert polynomials for the classical groups}, Journal of American Mathematical Society, 8: 443–482, 1995.

\bibitem{B03} L. J. Billera, S. K. Hsiao, S. van Willigenburg, 
\newblock \textit{Peak quasisymmetric functions and Eulerian enumeration},
\newblock Advances in Mathematics, 176: 248--276, 2003.



\bibitem{BB05} A. Björner, F. Brenti, Combinatorics of Coxeter Groups, Graduate Texts in Mathematics, vol. 231, Springer,
New York, 2005.





\bibitem{B14} F. Brenti, F. Caselli,
\newblock \textit{Peak algebras, paths in the Bruhat graph and Kazhdan–Lusztig polynomials}, 
\newblock Advances in Mathematics 304: 539--582, 2017.


\bibitem{BK81} J.-L. Brylinski, M. Kashiwara,
\newblock\textit{ Kazhdan–Lusztig conjecture and holonomic system}, Inventiones Mathematicae, 
64: 387--410,  1981.


\bibitem{C15} S. Choi, H. Park, 
\newblock \textit{A new graph invariant arises in toric topology},
\newblock Journal of the Mathematical Society of Japan, 67: 699--720, 2015.







\bibitem{G95} I. M. Gelfand, D. Krob, A. Lascoux, B. Leclerc, V. S. Retakh, J-Y. Thibon,
\newblock \textit{Noncommutative symmetric functions},
\newblock Advances in Mathematics, 112: 218--348, 1995.

\bibitem{Ges16} I. Gessel, A historical survey of P-partitions, The mathematical legacy of Richard P. Stanley,
169–188, American Mathematical Society, Providence, RI, 2016.

%
%




\bibitem{H07} S. K. Hsiao,
\newblock \textit{Structure of the peak Hopf algebra of quasi-symmetric functions},
\newblock preprint, 2007.

\bibitem{H06} S. K. Hsiao, K. Peterson,
\newblock \textit{The Hopf algebras of type B quasisymmetric functions and peak functions},
\newblock \href{https://arxiv.org/abs/math/0610.976}{arXiv:0610.976}, 2006.



\bibitem{H90} J. E. Humphreys, Reflection Groups and Coxeter Groups, Cambridge Studies in Advanced Mathematics, vol. 29,
Cambridge University Press, Cambridge, 1990.

\bibitem{J15} N. Jing, Y. Li,
\newblock \textit{The shifted Poirier-Reutenauer algebra},
\newblock Mathematische Zeitschrift, 281: 611--629, 2015.

\bibitem{KL79} D. Kazhdan, G. Lusztig, 
\newblock \textit{Representations of Coxeter groups and Hecke algebras}, 
Inventiones Mathematicae 53: 165--184 (1979). 




\bibitem{M95} I.~G.~Macdonald,
\newblock {Symmetric Functions and Hall Polynomials},
\newblock Second Edition, Oxford University Press, 1995.

\bibitem{Mac11} P. MacMahon, Memoir on the theory of the partitions of numbers. Part V. Partitions in two dimensional space, Proceedings of the Royal Society of London. Series A, 85 (1911), no. 578, 304–305,
Reprinted in [40 pp. 1328–1363].

\bibitem{MR95} C. Malvenuto, C. Reutenauer,
\newblock \textit{Duality between quasi-symmetric functions and the Solomon descent algebra},
\newblock Journal of Algebra 177: 967--982, 1995.

\bibitem{Nym01} K. Nyman, Enumeration in Geometric Lattices and the Symmetric Group, Ph. D. thesis,
Cornell University, 2001.

\bibitem{N03} K. Nyman,
\newblock \textit{The peak algebra of the symmetric group},
\newblock Journal of Algebraic Combinatorics, 17: 309--322, 2003.

\bibitem{OEISA006154} OEIS Foundation Inc. (2023), Entry A006154 in The On-Line Encyclopedia of Integer Sequences, \href{http://oeis.org/A005132}{http://oeis.org/A005132}.

\bibitem{P95} S. Poirier and C. Reutenauer,
\newblock \textit{Alg\`{e}bres de Hopf de tableaux},
\newblock Annales Math\'{e}matiques du Qu\'{e}bec, 19: 79--90, 1995.	

\bibitem{R04} M. Rosas, B. Sagan, 
\newblock \textit{Symmetric functions in noncommuting variables},
\newblock Transactions of the American Mathematical Society, 358: 215--232, 2004.

\bibitem{S87} B. Sagan, 
\newblock \textit{Shifted tablauex, Schur $Q$-functions and a conjecture of R. Stanley},
\newblock Journal of Combinatorial Theory, Series A, 45:62--105, 1987.

\bibitem{S05} M. Schocker,
\newblock \textit{The peak algebra of the symmetric group revisited},
\newblock Advances in Mathematics, 192: 259--309, 2005.

\bibitem{S11} I. Schur,
\newblock \textit{\"Uber die Darstellung der symmetrischen und der alternierenden Gruppe durch gebrochene lineare Substitutionen.},
\newblock  Journal f\"ur die Reine und Angewandte Mathematik, 139: 155--250, 1911.

\bibitem{S76} L. Solomon,
\newblock \textit{A Mackey formula in the group ring of a Coxeter group},
\newblock  Journal of Algebra, 41: 255--264, 1976.

\bibitem{Sta71} R. Stanley, Ordered Structures and Partitions, Ph.D. thesis, Harvard University, 1971.


\bibitem{EC} R. Stanley,
\newblock {Enumerative Combinatorics},
\newblock vol. 1, Cambridge University Press, 1997.

\bibitem{S89} J. Stembridge,
\newblock \textit{Shifted tableaux and the projective representations of symmetric groups},
\newblock Advances in Mathematics, 74: 87--134, 1989.

\bibitem{S97} J. Stembridge,
\newblock \textit{Enriched P-partitions},
\newblock Transactions of the American Mathematical Society, 249: 763--788, 1997.


\bibitem{W36} M. Wolf,
\newblock \textit{Symmetric functions of non-commutative elements},
\newblock Duke Mathematical Journal, 2: 626--637, 1936.

\bibitem{W84} D. Worley,
\newblock A theory of shifted Young tableaux,
\newblock Ph.D thesis, Massachusetts Institute of Technology,1984.

\bibitem{ZLZ} T. Zhao, Z. Lin, Y. Zang
\newblock An identity relating Catalan numbers to tangent numbers with arithmetic applications,
\newblock  \href{https://arxiv.org/abs/2507.20965}{arXiv:2507.20965}, 2025.










\end{thebibliography}
\end{document}